\documentclass[a4paper,reqno]{amsart}

\usepackage{amsthm,amsmath,amssymb}
\usepackage[utf8]{inputenc}
\usepackage{mathtools}
\usepackage{comment}
\usepackage{bm}
\usepackage{xcolor} \definecolor{bluegreen2}{RGB}{0, 85, 127}
\usepackage[linktocpage]{hyperref}
\hypersetup{breaklinks, colorlinks=true, linkcolor=bluegreen2, citecolor=bluegreen2, filecolor=magenta, urlcolor=bluegreen2}
\usepackage{tikz} 
\usetikzlibrary{arrows.meta, cd, calc, intersections}
\usepackage{enumitem} \setlist{itemsep=1.5mm}
\usepackage{newunicodechar}
\usepackage[labelformat=simple]{subcaption}
\usepackage{placeins}
\theoremstyle{plain}
\newtheorem{theorem}{Theorem}[section]
\newtheorem{prop}[theorem]{Proposition}
\newtheorem{cor}[theorem]{Corollary}
\newtheorem{lemma}[theorem]{Lemma}
\newtheorem{hol_conj}[theorem]{Holomorphy conjecture}
\newtheorem{mon_conj}[theorem]{Monodromy conjecture}
\newtheorem{strategy}[theorem]{Strategy}
\newtheorem{caso}{Case}
\theoremstyle{definition}
\newtheorem{defn}[theorem]{Definition}
\newtheorem{example}[theorem]{Example}
\newtheorem{examples}[theorem]{Examples}
\newtheorem{notation}[theorem]{Notation}
\theoremstyle{remark}
\newtheorem{remark}[theorem]{Remark}

\newcommand{\RR}{\mathbb{R}}
\newcommand{\CC}{\mathbb{C}}
\newcommand{\NN}{\mathbb{N}}
\newcommand{\ZZ}{\mathbb{Z}}
\newcommand{\QQ}{\mathbb{Q}}
\newcommand{\PP}{\mathbb{P}}
\newcommand{\LL}{\mathbb{L}}
\newcommand{\one}{\mathbf{1}}
\newcommand{\zero}{\mathbf{0}}
\newcommand{\ba}{{\mathbf{a}}}
\newcommand{\bb}{{\mathbf{b}}}
\newcommand{\bx}{{\mathbf{x}}}
\newcommand{\bN}{{\mathbf{N}}}
\newcommand{\bnu}{{\boldsymbol{\nu}}}
\newcommand{\KVarC}[1]{K^{#1}_0(\Var_\CC)}

\newcommand{\Zmot}[1]{{Z^{#1}_{\mot}}}
\newcommand{\Znv}[1]{{Z^{#1}_{\nv}}}
\newcommand{\Ztop}[1]{{Z^{#1}_{\topo}}}

\newcommand{\abs}[1]{\left\lvert#1\right\rvert}
\newcommand{\ex}{k}
\newcommand{\exz}{m}
\newcommand{\estrato}[1]{H_{#1}}

\DeclareMathOperator{\str}{W}
\DeclareMathOperator{\ord}{ord}
\DeclareMathOperator{\ac}{ac}
\DeclareMathOperator{\Var}{\textbf{Var}}
\DeclareMathOperator{\id}{id}
\DeclareMathOperator{\mot}{mot}
\DeclareMathOperator{\sing}{Sing}
\DeclareMathOperator{\topo}{top}
\DeclareMathOperator{\nv}{naive}
\DeclareMathOperator{\dif}{d\!}
\DeclareMathOperator{\lcm}{lcm}

\DeclareMathOperator{\arco}{\mathcal{A}}
\DeclareMathOperator{\divisor}{div}
\newcommand{\pol}{\text{\rm Pol}^+}

\newcommand{\mkl}{\mathfrak{m}}
\newcommand{\cpolo}{\mathfrak{a}}
\newcommand{\poloq}{\mathfrak{n}}
\newcommand{\res}{\mathcal{R}}

\title[Denef-Loeser zeta functions of suspensions and Lê-Yomdin]{Denef-Loeser zeta functions of\\ suspensions and Lê-Yomdin  singularities}

\author[E.~Artal]{Enrique Artal Bartolo}
\address{Departamento de Matemáticas, IUMA, Universidad de Zaragoza, C. Pedro Cerbuna 12,
50009, Zaragoza, Spain.
}
\email{\href{mailto:artal@unizar.es}{artal@unizar.es},  \href{mailto:m.gonzalez@unizar.es}{m.gonzalez@unizar.es}, \href{mailto:eleon@unizar.es}{eleon@unizar.es}}

\author[P.~Gonz\'alez P\'erez]{Pedro D. Gonz\'alez P\'erez}
\address{Instituto de Matemática Interdisciplinar, 
Departamento de Álgebra, Geometría y Topología, 
Facultad de Ciencias Matemáticas, 
Universidad Complutense de Madrid, 
Plaza de las Ciencias 3, 28040, Madrid, Spain
}
\email{\href{mailto:pdperezg@ucm.es}{pgonzalez@mat.ucm.es}}

\author[M.~Gonz\'alez Villa]{Manuel Gonz\'alez Villa}
\address{
Centro de investigaci\'on en 
Matem\'aticas\\
Apartado Postal
402,  
C.P. 36000, Guanajuato, GTO, M\'exico.}
\email{\href{mailto:manuel.gonzalez@cimat.mx}{manuel.gonzalez@cimat.mx}}

\author[E.~Le\'on Cardenal]{Edwin Le\'on Cardenal}

\thanks{
This research was partially supported by PID2024-156181NB-C32 and PID2024-156181NB-C33 funded by MICIU/AEI/10.13039/501100011033
and by FEDER, UE. 
MGV and ELC were also partially supported 
by SECIHTI project CF-2023-G33. 
}


\date{\today}
\begin{document}

\begin{abstract}
The  holomorphy conjecture for suspensions of plane curve singularities and the holomorphy and monodromy conjectures  for Lê-Yomdin singularities of surfaces are proved. 

The first part of this  paper provides formul{\ae} for the motivic and topological zeta functions for a family of hypersurfaces, including the suspensions by an arbitrary number of points and which are more general than Thom-Sebastiani type. These formul{\ae} generalize and are inspired by the description of  the topological and the $2$-twisted topological zeta functions of suspensions by~$2$ points of hypersurfaces, due to the first named author, Cassou-Noguès, Luengo and Melle. The new general formul{\ae} deal with arbitrary values of the twisting parameter.  An interesting feature of these general formul{\ae}   is the appearance of values of the Jordan's totient function as coefficients of the topological and the twisted topological zeta functions of some auxiliary hypersurfaces of smaller dimension.

As an application, a proof of the holomorphy conjecture for the suspension of a plane curve singularity is given, completing  the  case left open  by Denef and Veys.  The aforementioned role of the Jordan's totient function  and some properties of resolutions of plane curve singularities are  key tools for our proof of the holomorphic conjecture for suspensions of curves. 

Superisolated and Lê-Yomdin singularities of surfaces are studied too.
Formul{\ae} for the motivic and topological zeta functions of superisolated and Lê-Yomdin singularities are obtained using a natural stratification of the exceptional divisor of the blow-up of the origin and our general expressions. The proof of the monodromy conjecture follows the same strategy of the superisolated case also due to to the first named author, Cassou-Noguès, Luengo and Melle, but deals with a richer geometric situation and higher technical issues. Finally, the analysis of the orders of the eigenvalues   allow us to prove the holomorphy conjecture for superisolated and  Lê-Yomdin singularities of surfaces. 
\end{abstract}
\maketitle

\thispagestyle{empty}
\vspace{-5mm}

\tableofcontents

\renewcommand{\theequation}{I.\arabic{equation}}
\section*{Introduction}

The monodromy conjecture, proposed by Igusa~\cite{Igu:88},  hints a surprising connection  between the arithmetic and topology of singular points of polynomials. For a prime $p$ and a polynomial $f \in \ZZ[x_1,\dots, x_d]$, the asymptotic behavior of the number of solutions of the congruence 
$f \equiv 0 \pmod{p^k}$ with respect to $k \in \ZZ_{>0}$ is encoded by the poles of the Igusa zeta function of $f$. The monodromy conjecture asserts that each pole corresponds  to  an eigenvalue of the monodromy of the Milnor fibration associated to the complexification of $f$. Denef and Loeser~\cite{DL-JAMS, Denef-LoeserIgusa}  formulated versions of the  conjecture for the topological zeta function, which corresponds to the limit of the Igusa zeta function when $p$ tends to $1$, and for the motivic zeta function, which is a generalization of the Igusa zeta function within the framework of motivic integration envisioned by Kontsevich, see~\cite{nic:10} for a survey on these matters.
Although this conjecture has been intensively studied in the last decades and several cases have been proved, the general problem remains widely open. The known proofs of particular cases are considered \emph{ad hoc}, proceeding by the comparison of sufficiently detailed descriptions of the set of poles and eigenvalues. These descriptions are typically obtained with the help of a (partial) resolution of the singularities of $f$ and therefore non-trivial combinatorial arguments are needed to compare them. The survey~\cite{Veys_Intro} offers a gentle introduction to the monodromy conjecture and a recent account of its status. 

The holomorphy conjecture, proposed by Denef ~\cite[Conjecture 4.4.2]{denefBourbaki}, is complementary to the monodromy conjecture. The holomorphy conjecture asserts that the Igusa zeta function of~$f$ with respect to a character whose order does not divide the order of any eigenvalue of the complexification of $f$ is holomorphic and, hence, does not contain any arithmetic information about $f$. The holomorphy conjecture has been comparatively less studied than the monodromy conjecture. It has been proved for curves by Veys~\cite{veys:curves_hol}.
Denef and Veys~\cite{dv:95} also proved the holomorphy conjecture for suspensions of the form
$F=z^\ex +f$ and $\ex >2$ (also for $\ex=2$ with additional hypotheses) conditionally on the holomorphy conjecture for $f$.
More recently, Castryck, Ibadula and Lemahieu established the case of non degenerate surfaces,  see~\cite{zbMATH07034647}. We refer to the latter for a recent account of the status of the conjecture. 

This paper proves the holomorphy conjecture for suspensions of curves, completing the remaining cases in the work of Denef and Veys in~\cite{dv:95}. We also prove the monodromy and holomorphy conjectures for the local topological zeta functions of Lê-Yomdin singularities of surfaces. The monodromy conjecture
for suspensions is quite straightforward. Indeed, the candidate poles of the suspension $F$ are of the form $s_0 + \frac{1}{\ex}$ where $s_0$ is a pole of $f$. Hence, if the monodromy conjecture for $f$ holds, then $\exp(-2 \pi i s_0)$ is an eigenvalue of $f$, and $\exp\left(-2 \pi i \left(s_0+\frac{1}{\ex}\right)\right) = \exp(-2 \pi i s_0) \cdot \exp (-\frac{2 \pi i}{\ex} )$ is an eigenvalue of $F$ due to the description of the monodromy of $F$ by Thom and Sebastiani~\cite{ST-71}. 

Lê-Yomdin singularities are hypersurface singularities
in $(\CC^{d+1},\zero)$ defined by a germ 
$F\in\CC\{x_0,x_1,\dots,x_d\}$ with homogeneous decomposition
$F=f_\exz+f_{\exz+\ex}+\dots$. Here $\exz>1$, $\ex\geq 1$, the homogeneous polynomials  $f_\exz,$ and $f_{\exz+\ex}$ are not zero,
and the singularities of the projective zero locus $C_\exz:=V_{\PP}(f_\exz)$ are not in $V_{\PP}(f_{\exz+\ex})$.
The case  $\ex=1$ is also called the case of superisolated singularities
after Luengo~\cite{Luengo87}. An interesting and useful feature
of Lê-Yomdin  singularities is that they allow
to transfer $(d-1)$-dimensional global information of the projectivized tangent cone $C_\exz$,
into $d$-dimensional local information of the germ~$F^{-1}(0)$.
This property has been used to disprove conjectures which were supported by other large families of singularities, e.g., the non-smoothness of the $\mu$-constant stratum~\cite{Luengo87},
Yau's conjecture relating abstract and embedded topology~\cite{ea:mams}, and  Seiberg-Witten conjectures~\cite{FLMN:06}, among others. Because of these counterexamples, a proof of some conjecture for superisolated  or, more generally, Lê-Yomdin
singularities, would provide strong evidence in favor of that conjecture. 

Actually, the monodromy conjecture was already proved for superisolated
singularities of surfaces in~\cite{ACNLM-ASENS}.
Our proof in this work of the 
monodromy conjecture for Lê-Yomdin
singularities of surfaces builds upon this evidence.
By the same reason our work  provides also evidence 
for the holomorphy
conjecture.

Our analysis of  Lê-Yomdin
singularities combines  the study of suspensions appearing in the resolution of singularities of the projectivized tangent cone $C_m$ through a motivic version of the stratification principle.  

The suspension of $f\in\CC\{x_1,\dots,x_d\}$ by $\ex$ points
is defined by
\[
F(x_1, \dots, x_d, z) = z^\ex + f(x_1, \ldots, x_d)\in\CC\{x_1,\dots,x_d,z\}. 
\]
Notice, by comparison with  Lê-Yomdin
singularities, that suspensions allow to
transfer local information of the $(d-1)$-dimensional
germ $f^{-1}(0)$
into $d$-dimensional local information of the germ $F^{-1}(0)$. 

The rest of the introduction is devoted to a more detailed  discussion of our results and techniques. The main objects are the naive and motivic zeta functions of a germ $f$
\[
\Znv{}(f, T)_\zero\in{\mathcal{M}}_\CC[[T]], \quad \text{and} \quad   \Zmot{}(f, T)_\zero \in{\mathcal{M}}^{\hat{\mu}}_\CC[[T]],   
\]
and their topological specializations
\[
\Ztop{}(f, s)_\zero,
\quad \text{and} \quad 
\Ztop{(\ell)}(f, s)_\zero \in \QQ(s),  \quad \text{for} \quad \ell>1,
\]
see \S\ref{subsec:arcs} and \S\ref{subsec:zeta_functions}. 
 
 Our initial motivation comes from the paper~\cite{ACNLM-JLMS}
 which proves that the topological zeta function is not a topological invariant of the singularity using the formul{\ae}
\begin{align}
\label{eq:orgininalF2points}\begin{split}
 \Ztop{}(F,s)_\zero =& \frac{1}{2s+1} + \frac{s(2s+3)}{2(s+1)(2s+1)}\Ztop{}\left(f,s+\frac{1}{2}\right)_\zero - \frac{3s}{2(s+1) } \Ztop{(2)}\left(f,s+\frac{1}{2}\right)_\zero, \\
 \Ztop{(2)}(F,s)_\zero =&
 \frac{1}{2s+1} - \frac{2s+3}{2(2s+1)}\Ztop{}\left(f,s+\frac{1}{2}\right)_\zero - \frac{1}{2} \Ztop{(2)}\left(f,s+\frac{1}{2}\right)_\zero.
 \end{split}
 \end{align}
Here the topological and the 2-twisted topological zeta functions of a suspension $F$ of a hypersurface defined by $f$ by $2$ points, are expressed in terms of the topological and the 2-twisted topological zeta functions of $f$, 
see~\cite[Theorem 1.1]{ACNLM-JLMS}. Notice that the second expression above, i.e., the  formula for  $\Ztop{(2)}(F,s)_\zero$, is needed to compute iterated suspensions by $2$ points. 
Moreover, 
a correspondent pointed out 
similar formul{\ae} for the suspension by $\ex$ points, based upon the motivic Thom-Sebastiani theorem, see~\cite[\emph{Note added in proof}, p.~53]{ACNLM-JLMS}. The proposed formul{\ae} for   suspensions by  $\ex$ points have apparently been unquestioned for more than two decades now, but they are inaccurate as we will see in \S\ref{comp}.

The first main results of this paper are Theorems~\ref{thmpsuspACNLM} and~\ref{thmfsuspACNLM}. They consist in a  far reaching generalization of the formul{\ae} in~\cite[Theorem 1.1]{ACNLM-JLMS}.  First, we consider a pair $(G, \omega_{d+1})$ consisting of a germ defined by 
\[
G(x_1,\ldots,x_d,z)=z^\exz F(x_1,\ldots,x_d,z), \quad F(x_1,\ldots,x_d,z)=f(x_1,\ldots,x_d)+z^\ex, 
\]
 and $\omega_{d+1}$ is a monomial differential form. Secondly, our results provide explicit formul{\ae} at the level of motivic and naive zeta functions. In particular, \S\ref{sec:motivic_zeta_binomials} is devoted to the computation of  $\Zmot{}(g, \omega_{d+1}, T)_{\zero}$ and $\Znv{}(g, \omega_{d+1}, T)_{\zero}$, where $g=z^\exz(z^\ex + \bx^{\bN_q})$ is a binomial.
Notice that $g$ corresponds to $G$ when $F$ is a suspension of the monomial~$\bx^{\bN_q}$. This computation is done using techniques developed in~\cite{ACNLM-AMS} and~\cite{GG:14} to give combinatorial formul{\ae} for the motivic and naive zeta functions, using generating functions of sets of integral points of cones and Newton polyhedra. By comparison, values of the topological zeta functions of suspension of monomials by $2$ points in 
\cite[Lemmas 1.2--1.4]{ACNLM-JLMS} were computed by induction.  Working at the motivic level provides a better conceptual understanding and  a framework prone to applications. 

Theorems~\ref{thmpsuspACNLM} and~\ref{thmfsuspACNLM} are obtained from the computations in~\S\ref{sec:motivic_zeta_binomials} for binomials using a motivic version of the Stratification Principle of~\cite{ACNLM-ASENS} and taking the Euler characteristic specialization, see~\S\ref{sec:top_zeta_binomials}. It is worth noting that these results compute the twisted topological zeta functions $\Ztop{(\ell)}\left(F,s\right)_\zero$ for arbitrary values of $\ell$ which is crucial for the study of the holomorphy conjecture in~\S\ref{sec:hc_kLYS}.

In the case of a suspension, i.e. $\exz=0$, and after specialization of the naive zeta function, we get the following explicit expression for the local topological zeta function
\begin{align}\label{fsuspACNLM}\Ztop{}(F,s)_\zero&= \frac{1}{\ex t}+
	\frac{\ex-1}{\ex} 
	\cdot \frac{s}{s+1} 
	\cdot \frac{t + 1}{t} 
	\cdot \Ztop{}\left(f,t\right)_\zero- \frac{s}{s+1} \sum_{1\neq e \mid \ex} \frac{J_2(e)}{\ex} \Ztop{(e)}\left(f,t\right)_\zero, 
\end{align}
where $t=s+\frac{1}{\ex}$ and  $J_2 : \NN \rightarrow \NN$ is a Jordan's totient function. Formula \eqref{fsuspACNLM} and similar expressions for the twisted zeta functions can be found in Theorem~\ref{thmfsuspACNLM}. The  combination of values of the arithmetic functions  and shifts of the (twisted) topological zeta functions of $f$
are essentially new in the literature and will be used in the applications presented below.  

The formul{\ae} for a suspension $F$ can be seen as a particular case of Thom-Sebastiani formul{\ae}, but the expressions  for $G$ go beyond the Thom-Sebastiani  setting and, in combination with the Stratification principle, are useful for  applications.

Let us now describe the applications of  Theorems~\ref{thmpsuspACNLM} and~\ref{thmfsuspACNLM}. 
Denef and Veys proved the holomorphy conjecture for a suspension $F= z^\ex + f$ conditionally to the holomorphy conjecture of $f$, unless  $k = 2$ and the orders of the eigenvalues of $f$ satisfy some special divisibility property, see Definition~\ref{def:DVbad}, in~\cite{dv:95}.  Up to our best knowledge, the remaining cases have not been addressed since then. We are able to prove these cases thanks to Theorem~\ref{thmfsuspACNLM} and the properties of the multiplicities of the irreducible components of the total transform of the embedded resolution of plane curves singularities, see Lemma~\ref{lema:cc}. Hence, we complete the proof of Denef and Veys of the holomorphy conjecture for  $\ex=2$ and $f\in \CC\{x,y\}$. As an illustration of our contribution, notice that the holomorphy of $\Ztop{(18)}(F; s)_\zero$ for the suspension  $F=f+z^2$, with $f=(y^2-x^3)^3-x^6y^3$, is predicted by the holomorphy conjecture but it does not follow from the work of Denef and Veys. The reason is that the integer 18 is $f$-bad, see Definition~\ref{def:bad_div}, but  $\Ztop{(\ell)}(f; s)_\zero$, for $\ell$ a divisor of 18, does not vanish in general.  More precisely, $\Ztop{(18)}(f; s)_\zero$ and $\Ztop{(9)}(f; s)_\zero$ are not holomorphic. However, Theorem~\ref{thmfsuspACNLM} shows that $\Ztop{(18)}(F; s)_\zero$ is proportional to $\Ztop{(18)}\left(f,s+\frac{1}{2}\right)_\zero+\Ztop{(9)}\left(f,s+\frac{1}{2}\right)_\zero$ and Lemmata~\ref{lema:cc} and~\ref{lem:notextbad} show that the latter sum vanishes, see Examples~\ref{ex:f-bad} and~\ref{ex:f-bad-1}. Remark~\ref{rmk:hol_surf} comments on the effect of this result on the status of the holomorphy conjecture for suspensions.

Next, let us discuss applications to  Lê-Yomdin singularities. As pointed to the first author by S.~Gusein-Zade, the computation of the local (twisted) topological zeta functions for Lê-Yomdin singularities~\cite{Yomdin74, Le80} reduces, after a point blow-up and applying the  Stratification principle of~\cite[1.1]{ACNLM-ASENS} to the exceptional divisor $E \cong \PP^n$,    to  computations of $\Ztop{}(G,\omega_{d+1},s)_\zero$ for  functions $G$, with arbitrary values $\exz$ and $\ex$, and a monomial volume form $\omega_{d+1}$. This reduces to the situation of Theorem~\ref{thmpsuspACNLM}. The stratification consists of three types of strata, namely, the $n$ dimensional stratum $E \setminus C_{\exz}$, the $(n-1)$-dimensional strata $C_{\exz} \setminus \sing C_{\exz}$, and $0$-dimensional strata due to $\sing C_{\exz}$, see Lemma~\ref{lem:localeqkLYS}.  Proposition~\ref{prop:DLZFkLY} and Theorem~\ref{thm:DLZFkLY} give explicit formul{\ae} for the naive, topological and twisted topological zeta functions.  In particular, Theorem~\ref{thm:DLZFkLY} expresses the topological or $\ell$-twisted topological zeta function of a $\ex$-LYS surface $F$ as a combination of shifts of local topological  zeta functions $\Ztop{(e)}\left(f_q,t\right)_\zero$, for some values of $e$, of singular points $(C_\exz, q)$ of the projectivized tangent cone  of $F$ with coefficients given by values $J_2(e)$ of the second Jordan totient function, where $f_q$ is a germ of function defining $(C_\exz, q)$.

We use two key tools to attack the monodromy and holomorphy conjectures. On one hand, we apply the description of A'Campo's monodromy zeta function for Lê-Yomdin singularities in~\cite{GLM:97, jmm:14}. On the other hand 
we provide in~Theorem~\ref{thm:DLZFkLY} a precise relation between the candidate poles and eigenvalues of $F$ to the poles and eigenvalues of the local germs $f_q$ with $q \in \sing C_{\exz}$.

The proof of the monodromy conjecture follows the guidelines of the case of superisolated singularities  in~\cite{ACNLM-ASENS}, but the spectrum of cases to consider for Lê-Yomdin singularities is larger. 
The difficult part of the proof is to control the set of eigenvalues when $\chi (\PP^2 \setminus C_{\exz})\leq 0$. In order to deal with this situation, we extensively use  Kashiwara's classification of pencils of rational curves, see \S\ref{K-K-pencils}.
Splice diagram of Kashiwara's pencils are used to compute A'Campo's monodromy zeta function. Properties of products of cyclotomic polynomials are used to control possible cancellations of candidate eigenvalues.

The most intriguing candidate pole is the log canonical threshold $\frac{3}{\exz}$, mainly when $\chi (\PP^2 \setminus C_{\exz})\leq 0$.
Its treatment uses the notion of bad divisors introduced in~\cite{ACNLM-ASENS} and also Veys' structure theorem~\cite{veys:str}. 
In \S\ref{sec:preSIS} we compare the proofs for superisolated and $\ex$-Lê-Yomdin singularities.
Although some arguments of our proof of the $\ex$-Lê-Yomdin case are similar to those in~\cite{ACNLM-ASENS}, our presentation is more intrinsic, concise and aims for greater clarity. 

Our proof of the holomorphy conjecture for $\ex$-LYS singularities of surfaces defined by $F$ follows a strategy inspired by the seminal work of Denef and Veys for suspensions, see~\cite{dv:95}. 
First we apply the recursive formul{\ae} for the twisted zeta functions of $F$ in Theorem~\ref{thm:DLZFkLY}. Then we use the description of the orders of the eigenvalues of $F$ and the holomorphy conjecture for plane curves. We combine these results to prove that whenever $\Ztop{(e)}\left(f_q,s\right)_\zero$ appears in the expression for  
$\Ztop{(\ell)}\left(F,s\right)_\zero$, the former must vanish and this implies that the conjecture holds.

In comparison with the case of suspensions of curves,  if $F$ defines a $\ex$-LYS surface, all the terms of the expression for $\Ztop{(\ell)}(F,s)_0$ are zero, while this is not the case if $F=z^2 - f(x,y)$ is the suspension of a plane curve singularity, as explained above. The proof of the holomorphy for $\ex$-LYS surfaces is independent of the proof of the monodromy conjecture, although it uses some properties of the characteristic polynomial deduced from the splice diagrams of Kashiwara's pencils, see \S~\ref{K-K-pencils}.

We expect that the results from \S\ref{sec:motivic_zeta_susp}--\ref{sec:MainResults}    may be generalized for quotient singularities using the techniques in this work and the results of~\cite{lmvv:20}. This would help to compute the 
zeta functions for weighted Lê-Yomdin singularities, see~\cite{ABLM-milnor-number,ACMNemethi} and to extend the results of the paper to this cases.
The \texttt{Sagemath} package~\cite{viu:22} has been particularly useful during the research on this project. 

The paper is organized as follows. Section~\ref{sec:arc-gr-dl} is devoted to introduce general settings and notations about Denef-Loeser zeta functions, including the change of variables formula and the stratification principle. Section~\ref{sec:motivic_zeta_susp} explains the general strategy to compute the motivic and naive zeta functions of a suspension $F= z^\ex + f(x_1, \dots, x_d)$ and of a function  $G=z^\exz F$  using a double partition of the space of arcs. By applying the stratification principle the results  are reduced to the local computations of Denef and Loeser zeta functions for some binomials determined by the strata of an embedded resolution of $f$. The contribution of a particular stratum is analyzed in~\S\ref{sec:motivic_zeta_binomials} and~\S\ref{sec:top_zeta_binomials}. Some arithmetic functions and a technical  Lemma, treated in \S\ref{ssec:arithmetic_functions}, are necessary to express the topological zeta functions of these binomials. Theorems~\ref{thmpsuspACNLM} and~\ref{thmfsuspACNLM} in~\S\ref{sec:MainResults} describe the zeta functions of $F$ and $G$. The rest of the paper is devoted to applications of these results. Section~\ref{sec:hc_kLYS} gives a complete proof of the holomorphy conjecture for suspensions of plane curve singularities.   Applications to the computation of the Denef-Loeser zeta functions for $\ex$-Lê-Yomdin singularities are developed in~\S\ref{sec:yomdin}. These results upgrade the literature that just covered the topological zeta function of superisolated singularities. Sections~\ref{sec:MCk-LYS} and~\ref{subsec:hol_lys} give proofs of the monodromy and holomorphy conjectures, respectively, for $\ex$-LYS of surfaces. 
Finally, \S\ref{comp} compares Theorem~\ref{thmfsuspACNLM} with the formula proposed in the aforementioned \emph{Note added in proof}.
It also deals with the different approaches for the proof of the monodromy conjecture between SIS and $\ex$-LYS.
 
\renewcommand{\theequation}{\arabic{equation}}
\numberwithin{equation}{section}

\section{Arc spaces, Grothendieck rings and Denef-Loeser zeta functions}
\label{sec:arc-gr-dl}
\subsection{Arc spaces and Grothendieck rings}\label{subsec:arcs}
\mbox{}

Let $\varphi$ be an element  in $\CC [[t]] \setminus \{0\}$, we denote by ${\rm ord}(\varphi)$  the \textit{order} 
(with respect to~$t$) of the series  $\varphi$, and by $\mathrm{ac}(\varphi)$ the coefficient of the leading term of $\varphi$, i.e., if $\varphi= a_n t^n + a_{n+1}t^{n+1} + \cdots$ and $a_n \ne 0$, then ${\rm ord}(\varphi)=n$ and $\mathrm{ac}(\varphi)=a_n$.

Let us fix a system of coordinates $(\bx,z):=(x_1, \dots, x_d, z)$ of $\CC^d \times \CC = \CC^{d+1}$ centered at $\zero \in \CC^{d+1}$. 
Denote by $\varphi \in \mathcal{L}_n(\CC^{d+1})_\zero$ the space of $n$-jets $\varphi$ of $\CC^{d+1}$ centered at $\zero$, i.e. such that $\varphi(0)=\zero$. 
Denote the coordinates of $\varphi \in \mathcal{L}_n(\CC^{d+1})_\zero$ as $\varphi_{x_i}$ or $\varphi_z$, i.e., $\varphi=(\varphi_{{\bf x}}, \varphi_z)$.

Let us denote by $\Var_{\CC}$ the category of complex algebraic varieties, i.e.,
reduced and separated schemes of finite type over the complex field $\CC$. Let us denote by  $\KVarC{}$ the \textit{Grothendieck ring} of algebraic varieties over $\CC$, and  by the symbol $[X]$ the class of the 
complex algebraic variety  $X$ in $\KVarC{}$. 
Denote by
${\LL}$ the class of ${\mathbb{A}}^1_\CC$ in $\KVarC{}$ and by $\mathcal{M}_\CC$
the ring $\KVarC{}[\LL^{-1}]$. 
We also consider the category of complex algebraic varieties endowed with a good action of the profinite group $\hat{\mu}$ of the roots of the unity. The corresponding Grothendieck rings are denoted by adding the superscript $\hat{\mu}$ to the usual symbols.

We consider the space of arcs $\mathcal{L}(\CC^{d})_\zero$ centered at $\zero$ and 
its truncations, the spaces of $n$-jets $\mathcal{L}_n(\CC^{d})_\zero$,
and the natural truncation map $\pi_n:\mathcal{L}(\CC^{d})_\zero\to\mathcal{L}_n(\CC^{d})_\zero$.
We have a measure~$\mu$ on the space of measurable subsets of 
$\mathcal{L}(\CC^{d})_\zero$ with values in some completion of $\KVarC{}[\LL^{-1}]$.
This measurable sets contain the cylindrical subsets $\pi_n^{-1}(A)$,
where $A$ is an algebraic (or constructible) subset of $\mathcal{L}_n(\CC^{d})_\zero$.
Then,
\begin{equation}\label{eq:measure}
\mu(\pi_n^{-1}(A))=[A]\LL^{-nd}.
\end{equation}
This induces a well-defined measure $\mu$ which also has
an equivariant version $\mu^{\hat{\mu}}$.

\subsection{Denef and Loeser zeta functions}\label{subsec:zeta_functions}
\mbox{}

Let us fix a germ $g:(\CC^d,\zero)\to(\CC,0)$ at~$\zero$ and a monomial volume form 
\[
\omega:=x_1^{\nu_1^0-1}\cdot\ldots\cdot x_d^{\nu_d^0-1}\cdot\dif x_1\wedge\dots\wedge\dif x_d\in\Omega^d(\CC^d).
\]
The local  \textit{naive motivic zeta function} and the local \textit{motivic zeta function} of  $g$ and $\omega$ are defined as
\begin{equation}\label{eq:zetamot}
\begin{aligned}
\Znv{}(g, \omega, T)_\zero&:=
\sum_{n\in \NN}\sum_{p \in \ZZ_{\geq 0}} \mu(\mathcal{X}_{n,p})
{\LL}^{-p}T^n
\in {\mathcal{M}}_\CC[[T]],\\
\Zmot{}(g, \omega, T)_\zero&:=\sum_{n\in \NN}\sum_{p \in \ZZ_{\geq 0}} 
\mu^{\hat{\mu}}(\mathcal{X}^1_{n, p }){\LL}^{-p}T^n
\in {\mathcal{M}}^{\hat{\mu}}_\CC[[T]].
\end{aligned}
\end{equation}
Here $\mathcal{X}_{n,p}$ stands for the set of arcs $\varphi \in \mathcal{L}(\CC^{d})_\zero$ 
such that $\ord g \circ \varphi =n$,
and $\ord \varphi^*\omega =p$. Moreover $\mathcal{X}^1_{n,p}$ stands for the set of 
arcs in $\mathcal{X}_{n,p}$  
such that $\ac(g \circ \varphi)=1$. The choice of a monomial form
allows the coefficients to be in $\KVarC{}[\LL^{-1}]$
and $\KVarC{\hat{\mu}}[\LL^{-1}]$ instead of completions.
As usual, if $\omega$ is not vanishing we may drop it from the notation.

\begin{remark}\label{rmkdq}    
Assume that $g(x_1, \dots, x_d) = g(x_1, \dots, x_q, 0, \dots, 0)$ for some $q < d$, i.e., $g$ does not depend on the variables $x_{q+1}, \dots, x_d$. Let $g' : (\CC^q,0) \rightarrow (\CC,0)$ be the germ given by $g'(x_1, \dots, x_q) = g(x_1, \dots, x_q, 0, \dots, 0)$ and let $\omega'$ be the restriction and contraction of $\omega$ to the subspace $x_{q+1}= \cdots = x_d=0$. Then, using \eqref{eq:zetamot}, it is easy to check that $\Znv{}(g, \omega, T)_\zero= \Znv{}(g', \omega', T)_\zero$. A similar equation for the motivic zeta function also holds.
\end{remark}

We also consider  the local topological zeta function $\Ztop{}(g,\omega,s)_\zero$ 
and the local twisted topological zeta function  $\Ztop{(\ell)}(g,\omega,s)_\zero$ of $g$ and $\omega$, at~$\zero$ given by 
\begin{equation}\label{eq:zetatop}
\chi_{\topo}(\Znv{}(g,\omega, \LL^{-s})_\zero), \text{ and }
\chi_{\topo}((\LL -1)\Zmot{}(g,\omega, \LL^{-s})_\zero, \alpha),
\end{equation}
where $\chi_{\topo}: \KVarC{}[\LL^{-1}] \rightarrow \ZZ$, denotes the usual Euler characteristic while $\chi_{\topo}(\cdot, \alpha) : \KVarC{\hat{\mu}}[\LL^{-1}] \rightarrow \ZZ$ denotes the equivariant Euler characteristic with respect to a character $\alpha: \hat{\mu} \rightarrow\CC^* $ of finite order $\ell \in \ZZ_{>0}$. 

We will refer to the functions in \eqref{eq:zetamot} and \eqref{eq:zetatop} as Denef-Loeser zeta functions.

\subsection{Formulas in terms of embedded resolutions and stratification principle}\label{subsec:formulas}
\mbox{}

Denef and Loeser gave formulas for the zeta functions in \eqref{eq:zetamot} and \eqref{eq:zetatop} in terms of resolution of singularities. Let us consider an embedded  resolution $\pi : Y \rightarrow \CC^d$ of $g\cdot x_1^{\nu_1^0-1}\cdot\ldots\cdot x_d^{\nu_d^0-1}$. Let $\{E_j\mid j\in J\}$ be the irreducible components of $\pi^* (g\cdot x_1^{\nu_1^0-1}\cdot\ldots\cdot x_d^{\nu_d^0-1})$, including both
exceptional and strict transform components.
As divisors, we write 
\begin{equation*}E:=(g\circ\pi)^{*}(0)=\sum_{j\in J} N_j E_j,
\quad
\divisor(\pi^*\omega)=\sum_{j\in J} (\nu_j-1) E_j
.
\end{equation*}
The intersection of the divisor $E$ with $\pi^{-1}(\zero)$ is stratified by
\begin{equation*}\{E^\circ_I\mid \emptyset\neq I \subset J\},\qquad
E^\circ_I := \bigcap_{i \in I} E_i \setminus \bigcup_{j \not \in I} E_j.
\end{equation*}
Since $E$ is a normal crossings divisor, for $I\subset J$ the stratum $E_I$ has dimension  $d-\abs{I}$. Therefore, given a point $o_I \in E_I^\circ$ there is a suitable coordinate system 
$\mathcal{U}$ in the Zariski topology with two subsets of coordinates. On one side there are coordinates $x_i$, $i\in I$, such that $E_i$ is defined by $x_i=0$; on the other side there is a subsystem of coordinates $y_1,\dots,y_{d-\abs{I}}$ which parametrizes~$E_I$. 
Moreover,
the pullbacks of $g$ and $\omega$ under $\pi$ will be denoted by $g_{I}$ and $\omega_{I}$, they have local equations \[
g\circ\pi=u\bx^{\bN_I}=u\prod_{i\in I}x_i^{N_i},\quad
\pi^*\omega=v\bx^{\mathbf{\nu}_I}\frac{\dif \bx}{\bx}\dif \mathbf{y}=v\prod_{i\in I}x_i^{\nu_i} \cdot  
\bigwedge_{i\in I}\frac{\dif x_i}{x_i}\wedge\dif y_1\wedge\dots\wedge\dif y_{d-\abs{I}}
\]
where $u,v$ are non-vanishing functions in $\mathcal{U}$ depending on the two subsets of coordinates, namely, coordinates
$x_i$, for $i\in I$, and coordinates $y_1,\dots,y_{d-\abs{I}}$.

The following formulas hold:
\begin{equation}\label{eq:zetas_res_mot}
\begin{aligned}
\Znv{}(g, \omega, T)_\zero&=
\sum_{\emptyset \ne I \subset J}
[E_I^\circ] \prod_{i\in I}\frac{\LL-1}{\LL^{\nu_i}-T^{N_i}}T^{N_i} \in {\mathcal{M}}_\CC[[T]],\\
\Zmot{}(g,\omega, T)_\zero&=
\sum_{\emptyset \ne I \subset J}
[\tilde{E}_I^\circ]^{\hat{\mu}} \prod_{i\in I}\frac{\LL-1}{\LL^{\nu_i}-T^{N_i}}T^{N_i} \in {\mathcal{M}}^{\hat{\mu}}_\CC[[T]],\\
\end{aligned}
\end{equation}
where $\tilde{E}_I^\circ$ denotes the unramified Galois cover  of $E_I^\circ$
introduced in~\cite[\S3.3]{Denef-LoeserBarca}.

\begin{remark}\label{rem:acampo}
   The monodromy zeta function of $g$ is computed in a similar way in A'Campo's formula, see~\cite{MR371889}:
\[
\zeta_g(\tau)=\prod_{i\in J} (\tau^{N_i} - 1)^{\chi (E_i^\circ)}.
\]
If $g$ has isolated singularity the characteristic polynomial $\Delta(\tau)$ of the 
monodromy for the reduced cohomology in dimension~$d-1$
of the Milnor fiber is obtained as $\Delta(\tau)=\zeta_g(\tau)^{(-1)^{d-1}}(\tau-1)^{(-1)^{d}}$.
In particular, if $d=2$, and $v_i$ is the valence of the vertex associated to $E_i$
in the dual graph of the embedded resolution of 
$\{ g = 0  \}$, then
\[
\Delta(\tau)=(\tau - 1)\prod_{i\in J} (\tau^{N_i} - 1)^{v_i-2}.
\]
Only the branching components, $v_i>2$, and the \emph{leaves}, $v_i=1$, matter.
\end{remark}

\subsubsection{Euler characteristic  specialization}\label{subsec:topological_specialization}
\mbox{}

To recover the expression of the topological zeta function in terms of $\pi$, we recall that $\chi$ is additive and $\chi(\LL)=1$. Since
$[\PP^a]= 1 + \LL + \LL^2 + \cdots + \LL^a\in \KVarC{}$, for $a\in \NN$, then $\chi([\PP^a])= a +1$. Moreover, the expression
\[
[\PP^a]= 1 + \LL + \LL^2 + \cdots + \LL^a=
\frac{1-\LL^{a+1}}{1-\LL}\in \mathcal{M}_\CC,
\] 
formally implies that (see~\cite[Corollaries 6.6 (iii)]{MR2325153})
\begin{equation}\label{eq:identity_euler}
	\chi \left (  \frac{1 - \LL}{ 1- \LL^{a+bs}} \right ) = \chi \left ( \frac{1}{[\PP^{a + bs -1}]} \right ) =\frac{1}{a+bs}.
\end{equation}
Hence the following formulas follow from \eqref{eq:zetas_res_mot}
\begin{equation}\label{eq:zetas_res_top}
	\begin{aligned}
		\Ztop{}(g,\omega,s)_\zero&=
		\sum_{\emptyset \ne I \subset J} \frac{ \chi( E_I^\circ)}{\prod_{i \in I} (N_is + \nu_i)}, \\
		\Ztop{(\ell)}(g,\omega, s)_\zero&=\sum_{\substack{\emptyset \ne I \subset J\\ \ell \mid N_i}} \frac{ \chi( E_I^\circ)}{\prod_{i \in I} (N_is + \nu_i)}.
	\end{aligned}
\end{equation}
According to~\cite{DL-JAMS} the local topological zeta function of a regular function $g$ satisfies the identity $\Ztop{}(g,0)_\mathbf{0} =1$, which can be rephrased,  using the expression of the local zeta function in terms of an embedded resolution \eqref{eq:zetas_res_top}, as
\begin{equation}\label{DLz(0)=1}
\Ztop{}(f,0)_\zero=\sum_{\emptyset \ne I \subset J}  \frac{\chi(E_I^\circ) }{\displaystyle\prod_{i \in I}  \nu_i} = 1.
\end{equation}
If $\omega = x_1^{\nu_1^0-1}\cdot\ldots\cdot x_d^{\nu_d^0 - 1}\cdot\dif x_1\wedge\ldots\wedge\dif x_d$, then we have
\begin{equation}\label{DLz(0)=1omega}
\Ztop{}(f,\omega,0)_\zero=\sum_{\emptyset \ne I \subset J}  \frac{\chi(E_I^\circ) }{\displaystyle\prod_{i \in I}  \nu_i} = 
\prod_{j=1}^d \frac{1}{\nu^0_i}.
\end{equation}
\subsubsection{Change of variables formula and stratification principle}\label{subsec:stratum_principle}
\mbox{}

Motivic and naive zeta functions can be defined in a more general context. We consider, for example, the following  setting. Let $V$ be an analytic manifold and let $E\subset V$ be a compact analytic subvariety with projective structure. Take $F:(V,E)\to\CC$ a germ of analytic function on~$E$, and let $\omega$ be a germ holomorphic form of maximal dimension on~$E$. Then $\Znv{}(F,\omega, T)_E$, $\Zmot{}(F,\omega, T)_E$, and $\Ztop{}(F,\omega, s)$ are defined as in~\eqref{eq:zetamot} but using this time arcs starting in~$E$.

\begin{prop}[Change of variables formula, {\cite[Lemma 3.3]{DLmot}}]
\label{prop:cv}
Consider two pairs of analytic manifolds and subvarieties, $(V,E)$ and $(\tilde{V},\tilde{E})$ as above. 
Assume that $\pi:(\tilde{V},\tilde{E})\to (V,E)$ is a proper birational morphism such that $\tilde{E}=\pi^{-1}(E)$.
Then
\[
\Znv{}(F,\omega, T)_E=\Znv{}(F\circ\pi,\pi^*\omega, T)_{\tilde{E}};
\]
similar formulas hold for $\Zmot{}$ and $\Ztop{}$.
\end{prop}

For the sake of completeness we state a generalization of the Stratification Principle, 
that will we be used in~\S\ref{sec:motivic_zeta_susp}.
\begin{prop}[{\cite[Stratification Principle]{ACNLM-ASENS}}]\label{prop:sp}
Let $X$ be a smooth analytic variety. Let $g : X \rightarrow \CC$ be a regular function and let $\omega$ be a volume form. Take $E\subset X$ as a projective subvariety of $X$. Assume that $E = \bigcup_{S \in\mathcal{S}} S$ is a finite prestratification of $E$ by quasiprojective subvarieties such that for each $S\in\mathcal{S}$ and for each $y \in S$, the local naive zeta function at $y$ depends  only on the stratum $S$. Then 
\[
\Znv{}(f,\omega,T)_E = \sum_{S \in \mathcal{S}} [S] \Znv{}(f, \omega, T ; S),
\]
where $\Znv{}(g, \omega, T; S)$ denotes the common local zeta function associated with the stratum $S$.
\end{prop}
The above principle extends to the motivic zeta function, and implies similar statements for the topological and twisted topological zeta functions. 

\section{Computation of the local motivic zeta functions of \texorpdfstring{$G$}{G}}
\label{sec:motivic_zeta_susp}

Let us introduce the main functions whose zeta functions will be described in this paper. 

Let us fix $\ex \in \ZZ_{>0}$ and  a regular function $f:(\CC^d,\zero) \rightarrow (\CC,0)$. The suspension of $f$ by $\ex$ points is defined by the regular function 
$F: (\CC^{d+1}, \zero) \rightarrow (\CC, 0)$, given by the expression 
\[
F(x_1, \dots, x_d, z) = z^\ex + f(x_1, \ldots, x_d).
\]
For the sake of applications, we are interested in the more general function
\[G(x_1, \dots, x_d, z):=z^\exz  F(x_1, \dots, x_d, z),\quad\text{with }\quad \exz \in \ZZ_{>0}.
\]

We outline the general strategy for the computation of the local naive and motivic zeta functions $\Zmot{}(G,\omega_{d+1}, T)_\zero$ and $\Znv{}(G,\omega_{d+1}, T)_\zero$ of $G$ and a monomial volume form $\omega_{d+1}$ given by 
\[
\omega_{d+1}:=\bx^{\bnu^0} z^{\nu_z}\frac{\dif \bx}{\bx}
\frac{\dif z}{z}=x_1^{\nu_1^0}\cdot\ldots\cdot x_d^{\nu_d^0}z^{\nu_z}
\frac{\dif x_1}{x_1}\dots \frac{\dif x_d}{x_d}\frac{\dif z}{z},
\]
with $\bnu^0=(\nu_1^0,\dots,\nu_d^0)\in\NN^d$ and $\nu_z\in\NN$.

We use the following double partition of $\mathcal{L}(\CC^{d+1})_\zero$. In first place, considering the contact of the arcs with the functions $z^\ex$, $f$, and $G$, each space $\mathcal{X}_{n,p}$ decomposes into the following subsets:
\begin{align}
\nonumber
	\mathcal{X}_{n,p}^{\sigma^+}&:=
	\{\varphi \in \mathcal{X}_{n,p}\mid (\exz+\ex) \ord_t \varphi_z > \ord_t (z^\exz f)\circ \varphi= n\}\\
    \nonumber
	\mathcal{X}_{n,p}^{\sigma^-}&:=
	\{\varphi \in \mathcal{X}_{n,p}\mid (\exz+\ex) \ord_t \varphi_z < \ord_t (z^\exz f)\circ \varphi = n\}\\
    \label{eq:arcos_thom_sebastiani}
	\mathcal{X}_{n,p}^{\rho\hphantom{{}^+}}&:=
	\{\varphi \in \mathcal{X}_{n,p}\mid (\exz+\ex) \ord_t \varphi_z = \ord_t (z^\exz f)\circ \varphi = n\}\\
    \nonumber
	\mathcal{X}_{n,p}^{\rho^\ast}\ &:=
	\{\varphi \in \mathcal{X}_{n,p}\mid (\exz+\ex) \ord_t \varphi_z = \ord_t (z^\exz f)\circ \varphi < n\}.
\end{align}
These subsets form a partition of $\mathcal{X}_{n,p}$ and they induce a decomposition of $\Znv{}(G,\omega_{d+1},T)_\zero$ into four pieces 
\begin{equation}\label{eq:globalbullet}
	\Znv{}(G,\omega_{d+1}, T)_\zero=\!\!\!\!\!\sum_{\bullet\in \{\sigma^+,\sigma^-,\rho,\rho^\ast\}}\!\!\!\!\!\!\! \Znv{\bullet}(G,\omega_{d+1},T)_\zero,
\end{equation}
where $\Znv{\bullet}(G,\omega_{d+1},T)_\zero$ is defined as in \eqref{eq:zetamot} by taking  $\mathcal{X}^{\bullet}_{n,p}$ instead of $\mathcal{X}_{n,p}$. Analogously one can make a partition of  $\mathcal{X}^{1}_{n,p}$ and induce a decomposition for $\Zmot{}(G,\omega_{d+1},T)_\zero$.

The second partition just takes into account the first $d$ coordinates and depends on the choice of an embedded  resolution $\pi : Y \rightarrow \CC^d$ of $f\omega_d$, where $\omega_d$ is the volume form $\bx^{\bnu^0} \frac{\dif \bx}{\bx}$. Following the notation of the previous section we denote by $\{E_j\mid j\in J\}$,  the irreducible components of $\pi^* (f\bx^{\bnu^0})$. Since the map $\pi\times\id:Y\times\CC\to\CC^{d+1}$ is proper and birational, the sets $\pi^{-1}(\zero)$ and $(\pi\times\id)^{-1}(\zero)$ can be canonically identified. In particular we still denote by $\{E_j\mid j\in J\}$ as the set of irreducible components of $(\pi\times\id)^{-1}(\zero)$. 

This second partition can be applied to the summands in the right-hand side of  \eqref{eq:globalbullet} to obtain
\begin{equation}\label{eq:doublepart}
	\Znv{\bullet}(G,\omega_{d+1},T)_\zero=
	\sum_{I \subset J}
	\mathcal{W}^{\bullet}_{I}\quad\text{and}\quad  \Zmot{\bullet}(G,\omega_{d+1},T)_\zero=
	\sum_{I \subset J}
	\mathcal{W}^{\hat{\mu},\bullet}_{I},
\end{equation}
where each of the terms $\mathcal{W}^{\bullet}_{I}$ and $\mathcal{W}^{\hat{\mu},\bullet}_{I}$ is a sum of contributions over all the points $o_I \in E_I^\circ$.

For $I\subset J$ and a point $o_I \in E_I^\circ$ we consider a system of coordinates in ~$\mathcal{U}\times \CC$ with $\mathcal{U}$ as in \S\ref{subsec:formulas} and an extra variable $z$ in the last place. Then the pullbacks of $G$ and $\omega_{d+1}$ under $\pi\times\id$, denoted by $G_{I}$ and $\omega_{I}$, have local equations
$z^\exz (z^\ex + \bx^{\bN_I} u)$ and
$v\bx^{\bnu_I}  z^{\nu_{z}}  \frac{\dif \bx}{\bx} \dif \mathbf{y} \frac{\dif z}{z}$,
for some non-vanishing functions $u, v$ which may depend on all the variables.

In order to compute the summands 	$\mathcal{W}^{\bullet}_{I}$ and $\mathcal{W}^{\hat{\mu},\bullet}_{I}$  associated to $G_I$ and $\omega_I$ at a point $o_I \in E_I^\circ$ in \eqref{eq:doublepart},  we notice that  $\ord (G_I \circ \varphi)$ and $\ord (\omega_I \circ \varphi)$ are independent of $\ord \varphi_{y}$. By Remark~\ref{rmkdq}  the subset $\mathbf{y}$ of coordinates can be omitted and we can reduce the computation to the case where
the ambient space is of dimension $|I|+1$. Then, we can assume:
\begin{equation}\label{eq:localeq}
G_I = z^\exz (z^\ex + \bx^{\bN_I}) \quad \text{and} \quad   \omega_I=\bx^{\bnu_I}  z^{\nu_{z}}  \frac{\dif \bx}{\bx}  \frac{\dif z}{z}.
\end{equation}
The computations of the zeta functions for \eqref{eq:localeq} are done in \S\ref{sec:motivic_zeta_binomials}.  Moreover, the exponents $N_I$, and $\nu_I$ depend only on the index set $I$ but not on the choice of $o_I \in E_I^\circ$. And the exponents $\exz, \ex$ and $\nu_z$ are independent of $I$. Hence, the Stratification Principle (Proposition~\ref{prop:sp}) applies to the summands $\mathcal{W}^{\bullet}_{I}$ and $\mathcal{W}^{\hat{\mu},\bullet}_{I}$. 

\begin{lemma}\label{lema:globalW}
	The  summands $\mathcal{W}^{\bullet}_{I}$ of the function
    $\Znv{^\bullet}(G_I, \omega_I, T)_\zero$ are given by
	\[\mathcal{W}_I^\bullet = [E_I^\circ]\str^\bullet_{I,\zero},
	\]
	where $\str^\bullet_{I,\zero}$ denotes the local contributions of $G_I,\omega_I$
    as in \eqref{eq:localeq}. Such contributions are computed in Proposition{\rm~\ref{prop:W}}. A similar expression can be given for $\Zmot{^\bullet}(G_I, \omega_I, T)_\zero$.
\end{lemma}
Finally,  by specializing with the Euler characteristic as explained in \S\ref{subsec:topological_specialization}, we obtain the following expressions 
\begin{equation}\label{eq:stratum_top}
	\Ztop{\bullet}(G,\omega_{d+1},T)_\zero=
	\sum_{I \subset J}
	\mathcal{W}^{\bullet}_{\topo,I}\quad\text{and}\quad  Z^{(\ell),\bullet}_{\topo}(G,\omega_{d+1},T)_\zero=
	\sum_{I \subset J}
	\mathcal{W}^{(\ell),\bullet}_{\topo,I},
\end{equation}
that are analogues of the ones in \eqref{eq:doublepart}.
Moreover, as in Lemma~\ref{lema:globalW}, we have
\begin{equation}\label{eq:stratum_top_bis}
    \mathcal{W}^{\bullet}_{\topo,I} = \chi(E_I^\circ) \str^\bullet_{\topo,I,\zero}\quad \text{and}\quad \mathcal{W}^{(\ell),\bullet}_{\topo,I} = \chi(\tilde{E}_I^\circ,\alpha) \str^{(\ell),\bullet}_{\topo,I,\zero},
\end{equation}
where $\alpha$ is a character of order~$\ell$.

  \section{Denef and Loeser Zeta functions for certain binomials} 
\label{sec:motivic_zeta_binomials}

In \S\ref{sec:motivic_zeta_susp} we reduce the computation of the Denef-Loeser zeta functions of $F$ and $G$ and a volume form $\omega_{d+1}$ to the computation of the zeta functions of the local pullbacks $F_I$, $G_I$ and $\omega_I$ given in~\eqref{eq:localeq}. The aim of this  section is to give explicit formulas for the latter. In particular we will give a detailed computation of the terms $\str^\bullet_{I,\zero}$ in Lemma~\ref{lema:globalW} in \S\ref{explicitW}.

Note that the number of variables in the local pullbacks $F_I, G_I$ and $\omega_I$ depends on the stratum $E_I$, it is exactly $|I|+1$. We set $q=|I|$. Following the discussion in \S\ref{sec:motivic_zeta_susp}, we assume in this section that the ambient space is $\CC^{q+1}$ with variables  
$\bx$ and $z$, where $\bx=(x_1,\dots,x_q)$. We consider the pair given by   
\begin{equation}\label{eq:g}
	g(\bx,z) =  z^\exz (z^\ex +  \bx^{\bN_q})\in\CC[\bx,z],
	\quad \text{and}\quad
	\omega := \bx^{\bnu_q} z^{\nu_z} \frac{\dif \bx}{\bx} \frac{\dif z}{z}\in\Omega^{q+1}(\CC^{q+1}).
\end{equation}

\subsection{Newton polyhedron and dual subdivision}\label{ssec:Newton_pol} 
\mbox{}

We consider the space $\RR^q \times \RR$ with the canonical basis $B=\{\mathbf{e}_1, \dots, \mathbf{e}_q, \mathbf{e}_z\}$.  We denote by $\langle \, , \,\rangle$ the standard scalar product of $\RR^q \times \RR$ for which $B$  is an orthonormal basis. 
If  $\bb=(\bb_q,b_z)\in \RR^q \times \RR$,  we mean $\bb_q\in \RR^q$ and $b_z\in \RR$.
We often consider the element $\bb$ as a linear form on 
$\RR^q \times \RR$ by using the scalar product and similarly for $\bb_q$.

\medskip

The Newton polyhedron of $g$  is a subset of $\RR^{q}_{\geq 0} \times \RR_{\geq 0}$ which has only one compact edge $\mathcal E$ with vertices 
\[ 
\bN:=(\bN_q,\exz ) \mbox{ and }  (\mathbf{0},\exz+\ex), 
\]
see Figure~\ref{fig:newton}. 
\begin{figure}[!htb]
\centering
\begin{subfigure}[t]{0.45\textwidth}
\centering
\begin{tikzpicture}[scale=1.2]
\coordinate (O) at (0,0);
\coordinate (X) at (-135:1);
\coordinate (Y) at (1,0);
\coordinate (Z) at (0,1.5);
\coordinate (P+) at ($3/4*(Z)$);
\coordinate (P-) at ($1/3*(Z) - .01*(X) + 1/3*(Y)$);

\draw (O) -- (X) (O) -- (Y) ;
\path[name path=a0]  (O) -- (Z);
\fill (P+) node[above right] {$(\zero,\exz +\ex)$} circle [radius=.05cm];
\draw[name path=a1] ($(X)+(P+)$) -- ($(O)+(P+)$) -- ($(Y)+(P+)$);
\path [name intersections={of=a0 and a1,by=Z1}];
\fill (Z1) circle [radius=.05cm];
\draw (Z) -- (Z1);

\fill (P-) node[below right] {$({\bN_q},\exz )$} circle [radius=.05cm];
\draw (P-)  -- node[left] {$\mathcal E$}  (P+);
\draw[name path=a2] ($(X)+(P-)$) -- ($(O)+(P-)$) -- ($(Y)+(P-)$);
\path [name intersections={of=a0 and a2,by=Z2}];
\draw[dotted] (Z1) -- (Z2);
\draw (Z2) -- (O);
\draw[dotted] (Z1) -- (Z2);
\end{tikzpicture}
\caption{Newton polygon of $g$ with three compact faces: the vertices $(\zero,\exz +\ex )$, $(\bN_q,\exz )$ and the edge~$\mathcal E$.}
\label{fig:newton}
\end{subfigure}
\hspace{1cm}
\begin{subfigure}[t]{0.45\textwidth}
\centering
	\begin{tikzpicture}[scale=2]
		\coordinate (Z) at (170:1);
		\coordinate (X) at (245:1);
		\coordinate (Y) at (-30:1);
		\coordinate (XZ) at ($.5*(X) + .5*(Z)$);
		\coordinate (YZ) at ($.5*(Y) + .5*(Z)$);
		
		\foreach \x in {Z, X, Y, XZ, YZ}
		{
			\fill (\x) circle [radius=0.075cm];
		}
		
		\draw (Z) node[left=3pt] {$\mathbf{e}_{z}$} --
		(X) node[left=3pt] {$\mathbf{e}_1$} --
		(Y) node[right=3pt] {$\mathbf{e}_d$} -- cycle;
		
		\node[left=3pt] at (XZ) {};
		\node[above right=3pt] at (YZ) {};
		
		\draw  (XZ) -- (YZ);
		\node at ($1/3*(Z) + 1/3*(XZ) + 1/3*(YZ)$) {$\sigma^+$};
		\node at ($1/3*(X) + 1/3*(Y) + 1/4*(XZ) + 1/4*(YZ)$) {$\sigma^-$};
		\node[below] at ($1/2*(XZ) + 1/2*(YZ)$) {$\rho$};
		
	\end{tikzpicture}
	\caption{Projectivization of the dual subdivision associated to the Newton polyhedron of $g$.}
	\label{fig:dual}
    \end{subfigure}
\caption{}
\end{figure}
It induces a dual subdivision, depicted in Figure~\ref{fig:dual}, of the cone $\RR^{q+1}_{\geq 0}$ consisting of the faces of the following three cones:
\begin{align}
	\sigma^{+}\!&:=\{ \bb \in \RR^{q+1}_{\geq 0} \mid \langle \bb , \bN \rangle
	< (\exz  + \ex)b_{z}\},\notag\\ 
	\sigma^{-}\!&:=\{ \bb \in \RR^{q+1}_{\geq 0} \mid \langle \bb , \bN \rangle
	> (\exz  + \ex)b_{z} \},\label{eq:partition}\\ 
	\rho\ &:=\{\bb \in \RR^{q+1}_{\geq 0} \mid \langle \bb , \bN \rangle
	= (\exz  + \ex)b_{z}\}.\notag
\end{align}
This subdivision is related with the decomposition~\eqref{eq:arcos_thom_sebastiani}.
We set 
\[
\bnu = (\bnu_q,  \nu_z ) \in \RR^q \times \RR\quad\text{ and }\quad
\one:= 
 \mathbf{e}_z +\sum_{i=1}^q \mathbf{e}_i.
\]
If $ \mathbf{b}\in\RR^{q+1}_{\geq 0}$ we denote 
\[
p_{\mathbf{b}}:=\langle\mathbf{b},\bnu - \mathbf{1}\rangle,\quad 
n_{\mathbf{b}}:=
\begin{cases}
	\langle\mathbf{b},\mathbf{N}\rangle&\text{ if }\mathbf{b}\in\sigma^+,\\
	(\exz  + \ex)b_{z} &\text{ if }\mathbf{b}\in\sigma^-,\\
	(\exz  + \ex)b_{z}=\langle\mathbf{b},\mathbf{N}\rangle&\text{ if }\mathbf{b}\in\rho.
\end{cases}
\]
The value $n_{\mathbf{b}}$ is the minimum of the linear form $\mathbf{b}$ on the Newton polygon of $g$ is determined by the dual subdivision.

\subsection{Generating functions over cones}\label{ssec:generating_functions}
\mbox{}

We require some notions from the theory of integer points on rational polyhedral cones, our source is~\cite[Section 4.6]{stanley}. 
 Given a pointed convex rational polyhedral cone $\mathcal{C} \subset \RR^{q}\times\RR$, let $\mathring{\mathcal{C}}$ be its relative interior.  
The generating function of  $\mathcal{C}$ is 
\[
\Phi_\mathcal{C}(\mathbf{t}) := \sum_{\boldsymbol{\alpha} \in 
	\mathring{\mathcal{C}} 
} \mathbf{t}^{\boldsymbol{\alpha}}. 
\]
The cone~$\mathcal{C}$ is simplicial if
it is spanned by the primitive integral vectors $\ba_i$, $i=1,\dots,\dim\mathcal{C}$,
 defining its rays. Then, we consider the finite set
\[
D_\mathcal{C}:=\left\{ \boldsymbol{\lambda} \in\mathring{\mathcal{C}}\, \middle|\ \boldsymbol{\lambda} = \sum_{i=1}^{\dim\mathcal{C}} \lambda_i \ba_i \quad \text{with}\ \lambda_i\in(0,1] 
\right\}.
\]
The multiplicity of the cone $\mathcal{C}$ is the number of elements of $D_\mathcal{C}$.
Then, we have that 
\[
\Phi_\mathcal{C}(\mathbf{t})=P_{\mathcal{C}}(\mathbf{t})\prod_{i=1}^{\dim\mathcal{C}} (1 - \mathbf{t}^{\ba_i})^{-1}\text{ where }P_{\mathcal{C}}(\mathbf{t}):=\sum_{\boldsymbol{\beta} \in D_\mathcal{C}}\mathbf{t}^{\boldsymbol{\beta}}.
\]
In the particular case of the dual subdivision associated to the Newton polyhedron of $g$ we note that the cones $\sigma^+$ and $\rho$ are simplicial.
A system of primitive integral vectors defining the rays of
$\rho$ 
is given by
\begin{equation}
    \label{eq:v_prim}
\mathbf{v}_i:=\frac{\ex\mathbf{e}_i + N_i\mathbf{e}_z}
{k_i}, \, \, 
i\in \{1,\ldots,q\},\quad \mbox{ where }
k_i :=\gcd(\ex,N_i).
\end{equation}
The primitive integral vectors defining the rays of $\sigma^+$ are the vectors
defining the rays of  $\rho$  together with  $\mathbf{e}_z$.
We consider the coordinates 
 $\mathbf{t} = (t_1, \dots, t_q, t_z)$ in order to distinguish the last coordinate.
Thus,
\[
\Phi_{\rho}(\mathbf{t}) = P_\rho(\mathbf{t})
\prod_{j=1}^q P_j(t_j, t_z)^{-1},\qquad 
\Phi_{\sigma^+}(\mathbf{t}) = P_{\sigma^+}(\mathbf{t})
P_z(z)^{-1}\prod_{j=1}^q P_j(t_j,t_z)^{-1},
\]
where 
\[
P_j(t_j, t_z):=1 - \left(t_j^{\ex}t_z^{N_j}\right)^{\frac{1}{{k_j}}},\ j\in \{1,\ldots,q\},\quad P_z(t_z)=1-t_z.
\]
In general, the cone $\sigma^-$ is not simplicial but we have
\begin{equation*}\Phi_{\sigma^-}=\Phi_{\RR^{q+1}_{\geq 0}} -\Phi_{\sigma^+} -\Phi_{\rho},\quad
    \text{ where }\quad
	\Phi_{\RR^{q+1}_{\geq 0}}=(1-t_z)^{-1}
	\prod_{i=1}^q (1 - t_i)^{-1}.
\end{equation*}

\subsection{A formula for $\str^\bullet_{I,\zero}$}\label{explicitW}
\mbox{}

Here we carry out the general strategy of 
\S \ref{sec:motivic_zeta_susp} for the expressions $g$ and $\omega$ in~\eqref{eq:g}.  Since the index set $I$ is fixed from the resolution of $g$, we write $\str^\bullet$ instead of $\str^\bullet_{I,\zero}$ to simplify the notation. We restrict ourselves to the space $\arco$  of arcs based at $\mathbf{0}$ which do not belong to the arc space of the union of the coordinate hyperplanes and the hypersurface $V(g)$ (as we discard a subset with measure~zero, this is not an issue).
To such an arc $\varphi$, we associate vectors $\ac\varphi:=\mathbf{a}\in(\CC^*)^{q+1}$ and $\ord\varphi:=\mathbf{b}\in\NN^{q+1}$ determined by the expressions
\[
\ord\varphi_j = b_j \quad \text{  and }  \quad \ac(\varphi_j) = a_j, \quad  \text{i.e.}, \quad  \varphi_j = a_jt^{b_j} + \cdots  \, \text{ for  } j\in \{1,\ldots,q,z\}.
\]
We consider a partition of $\arco$ into four subspaces according to the subdivision~\eqref{eq:partition} (compare with~\eqref{eq:arcos_thom_sebastiani}):
\begin{align*}
\arco_{\sigma^+}:&=\{\varphi\in\arco\mid\ord(z^{\exz+\ex} \circ \varphi) > \ord (\mathbf{x}^{\mathbf{N}_q}z^\exz \circ \varphi)=\ord(g \circ \varphi)\}\\
&=\{\varphi\in\arco\mid(\exz  + \ex)b_{z} > \langle\mathbf{b},\mathbf{N}\rangle\}\\
\arco_{\sigma^-}:&=\{\varphi\in\arco\mid\ord(g \circ \varphi)=
\ord(z^{\exz+\ex} \circ \varphi) < \ord (\mathbf{x}^{\mathbf{N}_q}z^\exz \circ \varphi)\}\\
&=\{\varphi\in\arco\mid(\exz  + \ex)b_{z} < \langle\mathbf{b},\mathbf{N}\rangle\}\\
\arco_{\rho}\ :&=\{\varphi\in\arco\mid\ord(g \circ \varphi)=\ord(z^{\exz+\ex} \circ \varphi)= \ord (\mathbf{x}^{\mathbf{N}_q}z^\exz \circ \varphi)\}\\
&=
\{\varphi\in\arco\mid(\exz  + \ex)b_{z} = \langle\mathbf{b},\mathbf{N}\rangle,\ 
\ac(z^{\exz+\ex} \circ \varphi ) \ne  \ac(\mathbf{x}^{\mathbf{N}_q}z^\exz   \circ \varphi)\}\\
\arco_{\rho^\ast}:&=\{\varphi\in\arco\mid\ord(g \circ \varphi)>\ord(z^{\exz+\ex} \circ \varphi)= \ord (\mathbf{x}^{\mathbf{N}_q}z^\exz \circ \varphi)\}\\
&=
\{\varphi\in\arco\mid(\exz  + \ex)b_{z} = \langle\mathbf{b},\mathbf{N}\rangle,\ 
\ac(z^{\exz+\ex} \circ \varphi ) =  \ac(\mathbf{x}^{\mathbf{N}_q}z^\exz   \circ \varphi)\}.
\end{align*}
This subdivision induces a partition of each  space $\mathcal{X}_{n,m}$ of $n$-jets into four subsets 
\[
\mathcal{X}_{n,p}^{\bullet} = \mathcal{X}_{n,p} \cap\arco_\bullet, \quad  \mbox{ with } \quad \bullet \in \{ \sigma^+, \rho , \sigma^-, \rho^\ast \},
\]
and splits the (naive) zeta function into four summands
\begin{equation} \label{eq:W}
  \Znv{}(g, \omega, T)_\zero =\!\!\!\!\!\!\! \!\!\!\sum_{\bullet \in \{ \sigma^+, \rho , \sigma^-, \rho^\ast \}} \!\!\!\!\!\!\!\!\!\!\!\!\str^{\bullet}, \qquad
\text {with }  \, \str^\bullet :=\sum_{n\in \NN}\sum_{p \in \ZZ_{\geq 0}} \mu(\mathcal{X}_{n,p}^\bullet)\LL^{-p}T^n
\end{equation}
Now, by using \eqref{eq:measure} and the definitions of $p_{\mathbf{b}}$ and $n_{\mathbf{b}}$, if $\bullet \in \{ \sigma^+, \sigma^-,\rho \}$, then
\[
\str^\bullet =  \sum_{ \mathbf{b} \in \bullet  \cap \NN^{q+1}} [\pi_{n_{\mathbf{b}}}(\mathcal{X}_{n_{\mathbf{b}},p_{\mathbf{b}}}^\bullet)] 
\LL^{-n_{\mathbf{b}} (q + 1)-p_{\mathbf{b}}}T^{n_{\mathbf{b}}}.
\]
If $\varphi\in\arco_{\rho^\ast}$, then we have $\mathbf{b}=\ord\varphi\in\rho$, $p_\mathbf{b}=\ord(\varphi^*\omega)$ and 
$n_\mathbf{b}<\ord(g\circ\varphi)$. We obtain
\[
\str^{\rho^\ast} = \sum_{ \mathbf{b} \in \rho  \cap \NN^{q+1}} 
\left (\sum_{i=1}^{\infty}   [\pi_{n_{\mathbf{b}}+i}(\mathcal{X}_{n_{\mathbf{b}}+i,p_{\mathbf{b}}}^{\rho^\ast})] \LL^{-i(q+1)} T^i
\right) \LL^{-n_{\mathbf{b}}(q + 1) - p_{\mathbf{b}}} T^{n_{\mathbf{b}}}. 
\]
The following lemma can be deduced from~\cite{dh:01,ACNLM-AMS,GG:14}. Set $g_0=z^\ex +\bx^{\bN_q}$ in $(\CC^*)^{q+1}$, then the key of the proof is to study the class of the zero locus of $g_0$, 
$[V(g_0)]$, and the class of its complement. Let us denote 
\begin{equation}  \label{eq:nqeq}
n_q
=\gcd \mathbf{N}_q, \mbox{ and } e_q:=\gcd(\ex, n_q ) \end{equation}
Since $V(g_0)$ is in $(\CC^*)^{q+1}$,
we can make a toric change of variables such that 
\[
V(g_0)\cong(\CC^*)^{q-1}\times\{(y,z)\in(\CC^*)^2\mid z^\ex - y^{n_q}=0\}\cong 
e_q \text{ copies of }(\CC^*)^{q}.
\]
Hence, in $\KVarC{}$ we have 
\[
[V(g_0)]= e_q (\LL-1)^q \quad \text{ and } \quad [(\CC^\ast)^{q+1} \setminus V(g_0)] = (\LL-1)^q (\LL -1 -e_q).
\]
\begin{lemma}\label{lema:b}
Let $\mathbf{b}\in\NN^{q+1}$.
\begin{enumerate}[label=\rm{[\arabic{enumi}]}]
\item If $\mathbf{b}\in\sigma^\pm$, then
$[\pi_{n_\mathbf{b}}(\mathcal{X}_{n_\mathbf{b},p_\mathbf{b}}^{\sigma^\pm})] \LL^{-n_{\mathbf{b}}(q + 1) - p_{\mathbf{b}}}=
(\LL-1)^{q + 1}\LL^{-\langle\mathbf{b},\bnu\rangle}$.
\item If $\mathbf{b}\in\rho$, then
\[
[\pi_{n_\mathbf{b}}(\mathcal{X}_{n_\mathbf{b}, p_\mathbf{b}}^{\rho})]\LL^{-n_{\mathbf{b}}(q + 1) - p_{\mathbf{b}}} =
(\LL-1)^{q}(\LL-1 - 
e_q )\LL^{- \langle\mathbf{b},\bnu\rangle}.
\]
\item If $\mathbf{b}\in\rho$ and $i>0$, then
\[
[\pi_{n_\mathbf{b}+i}(\mathcal{X}_{n_\mathbf{b} + i, p_\mathbf{b}}^{\rho^\ast})] \LL^{-i(q+1)-n_{\mathbf{b}}(q + 1) - p_{\mathbf{b}}}=
e_q (\LL - 1)^{q+1}\LL^{-i -\langle\mathbf{b},\bnu\rangle}.
\]
\end{enumerate}
\end{lemma}

\begin{remark}\label{Wmot}
For 
$\Zmot{}(g,\omega, T)_\zero$ the computation is similar, see~\cite{gui:02,GG:14, bn:20}, and one obtains an expression 
\[
\Zmot{}(g, \omega, T)_\zero = \str^{\hat{\mu},\sigma^+}+ \str^{\hat{\mu},\sigma^-} + \str^{\hat{\mu},\rho} + \str^{\hat{\mu},\rho^\ast}, 
\]
where $\str^{\hat{\mu},\bullet}$ is defined analogously as in the naive case,
 and 
\begin{enumerate}[label=\rm{[\arabic{enumi}]$^{\hat{\mu}}$}]
\item if  $\mathbf{b} \in \sigma^+$ then   $[\pi_{n_\mathbf{b}}(\mathcal{X}^{1,\sigma^+}_{n_\mathbf{b}, p_\mathbf{b}})]^{\hat{\mu}} 
\LL^{-n_{\mathbf{b}}(q + 1) - p_{\mathbf{b}}}= [\mu_{\gcd(n_q, \exz+\ex)}]^{\hat{\mu}}(\LL-1)^{q}\LL^{- \langle\mathbf{b}, \bnu\rangle}$;
\item if  $\mathbf{b} \in \sigma^-$ then   $[\pi_{n_\mathbf{b}}(\mathcal{X}^{1,\sigma^-}_{n_\mathbf{b}, p_\mathbf{b}})]^{\hat{\mu}} 
\LL^{-n_{\mathbf{b}}(q + 1) - p_{\mathbf{b}}} = [\mu_{\exz+\ex}]^{\hat{\mu}}(\LL-1)^{q }
\LL^{- \langle\mathbf{b}, \bnu\rangle}$;
\item if  $\mathbf{b} \in \rho$  then 
\begin{align*}
[\pi_{n_\mathbf{b}}(\mathcal{X}^{1,\rho}_{n_\mathbf{b}, p_\mathbf{b}})]^{\hat{\mu}}\LL^{-n_{\mathbf{b}}(q + 1) - p_{\mathbf{b}}} &= 
[V(g_1)]^{\hat{\mu}}
\LL^{- \langle\mathbf{b}, \bnu\rangle}\\
[\pi_{n_\mathbf{b}+i}(\mathcal{X}^{1,\rho^\ast}_{n_\mathbf{b}+i, p_\mathbf{b}})]^{\hat{\mu}} 
\LL^{-i(q+1)-n_{\mathbf{b}}(q + 1) - p_{\mathbf{b}}}&
= e_q (\LL-1)^{q} \LL^{-i -\langle\mathbf{b}, \bnu\rangle}\text{ for } i>0,
\end{align*}
where $V(g_1)$ is the zero locus of $g_0-1$ in $(\CC^*)^{q+1}$ with a good action of $\hat{\mu}$. Note that after a toric change of variables it satisfies
\[
V(g_1)\cong(\CC^*)^{q-1}\times\{(y,z)\in(\CC^*)^2\mid z^\ex - y^{n_q}=1\}.
\]
\end{enumerate}
\end{remark}
\subsection{Rational form for the terms $\str^\bullet$}
\mbox{}

According to a celebrated result by Denef and Loeser~\cite[Theorem 2.2.1]{Denef-LoeserIgusa} or~\cite[Corollary~3.3.2]{Denef-LoeserBarca}, each of the previous partial zeta functions in~\eqref{eq:W} can be expressed as a rational function, i.e., an element in the $\mathcal{M}_\CC$-submodule of $\mathcal{M}_\CC[[T ]]$ 
generated by 1 and by finite products of terms of the form $\LL^j T^i (1- \LL^j
T^i)^{-1}$, with $j \in \ZZ$ and $i \in \ZZ_{>0}$. This will be accomplished in the next proposition by using the machinery of integer points on rational polyhedral cones 
from~$\S$\ref{ssec:generating_functions}. Beware the generating functions are factorized in order to simplify the Euler specialization in~\S\ref{sec:top_zeta_binomials}.

For  $\ba,\mathbf{b}\in\ZZ_{\geq 0}^{q+1}$, we denote  
$\LL^{-\ba} T^{\mathbf{b}}:=(\LL^{-{a}_1}T^{{b}_1},\ldots, \LL^{-{a}_{q+1}}T^{{b}_{q+1}})$.

\begin{prop}\label{prop:W}
Using the generating functions of the cones from {\rm\S\ref{ssec:generating_functions}}, the terms $\str^{\bullet}$ from \eqref{eq:W} are expressed as follows:
\begin{align*}
\str^{\sigma^+\!\!}{}&=
P_{\sigma^+}\left(\mathbf{t}=\LL^{-\bnu}T^{\mathbf{N}}\right)
K(\LL, T)
H(\LL, T),
\\
\str^{\sigma^-\!\!}{}&=
\tilde{K}(\LL, T)
\prod_{j=1}^q\frac{\LL-1}{1-\LL^{-\nu_j}}-P_{\sigma^+}(\mathbf{t}=\LL^{-\bnu}T^{(\exz  + \ex)\mathbf{e}_z}) 
\tilde{K}(\LL, T)
\tilde{H}(\LL, T)\\
&-(\LL-1)
P_{\rho}(\mathbf{t}=\LL^{-\bnu}T^{(\exz  + \ex)\mathbf{e}_z})
\tilde{H}(\LL, T),
\\
\str^{\rho}&=
(\LL - 1 - e_q)P_{\rho}\left(\mathbf{t}=
\LL^{-\bnu}T^{\mathbf{N}}\right)
H(\LL, T),
\\
\str^{\rho^\ast\!\!}{}&=
e_q\LL^{-1} T
 P_{\rho}\left(\mathbf{t}=
\LL^{-\bnu}T^{\mathbf{N}}\right)
\frac{\LL - 1}{1-\LL^{-1} T}
H(\LL, T),
\end{align*}
where
\[
H(\LL, T):=\prod_{j=1}^q\frac{\LL - 1}{P_j\left(\LL^{-\nu_j} T^{N_j}, \LL^{-\nu_z} T^{\exz}\right)},\quad
\tilde{H}(\LL, T):=\prod_{j=1}^q\frac{\LL-1}{P_j(\LL^{-\nu_j},\LL^{-\nu_z}T^{\exz  + \ex})},
\]
and
\[
K(\LL, T):=\frac{\LL - 1}{1-\LL^{-\nu_z} T^{\exz}},\qquad
\tilde{K}(\LL, T):=\frac{\LL - 1}{1-\LL^{-\nu_z} T^{\exz  + \ex}}.
\]
\end{prop}

\begin{proof}
We combine the definition of $\str^\bullet$, with Lemma~\ref{lema:b} and the generating functions.
Let us start with $\sigma^+$:
\[
\str^{\sigma^+} =  (\LL-1)^{q + 1}\sum_{ \mathbf{b} \in \sigma^+  \cap \NN^{q+1}} \LL^{-\langle\mathbf{b},\bnu\rangle}T^{\langle\mathbf{b},\mathbf{N}\rangle} = (\LL-1)^{q + 1}\Phi_{\sigma^+}(\mathbf{t}=\LL^{-\bnu}T^{\mathbf{N}}).
\]
Using the expression for $\Phi_{\sigma^+}$, we obtain the result in the statement.
For $\sigma^-$ we have
\begin{align*}
\str^{\sigma^-} = (\LL-1)^{q + 1} \sum_{ \mathbf{b} \in \sigma^-  \cap \NN^{q+1}} 
\LL^{- \langle\mathbf{b},\bnu\rangle}
T^{(\exz  + \ex)b_{z}}
=
(\LL-1)^{q + 1}\Phi_{\sigma^-}(\mathbf{t}=\LL^{-\bnu}T^{(\exz  + \ex)\mathbf{e}_z}).
\end{align*}
and the statement follows.
For $\rho$ we have:
\[
\str^\rho = (\LL-1)^{q}(\LL-1 - e_q) \sum_{ \mathbf{b} \in \rho  \cap \NN^{q+1}} \LL^{- \langle\mathbf{b},\bnu\rangle}T^{\langle\mathbf{b},\mathbf{N}\rangle}=
 (\LL-1)^{q}(\LL-1 - e_q) \Phi_{\rho}(\mathbf{t}=\LL^{-\bnu}T^{\mathbf{N}}).
\]
Finally, for $\rho^\ast$:
\[
\str^{\rho^\ast} = e_q(\LL - 1)^{q+1}\sum_{i=1}^{\infty}  \LL^{-i} T^i\cdot\sum_{ \mathbf{b} \in \rho  \cap \NN^{q+1}} 
 \LL^{-\langle\mathbf{b},\bnu\rangle} T^{\langle\mathbf{b},\mathbf{N}\rangle}\!=\!
e_q(\LL - 1)^{q+1}
\frac{\LL^{-1} T}{1-\LL^{-1} T}\Phi_{\rho}(\mathbf{t}=\LL^{-\bnu}T^{\mathbf{N}}),
\]
and the result follows.
\end{proof}

\begin{remark} For the summands of $\Zmot{}(g,\omega, T)_\zero$  one obtains 
\[
(\LL-1 )\str^{\hat{\mu},\sigma^+}=[\mu_{\gcd(N_q, \exz + \ex)}]^{\hat{\mu}}{\str}^{\sigma^+}, \quad (\LL-1 )\str^{\hat{\mu},\sigma^-}=[\mu_{\exz+\ex}]^{\hat{\mu}}{\str}^{\sigma^-}, \quad (\LL-1 )\str^{\hat{\mu},\rho^*}={\str}^{\rho^*}
\]
and
\[ 
\str^{\hat{\mu},\rho} =
(\LL -1)^{q - 1}[ \{(y,z)\in(\CC^*)^2\mid z^\ex - y^{n_q}=1\} ]^{\hat{\mu}} \cdot P_{\rho}\left(\mathbf{t}=
\LL^{-\bnu}T^{\mathbf{N}}\right)
H(\LL, T).
\]
\end{remark}
 \section{Topological Zeta functions for certain binomials}
\label{sec:top_zeta_binomials}

This section is devoted to compute the topological zeta function of the pair in~\eqref{eq:g}. Recall that the topological zeta function
is obtained from the naive zeta function by specializing with the  
Euler characteristic, see \eqref{eq:zetatop}. We define  $\str^{\bullet}_{\topo}$ as $\chi_{\topo} (\str^{\bullet}(T=\LL^{-s}))$
for the terms $\str^{\bullet}$ of \eqref{eq:W}, 
where $\bullet$ belongs to 
$\{\sigma^+, \sigma^-, \rho, \rho^\ast\}$.

\begin{lemma}\label{aportop} 
Consider $g(\mathbf{x},z)= z^\exz  (z^\ex + \mathbf{x}^{\mathbf{N}_q})$, and 
$\omega = \mathbf{x}^{\bnu} z^{\nu_z} \frac{\dif \mathbf{x}}{\mathbf{x}} \frac{\dif z}{z}$ as in \eqref{eq:g}.
The  terms   $\str^{\bullet}_{\topo}$ of the topological
zeta function $\Ztop{}(g,\omega,s)_\zero$ are expressed in terms of \begin{equation} \label{eq:rdef}
r :=\frac{(\exz+\ex)s+ \nu_z}{\ex}
\end{equation}
as follows:
\begin{align*}
\str^{\sigma^+}_{\topo}&=
\frac{1}{\nu_z + \exz s}
\prod_{j=1}^q 
\frac{1}{
N_j r + {\nu}_j},\\
\str^{\sigma^-}_{\topo}&=
\frac{1}{\ex r}
\left(\prod_{j=1}^q \frac{1}{ {\nu}_j} - 
\prod_{j=1}^q
\frac{1}{N_j r + {\nu}_j} \right),
\\
\str^{\rho}_{\topo}&=
\frac{-  e_q^2 }{\ex}
\prod_{j=1}^q \frac{1}{N_j r + {\nu}_j},
\\
\str^{\rho^\ast}_{\topo}&=
-\frac{1}{s+1}\str^{\rho}_{\topo}.
\end{align*}
\end{lemma}

\begin{proof} 
	From~\eqref{eq:identity_euler} and the definitions in Proposition~\ref{prop:W} it follows that 
\[
\chi(K(\LL, \LL^{-s}))=\frac{1}{\nu_z +\exz  s}=
\frac{\ex}{r - s},\quad\text{and}\quad
\chi(\tilde{K}(\LL, \LL^{-s}))=\frac{1}{\nu_z + (\exz  + \ex) s}=\frac{\ex}{r}.
\]
On the other side, the Euler characteristics of the terms $P_\bullet$ appearing in 
Proposition~\ref{prop:W} coincides with the multiplicity 
$\abs{D_\bullet}$
of the cone $\bullet \in\{\sigma^+, \sigma^-, \rho\}$.
In \S\ref{ssec:generating_functions} we describe the primitive 
integral vectors defining the rays 
of 
the 
$(q+1)$-dimensional simplicial cone $\sigma^+$ and 
the $q$-dimensional simplicial cone~$\rho$.
The multiplicity of $\sigma^+$ equals the absolute
value of the determinant of the matrix formed by these primitive integral vectors,
i.e.,
\begin{equation*}\ex^{q}\displaystyle\prod_{j=1}^q  k_j^{-1},
\end{equation*}
where $k_j=\gcd(\ex,N_j)$. 
The multiplicity of  $\rho$ is the $\gcd$ of the $q$-minors of the matrix of the primitive integral vectors defining the rays of $\rho$,
which turns out to be
\begin{equation*}\ex^{q-1} e_q\displaystyle\prod_{j=1}^q k_j^{-1}.
\end{equation*}
The Euler characteristic of the factor of $\tilde{H}(\LL, T)$ for $j\in \{1,\ldots,q\}$ is
\[
\chi\left(\frac{\LL-1}{1-\LL^{-\frac{(\ex\nu_j + \nu_z N_j)+s(\exz  + \ex) N_j}{k_j
}}}\right)=
\frac{k_j
}{(\ex\nu_j + \nu_z  N _j)+s(\exz  + \ex) N_j}=\frac{ 
k_j
}{\ex (N_jr + \nu_j)},
\]
and it coincides with the Euler characteristic of the corresponding factor in $H(\LL, T)$.
Combining all these facts the result is proved.
\end{proof}

Recall also that the local twisted topological zeta function
is obtained from the motivic one by specializing with the equivariant Euler characteristic, see~\eqref{eq:zetatop}. 
As before, we define  the terms $\str^{(\ell),\bullet}_{\topo}$ as $\chi_{\topo} ((\LL-1) \str^{\hat{\mu},\bullet}(T=\LL^{-s}), \alpha)$, with $\alpha$ a character of order $\ell$, 
for the terms $\str^{\hat{\mu},\bullet}$ introduced in Remark~\ref{Wmot}. 

For $\bullet\in\{\sigma^+, \sigma^-, \rho\}$ let us define 
\begin{equation} \label{eq:Nbullet}
\mathcal{N}(\bullet):=\gcd
\left\{ 
\min\left(\langle \ba , (\zero,\exz +\ex)\rangle , \left\langle \ba, \mathbf{N} 
\right\rangle  \right)
\middle| \,  \ba \in 	\bullet\cap\ZZ_{> 0}^{q+1} \right\}.
\end{equation}
Recall that if $\varphi \in \mathcal{A}_\bullet$ and $\ba = \ord \varphi$, then $\ord g \circ \varphi = \min\left(\langle \ba , (\zero,\exz +\ex)\rangle , \left\langle \ba, \mathbf{N} 
\right\rangle  \right)$, see~\S \ref{explicitW}. We will use $\mathcal{N}(\bullet)$ to describe
the equivariant Euler  specializations in the following proposition, whose proof will be given 
at the end of this section.

\begin{prop}\label{prop:Nbullet}
	The numbers $ \mathcal{N}(\bullet)$ are given by
\[
\mathcal{N}(\bullet) =
\begin{cases}
\gcd(n_q,\exz ) & \text{if } \bullet =  \sigma^+\\
\exz  + \ex& \text{if } \bullet= \sigma^-\\
\dfrac{(\exz  + \ex)n_q}{e_q} 
& \text{if } \bullet= \rho,
\end{cases}
\]
see \eqref{eq:nqeq} for the definition of $n_q$ and $e_q$.
\end{prop}

\begin{cor}\label{twistedaportop} Consider 
$g(\mathbf{x},z) $ 
and $\omega := \mathbf{x}^{\bnu} z^{\nu_z} \frac{\dif \mathbf{x}}{\mathbf{x}} \frac{\dif z}{z}$, i.e., as in \eqref{eq:g}.
For $\ell>1$
the  terms   $\str^{(\ell),\bullet}_{\topo}$ of the topological
zeta function $\Ztop{(\ell)}(g,\omega,s)_\zero$, are expressed 
as follows:
\[
\str^{(\ell),\bullet}_{\topo} =
\begin{cases}
\str^{\bullet}_{\topo} & \text{if } \ell \mid \mathcal{N}(\bullet) \\
0 & \text{otherwise, }
\end{cases} \text{ for } \bullet\in\{\sigma^+, \sigma^-, \rho\} 
\]
and  \[
\str^{(\ell),\rho^\ast}_{\topo} = 0
\]
\end{cor}
\begin{proof}
	The proof follows the same lines of the proof of Lemma~\ref{aportop}, taking into account the definition of $\str^{(\ell),\bullet}_{\topo}$, Proposition~\ref{prop:Nbullet},
and~\cite[\S5.3]{DL-JAMS}.   
\end{proof}

To finish this section with give the proof of a technical lemma.

\begin{proof} [Proof of Proposition{\rm~\ref{prop:Nbullet}}]
The general strategy for computing $\mathcal{N}(\bullet)$   is to use the combinatorial description \eqref{eq:partition} to find a divisor of $\mathcal{N}(\bullet)$ (see \eqref{eq:Nbullet}). Then, the values of $\min\left(\langle \ba , (\zero, \exz + \ex)\rangle , \left\langle \ba, \mathbf{N} \right\rangle  \right)$ for particular elements  $\ba  \in \bullet \cap \ZZ^{q+1}_{>0}$ are used to determine $\mathcal{N}(\bullet)$. Recall our notation: $\bN =(\bN_q,\exz )$ and $n_q = \gcd \bN_q$.

\begin{itemize}[wide, leftmargin=0pt]
    \item[Case $\bullet = \sigma^-$.] 
 By Definition \ref{eq:partition},  if $\ba \in 	\sigma^-\cap\ZZ_{\geq 0}^{q+1}$, then $\langle \ba,\mathbf{N}\rangle
	> (\exz  + \ex)a_{z} =\langle \ba , (\exz  + \ex)  \mathbf{e}_z \rangle$. Hence $\exz  + \ex$ divides $\mathcal{N}(\sigma^-)$. In order to check that $\mathcal{N}(\sigma^-) = \exz+\ex$, it is enough to  exhibit a vector 
\[
\ba \in \sigma^- \cap \ZZ^{q+1}_{>0},\quad \text{such that} \quad  \langle\ba,(\exz  + \ex)\mathbf{e}_z\rangle=\exz  + \ex.
\]
Recall the vectors $\mathbf{v}_i$ are the primitive integer vectors defining the rays of the closure of the cone $\rho$, see  \eqref{eq:v_prim}, and let us take  $\ba =\ba'+\ba''$ with 
\[
	\ba':=\frac{1}{\langle \mathbf{N}_q, \one \rangle } \sum_{i=1}^{q} k_i\mathbf{v}_i=
	\frac{\ex}{\langle \mathbf{N}_q, \one  \rangle }\sum_{i=1}^q\mathbf{e}_i + \mathbf{e}_z,\quad \text{and} \quad 
\ba'':=\left(1+\left\lceil\frac{\ex}{\langle \mathbf{N}_q, \one\rangle}\right\rceil-\frac{\ex}{\langle \mathbf{N}_q, \one\rangle  }\right)\sum_{i=1}^q \mathbf{e}_i,
	\]
    Let us check that 
    $\ba \in \sigma^- \cap \ZZ^{q+1}_{>0}$ using \eqref{eq:partition}. Indeed,
\[
\langle\ba, \mathbf{N} \rangle = \left(1+\left\lceil\frac{\ex}{\langle \mathbf{N}_q, \one\rangle}\right\rceil \right ) \langle \mathbf{N}_q, \one\rangle +\exz  > \exz  + \ex = \langle \ba , (\exz  + \ex)  \mathbf{e}_z \rangle.
\]

  \item[Case $\bullet = \sigma^+$.] By Definition \ref{eq:partition},  if $\ba \in 	\sigma^+\cap\ZZ_{\geq 0}^{q+1}$, then $\langle \ba,\mathbf{N}\rangle
	< (\exz  + \ex)a_{z} =\langle \ba , (\exz  + \ex)  \mathbf{e}_z \rangle$. Since 
	$\langle \ba,\mathbf{N}\rangle$ is a linear combination of the coordinates of $\mathbf{N}$ over $\ZZ_{>0}$, then  $\gcd\mathbf{N} = \gcd (n_q,\exz )$ divides  $\mathcal{N}(\sigma^+)$. Let us exhibit particular elements $\ba \in 	\sigma^+\cap\ZZ_{\geq 0}^{q+1}$ to show that $\mathcal{N}(\sigma^+)$ divides all coordinates of $\mathbf{N}$ and, therefore $\mathcal{N}(\sigma^+)= \gcd \mathbf{N}$. 
Let us first show that $\mathcal{N}(\sigma^+)$ divides the coordinates of $\mathbf{N}_q$.  Fix $i\in \{1, \dots, q\}$. For $M \gg 0$, consider the following two vectors 
	\[
	\ba_1=\mathbf{e}_i+N_i\sum_{\substack{j=1 \\ j \ne i}}^q \mathbf{e}_j + N_i M e_z\in\ZZ_{>0}^{q+1}, \quad 
\ba_2= \mathbf{e}_i + \ba_1.
\]
    	Note that $\ba_1,\ba_2\in\sigma^+ \cap \ZZ_{>0}^{q+1}$ because $M \gg 0$. Moreover,  we have  
	\[
\langle\ba_j,\mathbf{N}\rangle=N_i\left(j+\sum_{\substack{j=1,\\ j \ne i}}^q N_j +M \exz\right) \quad \text{for } j \in \{1,2\}. 
    \]
    Therefore, 	
\[
\gcd(\langle\ba_1,\mathbf{N}\rangle,\langle\ba_2,\mathbf{N}\rangle)=N_i,
\]
and 
 $\mathcal{N}(\sigma^+)$ divides $N_i$, for $i\in \{1, \dots, q\}$.
Let us now check that $\mathcal{N}(\sigma^+)$ divides $\exz $, which is the last coordinate of $\mathbf{N}$. 
We  exhibit vectors $\ba_M\in\sigma^+ \cap \ZZ_{>0}^{q+1}$ for $M \gg 0$ such that  $\langle\ba_M,\mathbf{N}\rangle=\exz(\langle \mathbf{N}_q, \one \rangle + M)$.  Since we have just shown that $\mathcal{N}(\sigma^+)$ divides $\exz\langle \mathbf{N}_q, \one \rangle$, it must divide  $\exz M$
	for $M\gg 0$ too. Hence,  it divides~$\exz$. Indeed, 
    \[
	\ba_M=\exz\sum_{j=1}^q \mathbf{e}_j + M e_z, \quad \text{ for }
	M\gg 0.
	\]
    Since $M \gg 0$ we have $\ba_M\in\sigma^+ \cap \ZZ_{>0}^{q+1}$, and 
	it is easy to check $\langle\ba_M,\mathbf{N}\rangle=\exz(\langle \mathbf{N}_q, \one \rangle + M)$
\item[Case $\bullet = \rho$.] 
By Definition \ref{eq:partition}, if $\ba\in\rho\cap\ZZ^{q+1}_{>0}$, then $\langle \ba,\mathbf{N}\rangle
	=\langle \ba , (\exz  + \ex)  \mathbf{e}_z \rangle$. Equivalently, we have 
     \[
	\sum_{j = 1 }^q a_j N_j + a_z \exz  =a_z(\exz  + \ex) \]
     and canceling $a_z\exz $ out, we obtain  
     \[ \sum_{j = 1 }^q a_j N_j = a_z k.\]
This implies that $\ex$ divides $\sum_{j = 1 }^q a_j N_j$. 
Since this sum is a linear combination of the coordinates of $\mathbf{N}_q$ over $\ZZ_{>0}$,   we get that it is also divided by~$n_q$.
It follows that $\lcm(n_q,\ex)$ divides $\sum_{j = 1 }^q a_j N_j$.
Hence
\[
\frac{\exz  + \ex}{\ex}\lcm(n_q,\ex) \text{ divides }
\frac{\exz  + \ex}{\ex}\sum_{j = 1 }^q a_j N_j = a_z (m+k)=
\langle \ba , (\exz  + \ex)  \mathbf{e}_z \rangle.
\]
This implies that
\[
\frac{\exz  + \ex}{\ex}\lcm(n_q,\ex) \text{ divides }
\mathcal{N}(\rho).
\]
Let us now use some special elements in $\rho \cap \mathbb{Z}_{>0}^{q+1}$ to show that $\mathcal{N}(\rho)$ divides $\lcm(\ex,N_i)\frac{\exz  + \ex}{\exz}$ for each $i \in \{1, \dots, q\}$. 
     For  each $i \in \{1, \dots, q\}$, consider the two integral vectors 
          \[
	\ba_h=\frac{\lcm(\ex,N_i)}{N_i}\left(h\mathbf{e}_i+N_i\sum_{j=1, \dots q}^{j \ne i} \mathbf{e}_j\right)+ 
\frac{\lcm(\ex,N_i)}{\ex}\left(h+\sum_{j=1, \dots,q}^{ j \ne i} N_j\right) e_z \qquad \text{with } h= 1, 2.
	\]
Notice that $\ba_h \in\rho\cap\ZZ^{q+1}$ for $h=1,2$, since 
\[
     \langle \ba_h, \mathbf{N} \rangle = \langle \ba_h, (\exz+\ex) \mathbf{e}_z \rangle
= {\lcm (\ex, N_i)} \frac{\exz+\ex}{\ex }   \left ( h + \sum_{\substack{j=1, \dots, q}}^{j \ne i} N_j \right )
    .
     \]
     and $\gcd ( \langle \ba_1, \mathbf{N} \rangle, \langle \ba_2, \mathbf{N} \rangle ) = 	\lcm (\ex,N_i) \frac{\exz  + \ex}{\ex }$. Since
     \[
 \gcd( \lcm  (\ex,N_1),\dots,\lcm  (\ex,N_1)) 
 = 
\lcm ( \gcd(N_1,\dots,N_q), \ex)
     = 
     \lcm (n_q, \ex)
     \]
     it follows that $\mathcal{N}(\rho)$ divides
     $ \lcm (n_q, \ex) \frac{\exz+\ex}{\ex }$.
     From these two divisibility conditions, we conclude $
     \mathcal{N}(\rho) = \lcm (n_q, \ex) \frac{\exz+\ex}{\ex}$. Since 
     $e_q = \gcd(n_k, \ex)$ 
     we can reformulate the previous equality as
     $
     \mathcal{N}(\rho) = 
\frac{(\exz+\ex)n_q}{e_q}$. \qedhere
\end{itemize}
\end{proof}
 \section{Topological Zeta functions of generalized suspensions
}
\label{sec:MainResults}

We are going to use the result of the previous sections to compute the topological zeta function of~$G$
with respect to the form $\omega_{d+1}$. As it was stated in the Introduction we are mostly interested
in the case where $\nu_i=1$ for $i\in\{1,\dots,d\}$ and $\nu_z=d + 1$. Since 
for general monomial forms the arguments follow the same guidelines, we state Theorem~\ref{thmpsuspACNLM}
in a more general setting. We introduce first some results about arithmetic functions that will be used later on.

\subsection{Arithmetic functions and other prerequisites}\label{ssec:arithmetic_functions}
\mbox{}

For a detailed discussion of the content of this section see e.g.~\cite[Section 3.7]{SaCr04}. Denote the set of positive integers by $\NN$. Consider the commutative ring of arithmetic functions  $(A, +, \ast)$ of functions $f: \NN \rightarrow \NN$ such that $f(1) \ne 0$ with the operations of pointwise addition $+$ and the Dirichlet convolution $\ast$ that is given in terms of the usual product $\cdot$ of the integers  by 
\[
(f \ast g)(n)= \sum_{d \mid n}  f(d) \cdot  g\left(\frac{n}{d}\right).
\]  
There are some simple examples of arithmetic functions:

\begin{itemize}
	\item The constant function
	$\one: \NN \rightarrow \NN$ given  by  $\one(n)=1$ for any positive integer $n$.
	
	\item The 
	\emph{unit} function $\epsilon  : \NN \rightarrow \NN$  given by $\epsilon (1) = 1$ and $\epsilon(n) =0 $ if $n \ne 1$. 
	Note that $\epsilon \ast f = f \ast \epsilon =f$ for any $f \in A$.
	
	\item For any positive integer $m$, the functions $\sigma_m   : \NN \rightarrow \NN$, given by $\sigma_m(n) = n^m$ for any positive integer $n$.
	
	\item The \emph{Möbius function}  $\mu  : \NN \rightarrow \NN$ is the function that, given a positive integer $n$, 
	$\mu(n)$ is the sum of the primitive $n$th roots of unity. It is well-defined since 
	the cyclotomic polynomials are monic polynomials over the integers.
	
	\item The \emph{Euler's totient function} $\phi  : \NN \rightarrow \NN$ is the function that, given a positive integer $n$,  
	counts the number of
	coprime residues modulo $n$. It satisfies the famous Gauss' identity 
	\begin{equation}\label{eq:Gauss}
		\sum_{d \mid n} \phi(d)= n.
	\end{equation}
	
	\item For any positive integer $m$, the  \emph{Jordan's totient function}  $J_m : \NN \rightarrow \NN$ is the function that,  given a positive integer $n$, counts the $m$-tuples of positive integers all less than or equal to $n$
	that form a coprime $(m+1)$-tuple together with $n$. 
	The Jordan's totient functions are a generalization of  Euler's totient function since $J_1 = \phi$. 
	
\end{itemize}

\begin{remark} The values of  $J_m$  are given by the expression
	\[
	J_m(n) = n^m \prod_{\substack{p \mid n\\p\text{ prime}}} 
	\left(1 - \frac{1}{p^m}\right),
	\]
	Some values of $J_2$ are listed in OEIS~\cite{oeis007434}. 
\end{remark}

The following lemma will be useful to give explicit expressions for the Denef-Loeser zeta functions of a suspension.

\begin{lemma} 
	The Jordan's totient functions $J_m$ and the functions $\sigma_m$ are related under Dirichlet convolution by the following expressions 
	\[
	\sigma_m = \one \ast J_m, \text{ and }
	J_m = \mu \ast \sigma_m,
	\]
	or, equivalently, by the expressions
	\begin{equation}\label{jk}
		n^m = \sum_{d \mid n} J_m(d),  \text{ and }
		J_m(n) = \sum_{d \mid n} \mu(d) \left(\frac{n}{d}\right)^m.
	\end{equation}
\end{lemma}

We also need the following arithmetic lemma, that is proven at the end of this section
(page~\pageref{subsec:tp5}, for the computations in Theorem~\ref{thmfsuspACNLM}.

\begin{lemma}\label{lema:arit}
    Let $\ex, \ell, q \in \mathbb{Z}_{> 0}$. Then, 
	the set
	\[
	D(\ex , \ell, q):=\left\{ M\in\NN\ \middle|\
	\ell \gcd\left(\ex , M\right) \text{\rm divides } q M\right\}
	\]
    is closed under taking multiples and greatest common divisors.
    In particular, it is the set of multiples of a natural number 
    $\mkl(\ex , \ell, q)$. Moreover if $\ell_1:=\frac{\ell}{\gcd (\ell, q )}$, then
    \[
	   \mkl(\ex,\ell, q)=\ell_1\max_{n\in\mathbb{N}}\gcd(\ex, \ell_1^n)= \ell_1 \gcd(\ex, \ell_1^u),\text{ for } u\gg 1.
	\]
\end{lemma}

\subsection{Sum over strata} 
\mbox{}

We finished the strategy outlined in \S \ref{sec:motivic_zeta_susp}. The final step consists in the computation of the terms $\Ztop{\bullet} (G,\omega_{d+1},s)_\zero$, with $\bullet\in\{\sigma^+, \sigma^-, \rho, \rho^\ast\}$, 
from~\eqref{eq:stratum_top}
by adding up the contributions of  each stratum $E_I^\circ$, $\emptyset \ne I \subset J$, in the embedded resolution of $f$.  We use the conventions  and $\bN_q=\bN_I$, $n_q = n_I$ and $e_q=e_I = \gcd(\ex, n_I)$ in Lemma~\ref{aportop}, for each strata $E_I^\circ$.
From the discussion in \S \ref{sec:motivic_zeta_susp}, we adapt the results
in~\S\ref{sec:motivic_zeta_binomials} from the variables $x_1,\dots,x_q$ 
to the variables~$\bx_I$.

\begin{lemma}\label{lemma:Z_cones} The  terms   $\Ztop{\bullet}(G,\omega_{d+1},s)_\zero$ of the topological
zeta function $\Ztop{}(G,\omega_{d+1},s)_\zero$ are expressed in terms of $r$ (see \eqref{eq:rdef})
as follows:
\begin{align}
\label{eq:sigma+}
\Ztop{\sigma^+}(G,\omega_{d+1},s)_\zero &=
    \frac{\Ztop{}(f,\omega_{d},r)_\zero}{\nu_z +\exz s}. \\
\label{eq:sigma-}
   \Ztop{\sigma^-}(G,\omega_{d+1},s)_\zero &=\frac{1}{\ex r\prod\bnu_0}
    -\frac{\Ztop{}(f,\omega_{d},r)_\zero}{\ex r}.\\
\label{eq:rho}
   \Ztop{\rho}(G,\omega_{d+1},s)_\zero &=
    \frac{-1}{\ex}\sum_{e \mid \ex }  J_2(e) \Ztop{(e)}(f, \omega_d, r)_0.\\
\label{eq:rho*}
 \Ztop{\rho^*}(G,\omega_{d+1},s)_\zero &=
    \frac{1}{(s + 1)\ex}\sum_{e \mid \ex }  J_2(e) \Ztop{(e)}(f, \omega_d, r)_0,
    \end{align}
    where 
\[
\prod{\bnu^0}:=\prod_{j=1}^d \nu_j^0.
\]
\end{lemma}

\begin{proof}
From expressions  \eqref{eq:stratum_top} and  \eqref{eq:stratum_top_bis}  we need to add all contributions 
$\chi(E_I^\circ) \str^\bullet_{\topo,I,\zero}$ for each $\bullet\in\{\sigma^+, \sigma^-, \rho, \rho^\ast\}$. 
Recall that  $\str^\bullet_{\topo,I,\zero}$ depends only on the stratum $E^\circ_I$ and not on the choice of a point in $E^\circ_I$, see the discussion before Lemma~\ref{lema:globalW}.  The terms $\str^\bullet_{\topo,I,\zero}$ were computed in Lemma~\ref{aportop} where the subscripts $I$ and $\zero$ were omitted and, moreover,  $e_q=e_I$ and $n_q=n_I$ where $e_I = \gcd( \ex, n_I)$.

For $\bullet = \sigma^+$ the statement follows from the expression \eqref{eq:zetas_res_top} for the local topological zeta function in terms of the embedded resolution of $f$:
\begin{equation*}
    \Ztop{\sigma^+}(G,\omega_{d+1},s)_\zero = 
    \frac{1}{\nu_z +\exz s}
    \sum_{\emptyset \ne I \subset J}
    \frac{\chi(E_I^\circ)}{\displaystyle \prod_{i \in I} (N_i r + {\nu}_i)}=
    \frac{\Ztop{}(f,\omega_{d},r)_\zero}{\nu_z +\exz s}.
\end{equation*}

For $\bullet = \sigma^-$, the statement follows from \eqref{eq:zetas_res_top}, as above, and  \eqref{DLz(0)=1omega}:
\begin{equation*}
    \Ztop{\sigma^-}(G,\omega_{d+1},s)_\zero = 
    \frac{1}{\ex r}
    \left(
    \sum_{\emptyset \ne I \subset J}
    \frac{\chi(E_I^\circ)  }{\displaystyle \prod_{i \in I} {\nu}_i} -
    \sum_{\emptyset \ne I \subset J}
    \frac{\chi(E_I^\circ)}{\displaystyle \prod_{i \in I} (N_i r + {\nu}_i)}
\right)=\frac{1}{\ex r\prod\bnu^0} -
    \frac{\Ztop{}(f,\omega_{d},r)_\zero}{\ex r}.
\end{equation*}

In order to ease the notation,  we set
\begin{equation}\label{eq:terms_str}
    \estrato{I}:=\frac{\chi(E_I^\circ)}{\displaystyle \prod_{i \in I} (N_i r + {\nu}_i)} \quad  \text{for} \quad \emptyset \ne I \subset J.
\end{equation}

For $\bullet = \rho$, the statement follows by applying property~\eqref{jk} of the arithmetic functions $J_2$ and $\sigma_2$, and expression \eqref{eq:zetas_res_top} for the local twisted topological zeta function in terms of an embedded resolution of $f$:
\begin{align*}
    \nonumber
    \Ztop{\rho}(G,\omega_{d+1},s)_\zero &= \frac{-1}{\ex}
    \sum_{\emptyset \ne I \subset J} e_I^2
    \estrato{I}
=\frac{-1}{\ex}\sum_{p \mid \ex }  p^2 
    \sum_{\substack{\emptyset \ne I \subset J\\ e_I=p }}     \estrato{I}
    =\frac{-1}{\ex}
    \sum_{p \mid \ex }  \sum_{e \mid p } J_2(e) \!\!\!
    \sum_{\substack{\emptyset \ne I \subset J\\ e_I = p}}     \estrato{I}
    \\
    \nonumber
    &
    = \frac{-1}{\ex}\sum_{e \mid \ex }  J_2(e) \!  \sum_{e \mid p \mid \ex  } 
    \sum_{\substack{\emptyset \ne I \subset J\\ e_I = p }}     \estrato{I}
    =\frac{-1}{\ex}\sum_{e \mid \ex }  J_2(e)
    \sum_{\substack{\emptyset \ne I \subset J\\ e \mid n_I }}     \estrato{I}
    \\
    &
    =\frac{-1}{\ex}\sum_{e \mid \ex }  J_2(e) \Ztop{(e)}(f, \omega_d, r)_0.
\end{align*}
Finally, for $\bullet = \rho^*$, the statement follows similarly to the previous case.
\end{proof}

\begin{lemma}\label{lemma:Z_cones-tw} Take $\ell>1$. The  terms   $Z^{(\ell),\bullet}_{\topo}(G,\omega_{d+1},T)_\zero$ of the twisted topological
zeta function $\Ztop{(\ell)}(G,\omega_{d+1},s)_\zero$ are expressed in terms of $r$ (see \eqref{eq:rdef})
as follows:
\begin{enumerate}[label=\rm(\roman{enumi})]
\item\label{lemma:Z_cones-tw1}  $\Ztop{(\ell),\sigma^+}(G,\omega_{d+1},s)_\zero =
\frac{\Ztop{(\ell)}(f,\omega_{d},r)_\zero}{\nu_z +\exz s}$,
 if $\ell\mid  m$ and $0$ otherwise.
\item\label{lemma:Z_cones-tw2}  $\Ztop{(\ell),\sigma^-}(G,\omega_{d+1},s)_\zero =
\frac{1}{\ex r\prod\bnu^0} -
    \frac{\Ztop{}(f,\omega_{d},r)_\zero}{\ex r}$,
     if $\ell\mid \exz  + \ex$ and $0$ otherwise.
\item \label{lemma:Z_cones-tw3}  $\Ztop{(\ell),\rho}(G,\omega_{d+1},s)_\zero =
    \frac{-1}{\ex}\sum_{e \mid \ex } J_2(e)
\Ztop{(\lcm(e, \mkl(\ex,\ell,\exz  + \ex)))}(f,\omega_d,r)$, see
 Lemma{\rm~\ref{lema:arit}} for the notation.
 \item $\Ztop{(\ell),\rho^*}(G,\omega_{d+1},s)_\zero =0$.
\end{enumerate}
\end{lemma}

\begin{proof} As in the proof of 
Lemma~\ref{lemma:Z_cones} we follow \eqref{eq:stratum_top} and \eqref{eq:stratum_top_bis} and proceed to add all contributions 
$\chi(\tilde{E}_I^\circ, \alpha) \str^{(\ell),\bullet}_{\topo,I,\zero}$ for each $\bullet\in\{\sigma^+, \sigma^-, \rho, \rho^\ast\}$, where $\alpha$ is a character of order~$\ell$. We follow the conventions given there about the subscripts $I$ and $\zero$ and about the constants $e_q=e_I$ and $n_q=n_I$.
Notice that  only some strata contribute to each case $\bullet \in \{ \sigma^+, \rho \}$. The contributing strata $E_I^\circ$ are determined by some properties of $n_q$ and $e_q$ given in 
Corollary~\ref{twistedaportop}.
For these contributing strata, $\chi(\tilde{E}_I^\circ, \alpha)=\chi(E_I^\circ)$, see~\cite[Thm. 2.2.1 and Prop. 2.3.1]{Denef-LoeserIgusa}.

For $\bullet = \sigma^+$ the contributing strata are those such that 
$\ell\mid\gcd(n_I,\exz )$. If $\ell\nmid \exz $, there is no contribution and if 
$\ell\mid \exz $ the condition is equivalent to $\ell\mid n_I$. As a consequence of expression~\eqref{eq:zetas_res_top} for the local twisted topological zeta function in terms of an embedded resolution of $f$, we have
\begin{equation*}
    \Ztop{\sigma^+,(\ell)}(G,\omega_{d+1},s)_\zero=
    \begin{cases}
        \frac{\Ztop{(\ell)}(f,\omega_{d},r)_\zero}{\nu_z +\exz s}, & \text{ if }\ell\mid \exz \\
        0 & \text{ if }\ell\nmid \exz .
    \end{cases}
\end{equation*}
For $\bullet= \sigma^-$  all strata contribute. 
The statement follows again from expression \eqref{eq:zetas_res_top} for the twisted local topological zeta function in terms of an embedded resolution of $f$:
\begin{equation*}
    \Ztop{\sigma^-,(\ell)}(G,\omega_{d+1},s)_\zero=
    \begin{cases}
        \dfrac{1}{\ex r\prod\bnu^0} - \dfrac{\Ztop{}(f,\omega_{d},r)_\zero}{\ex r} & \text{ if }\ell\mid \exz +  \ex,\\
        0 & \text{ if }\ell\nmid \exz +  \ex.
    \end{cases}
\end{equation*}
For $\bullet= \rho$ the only contributing strata are those such that $\ell \mid \frac{(\exz  + \ex) n_I}{e_I}$ due to Corollary~\ref{twistedaportop}. The condition 
$\ell \mid \frac{(\exz+ \ex) n_I}{e_I}$ reformulates 
in the form $\mkl(\ex,\ell,\exz  + \ex) \mid n_I$ by 
Lemma~\ref{lema:arit}.
The statement follows by applying property \eqref{jk} of the arithmetic functions $J_2$ and $\sigma_2$ and expression \eqref{eq:zetas_res_top} for the local twisted topological zeta function. 

\begin{align*}
\nonumber
 \Ztop{\rho,(\ell)}(G,\omega_{d+1},s)_\zero&=\frac{-1}{\ex}\!\!
 \sum_{\substack{\emptyset \ne I \subset J\\ \ell \mid \frac{(\exz  + \ex) n_I}{e_I}  }} \!\!\!\! e_I^2
\estrato{I}
= \frac{-1}{\ex}\sum_{j | \ex } j^2 \!\!\!\!\!\!\!\!\sum_{\substack {  \emptyset \ne I \subset J\\  j=\gcd(n_I, \ex)  \\ \mkl(\ex,\ell,\exz  + \ex) \mid n_I} } \!\!\!\!\!\!\estrato{I}
=\frac{-1}{\ex}\sum_{j \mid  \ex }  \sum_{e \mid j | \ex } J_2(e)
\!\!\!\!\!\!\!\!
\sum_{\substack {  \emptyset \ne I \subset J\\  j=\gcd(n_I, \ex)  \\ \mkl(\ex,\ell,\exz  + \ex) \mid n_I}} \!\!\!\!  
\estrato{I}
\\ 
& =    \frac{-1}{\ex}\sum_{e \mid \ex } J_2(e) \sum_{ e \mid j \mid  \ex } 
\!\!\!\!\!\!
\sum_{\substack {  \emptyset \ne I \subset J\\  j= \gcd(n_I, \ex) \\ \mkl(\ex,\ell,\exz  + \ex) \mid n_I} } 
\!\!\!\!\!\!\!\!\estrato{I}
=\frac{-1}{\ex}\sum_{e \mid \ex } J_2(e)\!\!\!\!\!\!\!\! \sum_{\substack {\emptyset \ne I \subset J\\  
e\mid n_I\\
\mkl(\ex,\ell,\exz  + \ex) \mid n_I
} } 
\!\!\!\!\!\!\!\!\estrato{I}
\\ 
\nonumber
&=
\frac{-1}{\ex}\sum_{e \mid \ex } J_2(e) 
\!\!\!\!\!\!\!\!\!\!\!\!\!\!
\sum_{\substack {\emptyset \ne I \subset J\\  
\lcm(e, \mkl(\ex,\ell,\exz  + \ex))\mid n_I} } 
\!\!\!\!\!\!\!\!\!\!\!\!\!\!\estrato{I}
=\frac{-1}{\ex}\sum_{e \mid \ex } J_2(e)
\Ztop{(\lcm(e, \mkl(\ex,\ell,\exz  + \ex)))}(f,\omega_d,r).\qedhere
\end{align*}
    For $\bullet= \rho^*$,  it is  clear that there are no contributing strata due to Corollary~\ref{twistedaportop}. Hence,  $\Ztop{\rho^*,(\ell)}(G,\omega_{d+1},s)_\zero=0$.
\end{proof}

\subsection{Description of the topological zeta functions} 
\mbox{}

\begin{theorem}\label{thmpsuspACNLM}
Consider $G(\mathbf{x},z)= z^\exz  (z^\ex + f(\mathbf{x}))$, and 
$\omega_{d+1} = \mathbf{x}^{{\bnu^0}} z^{\nu_z} \frac{\dif \mathbf{x}}{\mathbf{x}} \frac{\dif z}{z}$.
Recall that $r =\frac{(\exz+\ex)s+ \nu_z}{\ex}$, see~\eqref{eq:rdef}. 
Then the following explicit expressions for the local topological zeta functions 
hold: 
\begin{enumerate}[label=\rm(Z\arabic{enumi})]
\item For $\ell=1$ we have
\begin{equation}
\label{fsuspACNLMth-p}  
\Ztop{}(G,\omega_{d+1},s)_\zero\!=\!
\frac{1}{\ex r\prod{\bnu^0}}\!+\! 
\frac{s (s \!-\! r \!+\! 1) ( r\! +\! 1)}{\ex r (r - s) (s \!+\! 1)}\Ztop{}(f, \omega_d, r)_\zero\!
	-\! \frac{s}{s\!+\!1} \!\!\sum_{1\neq e | \ex } \!\frac{J_2(e)}{\ex } \Ztop{(e)}(f, \omega_d, r)_\zero.
\end{equation}

\item For $1<\ell\mid\gcd(\ex, \exz  + \ex)=\gcd(\ex, \exz)$
we have
\begin{equation}
	\label{twisted1fsuspACNLMth-p-p+k} 
\Ztop{(\ell)}(G,\omega_{d+1},s)_\zero \!= \!
\frac{1}{\ex r\prod{\bnu^0}} + 
\frac{\Ztop{(\ell)}(f,\omega_{d},r)_\zero}{\ex (r - s)}-
\frac{(r \!+ \!1)\Ztop{}(f,\omega_{d},r)_\zero}{\ex r}-\!
\sum_{1\neq e \mid \ex }\! \frac{J_2(e)}{\ex}
\Ztop{(e)}(f,\omega_d,r)_\zero.
\end{equation}

\item For $\ell\mid \exz  + \ex$, $\ell\nmid \exz $ we have
\begin{equation}
	\label{twisted1fsuspACNLMth-nop-p+k} 
\Ztop{(\ell)}(G,\omega_{d+1},s)_\zero= 
\frac{1}{\ex r\prod\bnu^0} -
\frac{(r + 1)\Ztop{}(f,\omega_{d},r)_\zero}{\ex r}-
\sum_{1\neq e \mid \ex } \frac{J_2(e)}{\ex}
\Ztop{(e)}(f,\omega_d,r)_\zero.
\end{equation}

\item For $\ell\nmid \exz  + \ex$ and $\ell\mid \exz $  we have
\begin{equation}
	\Ztop{(\ell)}(G,\omega_{d+1},s)_\zero= 
\frac{\Ztop{(\ell)}(f,\omega_{d},r)}{\ex (r - s)}- \! \sum_{e | \ex } \frac{J_2(e)}{\ex } 
	\Ztop{(\lcm(e,\mkl(\ex,\ell, \exz  + \ex)))}(f,\omega_d,r)_\zero.
\label{twisted2fsuspACNLMth-p}
\end{equation}

\item For $\ell\nmid \exz  + \ex$ and $\ell\nmid \exz $  we have
\begin{equation}
	\Ztop{(\ell)}(G,\omega_{d+1},s)_\zero= 
	- \! \sum_{e | \ex } \frac{J_2(e)}{\ex } 
	\Ztop{(\lcm(e, \mkl(\ex,\ell, \exz  + \ex)))}(f,\omega_d,r)_\zero.
\label{twisted2fsuspACNLMth-p1}
\end{equation}

\end{enumerate}

\end{theorem}

\begin{proof} The statements follow from the combination of 
\eqref{eq:stratum_top}, \eqref{eq:stratum_top_bis} and Lemmata~\ref{lemma:Z_cones} and~\ref{lemma:Z_cones-tw}. 

For $\ell =1$ we add all the four terms from Lemma~\ref{lemma:Z_cones}. Equality \eqref{fsuspACNLMth-p} follows from two facts. First, the coefficients $\frac{1}{\nu_z +\exz s}$ of $\Ztop{}(f,\omega_{d},r)_\zero$ in \eqref{eq:sigma+} and $-\frac{1}{\ex r}$ of $\Ztop{}(f,\omega_{d},r)_\zero$ in \eqref{eq:sigma-}  add up to $\frac{s}{r (\nu_z +\exz s)}$.  Secondly,  the terms corresponding to $e=1$ in \eqref{eq:rho} and \eqref{eq:rho*} add up to  $\frac{s}{s+1}\frac{1}{\ex }\Ztop{}(f,\omega_{d},r)_\zero$.  Taking into account that $(r-s)k=\nu_z+\exz s$, we get the coefficient of $\Ztop{}(f,\omega_{d},r)_\zero$ in~\eqref{fsuspACNLMth-p} as $\frac{s}{r (\nu_z +\exz s)}-\frac{s}{(s+1)k}$.

For $\ell \mid \exz  + \ex$ and $\ell \mid \exz $ we add~\ref{lemma:Z_cones-tw1},~\ref{lemma:Z_cones-tw2}, and~\ref{lemma:Z_cones-tw3} in Lemma~\ref{lemma:Z_cones-tw}.  
Equality~\eqref{twisted1fsuspACNLMth-p-p+k} follows by adding up  the coefficients  $-\frac{1}{\ex r}$ of $\Ztop{}(f,\omega_{d},r)_\zero$ in~\ref{lemma:Z_cones-tw2} with the coefficient $-\frac{1}{\ex}$  of the term corresponding to $e=1$ in~\ref{lemma:Z_cones-tw3} to get  $-\frac{(r+1)}{\ex r}\Ztop{}(f,\omega_{d},r)_\zero$.

For $\ell \mid \exz  + \ex$ and $\ell \nmid \exz $ we add~\ref{lemma:Z_cones-tw2} and~\ref{lemma:Z_cones-tw3} in Lemma~\ref{lemma:Z_cones-tw} to get~\eqref{twisted1fsuspACNLMth-nop-p+k}.

For $\ell \nmid \exz  + \ex$ and $\ell \mid \exz $ we add~\ref{lemma:Z_cones-tw1} and~\ref{lemma:Z_cones-tw3} in Lemma~\ref{lemma:Z_cones-tw} to get \eqref{twisted2fsuspACNLMth-p}.

For $\ell \nmid \exz  + \ex$ and $\ell \nmid \exz $ we just consider~\ref{lemma:Z_cones-tw3} in Lemma~\ref{lemma:Z_cones-tw} that equals to  \eqref{twisted2fsuspACNLMth-p1}.
\end{proof}

In order to describe the set of poles of  $\Ztop{}(G,\omega_{d+1},s)_\zero$ we introduce the following notation.
\begin{notation}\label{ntc:set_poles}
For  a germ $H$ at~$\zero$ and a monomial volume form $\omega$, denote by 
$\pol(H,\omega)$ the set of absolute values of the poles of $\Ztop{}(H,\omega,s)$. Recall that in general the poles of $\Ztop{}(H,\omega,s)$ are negative rational numbers, see~\eqref{eq:zetas_res_top}. When $\omega$ is the standard volume form we simply write $\pol(H)$.
\end{notation}

\begin{notation}\label{ntc:polos}
    Given $\rho_0\in \QQ$; $\nu,\exz\in\ZZ$; and $\ex\in\NN$, with $\exz + \ex\neq 0$, we set 
    \[
    \cpolo(\rho_0,\nu,\exz,\ex):=\frac{\ex \rho_0 + \nu}{\exz + \ex }.
    \]
\end{notation}

\begin{remark}\label{rem:n+1m}
    Note that
    if $\rho_0=\frac{\nu}{\exz}$, then 
    $\cpolo\left(\frac{\nu}{\exz},\nu,\exz,\ex\right)=\frac{\nu}{\exz}=\rho_0$.
\end{remark}

\begin{cor}\label{cor:poles_G}
The set $\pol(G, \omega_{d+1})$
is contained in 
\[
\left\{
1, \frac{\nu_z}{\exz}
\right\}
\cup
\left\{
\cpolo(\rho_0,\nu_z,\exz,\ex)
\middle|\,
\rho_0\in\pol(f,\omega_d)
\right\}.
\]
\end{cor}

\begin{proof}
Examination of the formula~\eqref{fsuspACNLMth-p} for $\Ztop{}(G, \omega, s)_\zero$ in Theorem~\ref{thmpsuspACNLM} reduces the set of negatives of candidate poles of $\Ztop{}(F,\omega,s)_\zero$ to the set    
\[
\left\{
1, \frac{\nu_z}{\exz+\ex}, \frac{\nu_z}{\exz}
\right\}
\cup
\left\{
\cpolo(\rho_0,\nu_z,\exz,\ex)
\middle|\,
\rho_0\in\pol(f,\omega_d)
\right\}.
\]
Let us show that $s_0:=-\frac{\nu_z}{\exz + \ex }$ is a false candidate pole if $\frac{\nu_z}{\exz + \ex } \neq 1$. Notice that this candidate corresponds to $r=0$. In order to check if this is an actual pole, we can multiply by $\ex r$ and then evaluate $r=0$.
We obtain
\[
\frac{1}{\prod{\bnu^0}} - \Ztop{}(f, \omega_d, 0)_\zero,
\]
which vanishes
thanks to~\eqref{DLz(0)=1omega}, and 
can discard the candidate $-\frac{\nu_z}{\exz + \ex }$.
\end{proof}

Now we take $\exz=0$ and $\nu_z=1$, switching this way from $G$ to $F$.
Theorem~\ref{thmfsuspACNLM} below contains formulas which express the local topological zeta function of the suspension $F$  of the hypersurface $f$ by $\ex $ points in terms of the local topological zeta function and local twisted topological zeta functions of $f$ and $\ex $. 
\begin{theorem} \label{thmfsuspACNLM}
Consider 
$F(\mathbf{x},z)=  z^\ex + f(\mathbf{x})$ and 
$\omega_{d+1} = \mathbf{x}^{{\bnu^0}} z \frac{\dif \mathbf{x}}{\mathbf{x}} \frac{ \dif z}{z}$. 
The following explicit expressions for the local and the local $\ell$-twisted 
Denef-Loeser topological zeta functions  of the suspension hold (where $t=s+\frac{1}{\ex}$).
\begin{enumerate}[label=\rm(\alph{enumi})]
\item For $\ell=1$ we have
\begin{equation}
\label{fsuspACNLMth}  
\frac{s+1}{s}\Ztop{}(F,\omega_{d+1},s)_\zero=
\frac{s+1}{\ex st\prod{\bnu^0}}+\frac{\ex -1}{\ex }\frac{t + 1}{t}
\Ztop{}(f,\omega_d,t)_\zero
	- \sum_{1\neq e | \ex } \frac{J_2(e)}{\ex } \Ztop{(e)}(f,\omega_d,t)_\zero.
\end{equation}

\item For $1\neq \ell \mid \ex $ we have
\begin{equation}
	\label{twisted1fsuspACNLMth} 
\Ztop{(\ell)}(F,\omega_{d+1},s)_\zero=  \! \frac{1}{\ex t\prod\bnu^0}+
\Ztop{(\ell)}(f,\omega_d,t)_\zero - \! \frac{t\!+\!1}{\ex t}\Ztop{}(f,\omega_d,t)_\zero
	-  \!\sum_{ 1\neq e \mid \ex } \!\!\frac{J_2(e)}{\ex } \Ztop{(e)}(f,\omega_d,t)_\zero.
\end{equation}

\item For $\ell\nmid \ex $ we have
\begin{equation}
	\Ztop{(\ell)}(F,\omega_{d+1},s)_\zero= 
\Ztop{(\ell)}(f,\omega_d,t)_\zero
	- \! \sum_{e | \ex } \frac{J_2(e)}{\ex } 
	\Ztop{(\lcm(e, \mkl(\ex,\ell, \ex )))}(f,\omega_d,t)_\zero.
\label{twisted2fsuspACNLMth}
\end{equation}

\end{enumerate}

\end{theorem}

\begin{remark}Note that formula \eqref{fsuspACNLM} at the introduction is equivalent to \eqref{fsuspACNLMth}.\end{remark}

\begin{proof} Statement \eqref{fsuspACNLMth} is a consequence of \eqref{fsuspACNLMth-p} in Theorem~\ref{thmpsuspACNLM}. Taking $\exz=0$ the polynomial $G$ becomes $F$. Moreover taking $\exz=0$ and $\nu_z=1$ the parameter $r$ becomes $t$ and the term $\frac{s (s - r + 1) ( r + 1)}{\ex r (r - s) (s + 1)}$, coming from the coefficient of $\Ztop{}(f,\omega_{d},r)_\zero$ in~\eqref{fsuspACNLMth-p}, 
becomes $1 -  \frac{1}{\ex t} - \frac{s}{s+1} \frac{1}{\ex} = \frac{s}{s+1} \frac{\ex-1}{\ex}\frac{t+1}{t}$. Finally, multiplying by $\frac{s+1}{s}$ we get \eqref{fsuspACNLMth}. See Remark~\ref{matrix} for the motivation of the factor $\frac{s+1}{s}$.

In the same way if $1\neq \ell$ divides $\ex$, then \eqref{twisted1fsuspACNLMth} is a consequence of \eqref{twisted1fsuspACNLMth-p-p+k}
and if $\ell\nmid k$, then \eqref{twisted2fsuspACNLMth} is a consequence of  \eqref{twisted2fsuspACNLMth-p}.
\end{proof}

The following is a direct consequence of Corollary~\ref{cor:poles_G}.
\begin{cor}The set $\pol(F)$ is contained in 
\[
\left\{
1
\right\}
\cup
\left\{
\rho_0 + \frac{1}{\ex}\,
\middle| \,
\rho_0\in\pol(f)
\right\}.
\]
\end{cor}

\begin{remark}\label{matrix} Equations \eqref{fsuspACNLMth} and \eqref{twisted1fsuspACNLMth} from Theorem~\ref{thmfsuspACNLM} can be presented in a compact way using matrices. To see this consider a total order in the finite set $\mathcal{D}(\ex )$ of divisors of $\ex$ such that 1 is the smallest element, e.g.,  the usual order $\leq$ in $\NN$. Denote by $\ell_i$ the $i$-th element of $\mathcal{D}(\ex )$ under this order. We take the following $\abs{\mathcal{D}(\ex)}$-vectors:
\begin{itemize}
\item $ZF(s)$  whose first entry is   $\frac{s+1}{s}\Ztop{} (F,\omega_{d+1},s)$  and the other entries are  $\Ztop{(\ell_i)}(F,\omega_{d+1},s)$.
\item $Zf(t)$  whose first entry is   $\frac{t+1}{t}\Ztop{} (f,\omega_{d},t)$  and the other entries are  $\Ztop{(\ell_i)}(f,\omega_{d},t)$.
\item $A  = \frac{1}{\prod{\bnu^0}}(\frac{s+1}{s}, 1, 1, \dots, 1)^t$.
\item $J_2$ whose $i$th entry is $J_2(\ell_i)$.
\end{itemize}
Finally, set $B$ for the square matrix of size $\abs{\mathcal{D}(\ex)}$ defined by $\ex  \cdot\mathbf{Id}_{\abs{\mathcal{D}(\ex)}} - J $, where $\mathbf{Id}_{\abs{\mathcal{D}(\ex)}}$ is the identity matrix, and the rows of $J$ are all equal to the vector $J_2$. 
Then, the statement of Theorem~\ref{thmfsuspACNLM} becomes:
\[
\frac{1}{\ex }ZF(s)  =  \frac{1}{t}A + B  \cdot Zf(t).
\]
\end{remark}

\begin{example}\label{ex:5-6-10}
Consider $f= x^5 + y^6$ and $\ex =10$. We have:
\[
\Ztop{(\ell)}(f,s) = 
\begin{cases}
\frac{ 10 s + 11}{(30 s + 11)(s +1)}, & \ell=1\\
 \frac{4}{30 s + 11},  &\ell=2, 3, 6\\
\frac{5}{30s + 11}, &\ell=5\\
\frac{ -1}{30s + 11}, &\ell=10, 15, 30\\
\end{cases}
\quad \text{ and } \quad
\Ztop{(\ell)}(F,s) = 
\begin{cases}
\frac{3s + 7}{(15 s +7) (s +1 )},& \ell=1\\
\frac{6}{15s + 7} & \ell=3,6\\
\frac{1}{2(15s + 7)}, &\ell=5\\
\frac{-5}{2(15s + 7)}, &\ell=10\\
\frac{7}{2(15s + 7)}, & \ell=15, 30\\
\end{cases}
\]
For other values of $\ell$, the zeta functions vanish.
Therefore, for $t = s+ \frac{1}{10}$, we have that 
\[
10
\begin{pmatrix}
\frac{s+1}{s} \Ztop{}(F,s)_0 \\
\Ztop{(2)}(F,s)_0 \\
\Ztop{(5)}(F,s)_0 \\
\Ztop{(10)}(F,s)_0
\end{pmatrix} 
=
\frac{1}{t}
\begin{pmatrix}
\frac{s+1}{s}\\
1\\
1\\
1
\end{pmatrix} 
+
\begin{pmatrix}
9&-3&-24&-72\\
-1&7&-24&-72\\
-1&-3&-14&-72\\
-1&-3&-24&-62\\
\end{pmatrix} 
\begin{pmatrix}
\frac{t+1}{t} \Ztop{}(f,t)_0 \\
\Ztop{(2)}(f, t)_0 \\
\Ztop{(5)}(f, t)_0 \\
\Ztop{(10)}(f, t)_0
\end{pmatrix}. 
\]
Let us see what happens for the values $\ell = 3,6,15,30$ which do not divide $\ex = 10$.
Notice that $\ell_1 = \frac{\ell}{\gcd (\ell, 10)} =3$ always, and then
$3 = \mkl(10,\ell, 10) $ (see Lemma \ref{lema:arit}). Hence, by \eqref{twisted2fsuspACNLMth}
\[
\begin{aligned}
\Ztop{(\ell)}(F,s)_\zero &= \Ztop{(\ell)}(f,t)_\zero - \frac{1}{10} \Ztop{(3)}(f,t)_\zero -   \frac{3}{10} \Ztop{(6)}(f,t)_\zero  - \frac{24}{10} \Ztop{(15)}(f,t)_\zero - \frac{72}{10} \Ztop{(30)}(f,t)_\zero  \\
&= \Ztop{(\ell)}(f,t)_\zero - \frac{4}{10} \Ztop{(3)}(f,t)_\zero - \frac{96}{10} \Ztop{(15)}(f,t)_\zero . 
\end{aligned}
\]
\end{example}

\begin{example} \label{ex:LvP}
Consider the polynomial $f=x^{12} + y^{13} + z^{14} + xy^7 + y^2z^4 + x^2y^2z^3$ from~\cite[Example 2]{MR2806692} and $\ex =84$. 
The Newton polygon of $f$ has a compact face determined by the monomials $xy^7$, $y^4z^4$, and $x^2y^2z^4$ which moreover is contained in the affine hyperplane with equation $5p+7q+10r= 54$. For $\ex =84$ and $\ell=27 \nmid \ex =84$, we have that $\ell_1= 9$, and $m=3$. For $e \mid 84$ we have
\[
\lcm (e, 27) = 
\begin{cases}
27e& \text { if } 3 \nmid e\\
9e& \text{ otherwise},
\end{cases}
\]
and
\[  
\Ztop{(\lcm (e, 27) )}(f,\omega, s)_\zero = 
\begin{cases}
\frac{1}{2 (27s + 11)}& \text { if }  e = 1,2,3,6\\
0& \text{ otherwise}.
\end{cases}
\]
Therefore, we have 
\[
\begin{aligned}
\Ztop{(27)}(F,s) &=  \Ztop{(27)}(f,t) - \sum_{e \mid 84} \frac{J_2(e)}{84} \Ztop{(\lcm(27,e))}(f,t) = \Ztop{(27)}(f,t) - \sum_{e \mid  6} \frac{J_2(e)}{84} \Ztop{(\lcm(27,e))}(f,t)
\\
&=  \Ztop{(27)}(f,t) - \frac{1}{84} \Ztop{(27)}(f,t)  -  
\frac{3}{84} \Ztop{(54)}(f,t) - \frac{8}{84} \Ztop{(27)}(f,t) -  
\frac{24}{84} \Ztop{(54)}(f,t)\\
 &= \frac{12}{21} \Ztop{(27)}f,t)  = \frac{8}{756s + 317}.
 \end{aligned}
 \]
\end{example}

To finish this section with give the proof of a technical lemma. 

\begin{proof}[Proof of Lemma{\rm~\ref{lema:arit}}] 
\phantomsection\label{subsec:tp5}
First, let us check that 
\[
	D(\ex , \ell, q)=\left\{ M\in\NN\ \middle|\
	\ell \gcd\left(\ex , M\right) \text{\rm divides } q M\right\}
\]
is the set of all multiples of a positive integer $\mkl(\ex, \ell, q )$, i.e., it is of the form $\mathbb{N} \mkl(\ex. \ell, q)$.
Take $M_1 \in \NN$, put $d = \gcd (k, M_1)$, and write
\[
k = \tilde{k} d, \quad M_1 = \tilde{M}_1 d, \mbox{ with } \gcd( \tilde{k}, \tilde{M}_1) =1.
\]
Notice that $M_1 \in D(\ex , \ell, q)$ if and only if $\ell d$ divides $q M_1$.
Assume that $M_1 \in D(\ex , \ell, q)$ and  $M_2 = a M_1$, for some $a \in \NN$. Let us check that $M_2 \in D(\ex , \ell, q)$. 
Notice that
\[
 \gcd(k, M_2) = d \gcd (\tilde{k}, \tilde{M}_1 a ) = d_1 \gcd(\tilde{k}, a).
\]
Then, 
$
\ell  \gcd(k, M_2) = \ell d \gcd (\tilde{k}, a) $ divides $\ell d a$. 
Since $M_1 \in D(\ex , \ell, q)$, we obtain that 
$
\ell \gcd(k, M_2)$ divides $q  M_1 a = q M_2$,
hence $M_2 \in  D(\ex , \ell, q)$.

Assume that $M_1, M_2 \in D(\ex, \ell, q)$. Then, we have 
\[\ell \gcd(\ex, \gcd(M_1, M_2)) \mid  \ell \gcd(\ex, M_i)  \mid q M_i \qquad \text{ for } i=1,2. \]
Hence, $\gcd(M_1, M_2)  \in D(\ex, \ell, q)$.

Let us now describe the number $\mkl(\ex, \ell, q)$. Note that the condition $M \in D(\ex, \ell, q)$ can be reduced to 
$\ell_1 \gcd (\ex,M) \mid M$, by canceling out $\gcd(\ell, q)\gcd(\ex, M)$ to get
\[
\ell_1 := \frac{\ell}{\gcd (\ell, q)} \mid \frac{q}{\gcd (\ell, q)} \frac{M}{\gcd (\ex, M),} 
\]
and noticing that $\ell_1$ and $\frac{q}{\gcd (\ell, q)}$   are coprime. Notice that the condition $\mkl(\ex, \ell, q) \in D(\ex, \ell, q)$ implies that  $\ell_1$ divides $\mkl(\ex, \ell, q)$,  and denote by $\mkl'(\ex, \ell, q) :=\frac{\mkl(\ex, \ell, q)}{\ell_1}$. It is enough to describe $\mkl'(\ex, \ell, q)$.

If $M \in \NN \mkl(\ex, \ell, q)$, we write $M = \ell_1  M'$. Notice that
\begin{equation*}
    \mathbb{N}{\mkl}'(\ex, \ell, q)=
    \left\{ M'\in\mathbb{N}  \,  \middle  | \,  \gcd(\ex, \ell_1 {M}')
    \text{ divides } M'\right\}.
\end{equation*}
If $p$ is a prime number we denote the corresponding $p$-valuation by~$\nu_p$. The condition ${\mkl}'(\ex, \ell, q) \in  \mathbb{N}{\mkl}'(\ex, \ell, q)$ is equivalent to the condition that  for any prime~$p$ we have  
\[
\min(\nu_p(\ex), \nu_p(\ell_1) + \nu_p({\mkl}'(\ex, \ell, q)))\leq \nu_p({\mkl}'(\ex, \ell, q)).
\]
Equivalently, 
\begin{equation*}
    \nu_p({\mkl}'(\ex, \ell, q))=
    \min\{r\in\mathbb{N}\mid \min(\nu_p(\ex), \nu_p(\ell_1) + r)\leq r\}.
\end{equation*}
If $\nu_p(\ell_1)=0$, this minimal value is~$0$. Otherwise,   $\nu_p(\ell_1) \ne 0$, and
the minimal value is $\nu_p(\ex)$. Then 
\[\mkl'(\ex, \ell, q) = \prod_{\substack{ p \, \text{prime} \\ p \mid \gcd (\ex, \ell_1 ) }} p^{\nu_p(\ex)} = \displaystyle \max_{n\in\mathbb{N}}\gcd(\ex, \ell_1^n),
\]
and $\mkl(\ex, \ell, q)= \displaystyle\ell_1\max_{n\in\mathbb{N}}\gcd(\ex, \ell_1^n).$
\end{proof}
 
\section{The holomorphy conjecture for suspensions}
\label{sec:hc_kLYS}

The holomorphy conjecture was stated by Denef~\cite[Conjecture 4.4.2]{denefBourbaki} for the $p$-adic Igusa zeta function, see also~\cite{Denef93}, and it was generalized 
for the topological zeta functions in~\cite[\S3.3]{veys:curves_hol}.

Let $F:(\CC^{n+1},\zero)\to (\CC,\zero)$ be the defining equation of a germ of isolated singularity of  hypersurface. Given a point $x\in\{F=0\}$, the \emph{Milnor fibration of $F$ at $x$} is the locally trivial 
$\mathcal{C}^\infty$-fibration defined by $F_|:B_\varepsilon\cap F^{-1}(\mathbb{D}^\ast_\eta)\longrightarrow \mathbb{D}^\ast_\eta$, where $B_\varepsilon$
is the open ball of radius $\varepsilon$ centered at $x$, $\mathbb{D}_\eta= \{z \in \CC\ ;\  |z| < \eta\}$,  and $\mathbb{D}^\ast_\eta$ is the open punctured disk. As usual, we assume $0<\eta \ll \varepsilon$ for a small enough $\varepsilon$, and we term any fiber $\mathcal{F}_{F,x}$ of this fibration \emph{the Milnor fiber} of $F$ at $x$. It turns out that a small loop around $0\in \mathbb{D}_\eta$ induces a diffeomorphism of $\mathcal{F}_{F,x}$ that is well defined up to isotopy, such 
diffeomorphism $\Psi: \mathcal{F}_{F,x}\longrightarrow \mathcal{F}_{F,x}$ is called the \emph{monodromy transformation}. The  algebraic monodromy action of $F$ at $x$ is the corresponding linear transformation 
$\Psi_\ast : H_\ast (\mathcal{F}_{F,x},\CC)\longrightarrow H_\ast (\mathcal{F}_{F,x},\CC)$ 
on the homology groups.

\begin{notation} We denote by $\mathcal{E}_F$ the set of eigenvalues of the monodromy on $\tilde{H}_{*} (\mathcal{F}_{F, x},\CC)$, $x\in F^{-1}(0)$.
Since we are going to deal with 
the orders of these eigenvalues, we set
$\mathcal{E}^{\ord}_F:=
\{\ord\zeta\mid\zeta\in\mathcal{E}_F\}$.
For a subset $\mathcal{E}$ of natural numbers
$\overline{\mathcal{E}}$ denotes the union of the
sets of divisors of elements in $\mathcal{E}$.

\end{notation}

\begin{hol_conj}\label{hol_conj}
If $\ell \notin\overline{\mathcal{E}_F^{\ord}}$ and $\ell>1$ then $\Ztop{(\ell)}(F,s)_\zero$  is holomorphic on $\CC$
and, actually, vanishes.
\end{hol_conj}
This conjecture has been proved for curves by Veys, see~\cite{veys:curves_hol}, and for non degenerate surfaces by Castryck, Ibadula and Lemahieu, see~\cite{zbMATH07034647}. We refer to the latter for a recent account of the status of the conjecture. 

Denef and Veys also proved the holomorphy conjecture   for
$F=z^\ex +f$ and $k >2$ conditionally on the holomorphy conjecture for $f$, see~\cite[\S0.4]{dv:95} and Proposition~\ref{prop:dv}. Although they work with Igusa zeta functions, 
 their proof is still valid for topological zeta functions.
Using Proposition~\ref{prop:dv_ts} (taken from~\cite{dv:95}) and our notation, their strategy of proof can be described as follows.  

\begin{strategy}\label{strategy}
For $\ell, \ex \in \ZZ_{\geq 1}$, let us denote $\mathcal{D}_{\ell,\ex}$ the set of 
$\lambda \in \ZZ_{\geq 1}$ such that $\lambda\mid\ell$ and $\frac{\ell}{\lambda}\mid\ex$.
The choice of this subset is justified by Proposition~\ref{prop:dv_ts}, since the existence of poles
of $\Ztop{(\ell)}(F,s)_\zero$ implies the existence of poles of $\Ztop{(\lambda)}(f,s)_\zero$
for some $\lambda\in\mathcal{D}_{\ell,\ex}$. 

\begin{description} 
\item[Step 1] Fix  $\ell \notin\overline{\mathcal{E}_F^{\ord}}$ and prove that $\lambda\notin\overline{\mathcal{E}^{\ord}_f}$ for all $\lambda \in \mathcal{D}_{\ell, \ex}$.

\item[Step 2] Applying the holomorphy conjecture for~$f$, deduce that $\Ztop{(\lambda)}(f,s)_\zero=0$ for all $\lambda \in \mathcal{D}_{\ell, \ex}$.

\item[Step 3] From the above comment, $\Ztop{(\ell)}(F,s)_\zero=0$.
\end{description}
\end{strategy}

We study suspensions by two points of plane curve singularities. In this situation 
Strategy~\ref{strategy} may fail since one can find $\ell$ and $\lambda$ such that 
$\ell \notin\overline{\mathcal{E}_F^{\ord}}$, $\lambda \in \mathcal{D}_{\ell, \ex}$, and
$\lambda\in\overline{\mathcal{E}^{\ord}_f}$. In Lemma~\ref{lema:2l} we give formulas
for $\Ztop{(\ell)}(f+z^2,s)_\zero$ in terms of some twisted topological zeta functions
for~$f$.
Theorem~\ref{thm:holconfsuspcurve} shows that the holomorphy conjecture holds for $F= f + z^2$, when $f$ defines a plane curve germ.

Let us recall the cases for which the holomorphy conjecture
has been established for suspensions. First, we  recall the 
 notion of bad eigenvalue, introduced by Denef and Veys in~\cite{dv:95}.

\begin{defn}\label{def:DVbad}
An element  $\zeta \in \mathcal{E}_f$ of order $d$ is called  \emph{bad} if 
$d\equiv 2 \bmod{4}$,
$2d \notin\overline{\mathcal{E}^{\ord}_f}$,
and $\frac{d}{2}\notin{\mathcal{E}^{\ord}_f}$.
\end{defn}

\begin{remark}
    The third condition in Definition~\ref{def:DVbad} is not given explicitly in~\cite{dv:95}. Note that if 
    $\frac{d}{2}\in{\mathcal{E}^{\ord}_f}$, then $d\in{\mathcal{E}}^{\ord}_{f + z^2}$, see~\eqref{eq:casesorder2}. Therefore, the value of $\Ztop{(d)}(f + z^2, s)_\zero$ is irrelevant for the Holomorphy Conjecture.
\end{remark}

The following statement is the main result  in~\cite{dv:95}.

\begin{prop}[{\cite[Theorem 3.1 and Remark 3.3]{dv:95}}]
\label{prop:dv}
    Assume the  holomorphic conjecture holds for the singularity defined by the germ $f$.    Then the holomorphy conjecture holds for $f+z^\ex$ if one
    of the following conditions is satisfied:
    \begin{enumerate}[label=\rm(\arabic{enumi})]
        \item\label{prop:dv1} $\ex>2$;
        \item\label{prop:dv2} $\ex = 2$ and $f$ has no bad eigenvalue.
    \end{enumerate}
\end{prop}

In order to prove Proposition~\ref{prop:dv}, Denef and Veys characterize the poles
of the twisted zeta Igusa functions for Thom-Sebastiani sums. In the setting of  topological zeta functions and for the case of suspensions, their result reads as follows.

\begin{prop}[{\cite[Corollary 2.8]{dv:95}}]\label{prop:dv_ts}
For $\ell\in \ZZ_{\geq 1}$ and $\ell_1$ a divisor of $\ell$
such that $\frac{\ell}{\ell_1}$ divides~$\ex$,  if $-s_1\notin\ZZ$ is a pole of $\Ztop{(\ell)}(f+z^\ex,s)_\zero$, then
$-(s_1-\frac{1}{\ex})$ is a pole of $\Ztop{(\ell_1)}(f,s)_\zero$.
\end{prop}

It is worth mentioning that Proposition~\ref{prop:dv_ts} can be deduced from our 
Theorem~\ref{thmfsuspACNLM}. The latter is actually more general and it is used below to study the holomorphy conjecture for $f+ z^2$ with $f\in \CC\{x,y\}$. Lemma~\ref{lema:2l} expresses the formulas in Theorem~\ref{thmfsuspACNLM}
for $\ex=2$.

Let us begin recalling some properties of the set of eigenvalues of the monodromy of suspensions, particularly   in the case of $\ex=2$, which is the one not completely covered by~\cite{dv:95}.

The following general relation between the sets $\mathcal{E}_f$ and 
$\mathcal{E}_{f+z^\ex}$ is consequence of the main result of the classical Thom-Sebastiani theory~\cite{ST-71}:
\begin{equation*}\mathcal{E}_{f+z^\ex}=\{\eta\cdot\zeta\mid \eta^\ex=1, \eta\neq 1,\zeta\in\mathcal{E}_f\}.
\end{equation*}
When $\ex>2$, we have $\overline{\mathcal{E}_f^{\ord}}\subset\overline{\mathcal{E}_{f + z^\ex}^{\ord}}$ and
Proposition~\ref{prop:dv}\ref{prop:dv1} follows from this inclusion and Proposition~\ref{prop:dv_ts}.  In the case $\ex = 2$ we have the following relations between the orders of the eigenvalues of $f$ and the suspension $f+z^2$:
\begin{equation}\label{eq:casesorder2}
\begin{cases}
d\in\mathcal{E}_f^{\ord} \iff d\in\mathcal{E}_{f + z^2}^{\ord} &\text{ for }d\equiv 0 \bmod{4},\\
d\in\mathcal{E}_f^{\ord} \iff 2d\in\mathcal{E}_{f + z^2}^{\ord} &\text{ for }d\equiv 1 \bmod{2},\\
d\in\mathcal{E}_f^{\ord} \iff \frac{d}{2}\in\mathcal{E}_{f + z^2}^{\ord} &\text{ for }d\equiv 2 \bmod{4}.
\end{cases}
\end{equation}
In the next lemma, we state that the inclusion $\overline{\mathcal{E}_f^{\ord}}\subset\overline{\mathcal{E}_{f + z^2}^{\ord}}$ holds only under the additional hypothesis that $f$ has no bad eigenvalues, and in consequence the same hypothesis is needed for Proposition~\ref{prop:dv}\ref{prop:dv2}.
\begin{lemma}
    If $f$ has no bad eigenvalue, then 
    $\overline{\mathcal{E}_f^{\ord}} \subset\overline{\mathcal{E}_{f + z^2}^{\ord}}$.
\end{lemma}

\begin{proof}
It is enough to show that $\mathcal{E}_f^{\ord}\subset\overline{\mathcal{E}_{f + z^2}^{\ord}}$.
We fix $\ell\in\mathcal{E}_f^{\ord}$ and consider three possible cases.

\begin{caso} $\ell\equiv 0\bmod{4}$. 
\end{caso}
By \eqref{eq:casesorder2}, it is obvious that $\ell\in\mathcal{E}_{f+z^2}^{\ord}\subset\overline{\mathcal{E}_{f + z^2}^{\ord}}$.

\begin{caso} $\ell\equiv 1\bmod{2}$. 
\end{caso}
By \eqref{eq:casesorder2}, we have $2\ell\in\mathcal{E}_{f+z^2}^{\ord}$ and then $\ell\in\overline{\mathcal{E}_{f + z^2}^{\ord}}$.

\begin{caso} $\ell\equiv 2\bmod{4}$. 
\end{caso}

By \eqref{eq:casesorder2}, we have $\frac{\ell}{2}\in\mathcal{E}_{f+z^2}^{\ord}$.
Let $\zeta\in\mathcal{E}_f$ such that $\ord\zeta=\ell$. Since $\zeta$ is not bad,
we have that $2\ell\in\overline{\mathcal{E}_{f}^{\ord}}$ or $\frac{\ell}{2} \in\mathcal{E}_{f}$.

If $2\ell\in\overline{\mathcal{E}_{f}^{\ord}}$,
let $d\in\mathcal{E}_{f}^{\ord}$
such that $2\ell\mid d$. Since $d\equiv 0\bmod{4}$, then 
$d\in\mathcal{E}_{f+z^2}^{\ord}$ and $\ell\in\overline{\mathcal{E}_{f + z^2}^{\ord}}$.

If $\frac{\ell}{2} \in\mathcal{E}_{f}$, then 
$\frac{\ell}{2}$ is odd hence $\ell \in \mathcal{E}_{f+z^2}^{\ord}$, by \eqref{eq:casesorder2}.
\end{proof}

Let us study the bad eigenvalues of $f$. Actually, we are more interested in their orders,
and more precisely in some divisors of their orders, which justifies the following definition.

\begin{defn}\label{def:bad_div}
    We say that $\ell\in\ZZ_{>0}$ is $f$-bad if
    \begin{enumerate}[label=(B\arabic{enumi})]
    \item\label{b1} $\ell\equiv 2\bmod{4}$;
    \item\label{b2} $\ell\in\overline{\mathcal{E}_f^{\ord}}$;
    \item\label{b3} $2\ell\notin\overline{\mathcal{E}_f^{\ord}}$;
    \item\label{b4} $\frac{\ell}{2}\notin\overline{\left\{d\in\mathcal{E}_f^{\ord}\mid d \text{ is  odd}\right\}}$.
    \end{enumerate}
\end{defn}

\begin{remark}\label{Rmk:f_bad}
Note that if $\zeta$ is a bad eigenvalue of $f$, then $\ord\zeta$ is $f$-bad.
Moreover, if $\ell$ is $f$-bad, then there is a bad eigenvalue $\zeta$
such that $\frac{\ord\zeta}{\ell}$ is an odd integer, because of~\ref{b3}. In addition, $2$ is $f$-bad if and only if the set $\mathcal{E}_f^{\ord}$ contains only numbers which are congruent with $2\bmod{4}$.
\end{remark}

\begin{examples}\label{ex:f-bad}
Let us compute the characteristic polynomials and the bad orders in 
two examples.
\begin{enumerate}[label=\rm(\arabic{enumi})]
    \item 
    Let $f(x,y)=(y^2 - x^3)^3 - x^6 y^3$. The singular point defined by $f$ consist of three equisingular branches (ordinary cusps with equal pairwise intersection) and its resolution is attained after 4 point blow-ups, see Figure~\ref{fig:bad}.
    Using A'Campo's formula in Remark~\ref{rem:acampo}, the characteristic polynomial
    of the monodromy of the germ defined by $f$ is
    \[
    \Delta(f,t)=\frac{(t-1)(t^{18}-1)(t^{21}-1)^2}{(t^6-1)(t^9-1)}=
    \frac{(t - 1)(t^{9} + 1)(t^{21} - 1)^2}{t^6 - 1} = \Phi_1^2 \Phi_3 \Phi_7^2 \Phi_{18} \Phi_{21}^2,
    \]
    because the multiplicities of the exceptional divisors are $N_1=6, N_2=9, N_3=18$, and $N_4=21$. 
    It follows that $\mathcal{E}^{\ord}_f= \{1,3,7,18,21\}$ and $\overline{\mathcal{E}^{\ord}_f}=\overline{\{18,21\}}$. In particular, the eigenvalue $\exp\frac{\pi i}{9}$ of order~$18$ is the unique bad eigenvalue, and, hence, 18 is $f$-bad. Moreover, \eqref{eq:casesorder2} implies that $\mathcal{E}^{\ord}_{f+z^2}= \{2,6,9,14,42\}$ and $\overline{\mathcal{E}^{\ord}_{f+z^2}}=\overline{\{9,42\}}$. Notice that $\overline{\mathcal{E}_f^{\ord}} \setminus \overline{\mathcal{E}_{f+z^2}^{\ord}}=\{18\}$ is the set of $f$-bad orders. This is a general phenomena described in the next result.  
\item Let $g(x,y)=(y^2 - x^3)^2 - x^4 y^2$. The singular point defined by $g$ consist of just two of the equisingular branches of $f$.
The characteristic polynomial
    of the monodromy of the germ defined by $g$ is
    \[
    \Delta(g,t)=\frac{(t-1)(t^{12}-1)(t^{14}-1)}{(t^4-1)(t^6-1)}
    =\frac{(t - 1)(t^{6} + 1)(t^{14} - 1)}{t^4 - 1}=\Phi_1\Phi_7\Phi_{12}\Phi_{14},
    \] 
    because the multiplicities of the exceptional divisors are $N_1=4, N_2=6, N_3=12$, and $N_4=14$. 
    The eigenvalue $\exp\frac{\pi i}{7}$ of order~$14$ is not bad since $\exp\frac{2 \pi i}{7}$ is an eigenvalue of order 7. 
\begin{figure}[ht]
    \centering
    \begin{tikzpicture}
        \draw (-1.5, 0) -- (0, 0) -- (1.5, 0) (0, 0) -- (0, 1.5);
        \foreach \x in {(-1.5, 0), (0, 0), (1.5, 0), (0, 0), (0, 1.5)}
        {
        \fill \x circle[radius=.1cm];
        }
        \foreach \x in {(-1, 2.5), (0, 2.5), (1, 2.5)}
        {
        \draw[->] (0, 1.5) -- \x; 
        }
        \node[below] at (-1.5, 0) {$-3$};
        \node[above] at (-1.5, 0) {$E_1$};
        \node[below right] at (0, 0) {$-2$};
        \node[below left] at (0, 0) {$E_3$};
        \node[below] at (1.5, 0) {$-2$};
        \node[above] at (1.5, 0) {$E_2$};
        \node[below left] at (0, 1.5) {$-1$};
          \node[below right] at (0, 1.5) {$E_4$};
    \end{tikzpicture}
    \caption{Resolution graph of $(y^2 - x^3)^3 - x^6 y^3=0$.}
    \label{fig:bad}
\end{figure}
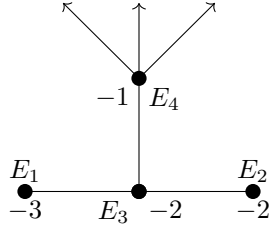
\end{enumerate}

\end{examples}

\begin{lemma}\label{lem:set-f-bad}
    The set of $f$-bad numbers is $\overline{\mathcal{E}_f^{\ord}} \setminus\overline{\mathcal{E}_{f+z^2}^{\ord}}$.
\end{lemma}

\begin{proof}
    Denote by $\mathcal{B}_f \subset \mathbb{Z}$ the set of $f$-bad numbers associated with $f$.  Let us start with the inclusion $\mathcal{B}_f \subset \overline{\mathcal{E}_f^{\ord}} \setminus\overline{\mathcal{E}_{f+z^2}^{\ord}}$. Let $\ell$ be $f$-bad. From~\ref{b2}, we have that 
$\ell\in\overline{\mathcal{E}_f^{\ord}}$. Let us prove by contradiction
    that $\ell\notin\overline{\mathcal{E}_{f+z^2}^{\ord}}$. If
$\ell\in\overline{\mathcal{E}_{f+z^2}^{\ord}}$, let $d\in\mathcal{E}_{f+z^2}^{\ord}$
    such that $\ell\mid d$. Note that $d$ is even.

    If $d\equiv 0\bmod{4}$, then $2\ell\mid d$ and $d\in\mathcal{E}_{f}^{\ord}$. Hence 
$2\ell\in\overline{\mathcal{E}_f^{\ord}} $ which contradicts~\ref{b3}.
    The case $d\equiv 2\bmod{4}$ implies $\frac{d}{2}\in\mathcal{E}_{f}^{\ord}$ which contradicts~\ref{b4}.

    Let us prove now $\mathcal{B}_f \supset \overline{\mathcal{E}_f^{\ord}} \setminus\overline{\mathcal{E}_{f+z^2}^{\ord}}$. Let $\ell\in\overline{\mathcal{E}_f^{\ord}}\setminus\overline{\mathcal{E}_{f+z^2}^{\ord}}$.
    The condition~\ref{b2} is obvious. let $d\in\mathcal{E}_{f}^{\ord}$
    such that $\ell\mid d$.
    
    Let us prove first that $\ell$ cannot be odd. We distinguish again the possible cases.

    If $d\equiv 0\bmod{4}$, then $d\in\mathcal{E}_{f + z^2}^{\ord}$.
    If $d\equiv 1\bmod{2}$, then $2d\in\mathcal{E}_{f + z^2}^{\ord}$. 
    If $d\equiv 2\bmod{2}$, then $\frac{d}{2}\in\mathcal{E}_{f + z^2}^{\ord}$ and note
    that $\ell\mid\frac{d}{2}$.
    In all cases, $\ell\in\overline{\mathcal{E}_{f+z^2}^{\ord}}$, giving a contradiction.
    Hence, $\ell\equiv 0\bmod{2}$. If $\ell\equiv 0\bmod{4}$ it is also the case for~$d$
    and we get a contradiction with the same arguments. Hence,~\ref{b1} follows.

    If $2\ell\in\overline{\mathcal{E}_f^{\ord}}$, since $2\ell\equiv 0\bmod{4}$, then 
    $2\ell\in\overline{\mathcal{E}_{f+z^2}^{\ord}}$ and this contradiction gives~\ref{b3}.

    If $\frac{\ell}{2}\in\overline{\left\{d\in\mathcal{E}_f^{\ord}\mid  d \text{ odd}\right\}}$,
    we can pick up an odd~$d_0\in\mathcal{E}_f^{\ord}$ such that $\frac{\ell}{2}\mid d_0$.
    As $2d_0\in\mathcal{E}_{f+z^2}^{\ord}$, then $\ell\in\overline{\mathcal{E}}_{f+z^2}^{\ord}$
    and the contradiction gives~\ref{b4}.
\end{proof}

The following Lemma gives expressions for the local $\ell$-twisted topological zeta functions of the suspension $f+z^2$.
This result is a further step of
Proposition~\ref{prop:dv_ts} (namely, explicit formulas instead of conditions for the existence of poles) and it is useful to study the case of bad eigenvalues.

\begin{lemma}\label{lema:2l}
    Let $f\in\CC\{x_1,\dots,x_n\}$, $F:=z^2 + f$, and
    $\ell=2^a\ell_2$, $a\geq 0$ and $\ell_2$ odd. 
    The following holds:
    \begin{enumerate}[label=\rm(\arabic{enumi})]
        \item\label{lema:2l1} If $a=0$, $\Ztop{(\ell)}(F, s)_\zero =-\frac{1}{2}\left(\Ztop{(\ell)}(f, t)_\zero + \Ztop{(2\ell)}(f, t)_\zero \right)$.
        \item\label{lema:2l2-2} $\Ztop{(2)}(F,s)_\zero=  \frac{1}{2 }\left(\frac{1}{t}
-  \Ztop{(2)}(f,t)_\zero - \! \frac{t\!+\!1}{t}\Ztop{}(f,t)_\zero\right)
$.
        \item\label{lema:2l2} If $a=1$ and $\ell_2>1$, $\Ztop{(\ell)}(F, s)_\zero =-\frac{1}{2}\left(\Ztop{\left(\frac{\ell}{2}\right)}(f, t)_\zero + \Ztop{(\ell)}(f, t)_\zero \right)$.
        \item\label{lema:2l3} If $a>1$, 
$\Ztop{(\ell)}(F, s)_\zero=-\Ztop{(\ell)}(f, t)_\zero$.
    \end{enumerate}
\end{lemma}

\begin{proof}
It is a direct consequence of \eqref{twisted1fsuspACNLMth}
and
\eqref{twisted2fsuspACNLMth} in~Theorem~\ref{thmfsuspACNLM}
since $J_2(2)=3$.
\end{proof}

For the sake of completeness, we state and prove this result,
which gives Proposition~\ref{prop:dv}\ref{prop:dv2}.

\begin{lemma}\label{lem:dvk=2}
If the  holomorphy conjecture holds for $f$,
$\ell\notin\overline{\mathcal{E}_{f + z^2}^{\ord}}$, 
and $\ell$ is not $f$-bad, then $\Ztop{(\ell)}(f + z^2)_\zero=0$. \end{lemma}

\begin{proof}
Since $\ell$ is not $f$-bad and $\ell\notin\overline{\mathcal{E}_{f + z^2}^{\ord}}$, Lemma~\ref{lem:set-f-bad} implies  $\ell\notin\overline{\mathcal{E}_{f}^{\ord}}$ 
and it is also the case for any multiple of~$\ell$. By the Holomorphy Conjecture for $f$ we have that $\Ztop{(\lambda)}(f + z^2)_\zero=0$ for $\lambda$ equal to 
any multiple of $\ell$.

Hence, if $\ell$ is either odd or $0 \bmod{4}$, we obtain the vanishing of all the terms 
in the formulas for cases either~\ref{lema:2l1} or~\ref{lema:2l3} in Lemma~\ref{lema:2l}. Therefore,
 $\Ztop{(\ell)}(f + z^2)_\zero=0$.

Otherwise $\ell$  satisfies~\ref{b1}, i.e., $\ell\equiv 2\bmod{4}$. Since 
$\ell\notin\overline{\mathcal{E}_{f + z^2}^{\ord}}$, then 
$\frac{\ell}{2}\notin\overline{\mathcal{E}_{f}^{\ord}}$ because of \eqref{eq:casesorder2}. The  Holomorphy Conjecture for $f$ implies that $\Ztop{\left(\frac{\ell}{2}\right)}(f)_\zero=0$.
Hence, using  Lemma~\ref{lema:2l}\ref{lema:2l2} we get $\Ztop{(\ell)}(f + z^2)_\zero=0$.
\end{proof}

In the rest of the section, we are going to 
study $\Ztop{(\ell)}(f + z^2, s)_\zero$ when $\ell$ is $f$-bad.
Note that from Lemma~\ref{lema:2l}\ref{lema:2l2-2},
the formulas for $\ell=2$ are different.
Notice that if $\ell=2$ is $f$-bad, then all the eigenvalues
are bad, and, hence, all the elements of 
$\mathcal{E}_f^{\ord}$ are congruent with $2$ modulo $4$. The following lemma characterizes the case $\ell=2$ is $f$-bad for curves.

\begin{lemma}\label{lem:A2}
    Let $f$ define a germ $(C, \zero)\subset(\CC^2, \zero)$ of plane curve singularity such that all its eigenvalues are bad (i.e., their orders are 
    $2\bmod{4}$). 
    Then, $(C, \zero)$ is either of type $\mathbb{A}_{2n}$, $n\geq 1$, or $2C_1$ with $C_1$
    smooth.
\end{lemma}

\begin{proof}
Let 
$C=r_1 C_1 + \cdots + r_l C_l$ be a decomposition of $(C, \zero)$ into locally irreducible components.
Assume that all the eigenvalues of  $(C, \zero)\subset(\CC^2, \zero)$ are bad. Since~$1$ is always an 
eigenvalue when $l > 1$, we can assume 
that there is only one irreducible component or branch, i.e, $C=r_1 C_1$. If 
$r_1 > 2$,  then $C$ has an eigenvalue
of order~$r$ for any $r>1$ divisor of $r_1$.  Hence,  $C$ has eigenvalues of order either odd or multiple of~$4$, and, therefore, they are not bad.
If $r_1=2$, and $C_1$ has an eigenvalue of order $r$, then $C$ has eigenvalues
of orders $r$ and~$2r$. Hence, the $C_1$ cannot have eigenvalues, i.e., $C=2C_1$ with $C_1$ 
smooth.

\begin{figure}
    \centering

\begin{tikzpicture}
\foreach \x in {0, ..., 5}
{
\coordinate (A\x) at (2*\x, 0);
\coordinate (B\x) at (2*\x, -1);
}
\foreach \x in {0, 1, 2, 4}
{
\fill (A\x) circle [radius=.1];
}
\foreach \x in {1, 2, 4}
{
\fill (B\x) circle [radius=.1];
}
\foreach \x in {0,1,2}
{
\draw (A\x) -- ($(A\x) + .5*(1,0)$);
\draw[dotted] ($(A\x) + .5*(1,0)$) -- ($(A\x) + 1*(1,0)$);
}
\foreach \x in {1, 2, 4}
{
\draw (A\x) -- ($(A\x) - .5*(1,0)$);
\draw[dotted] ($(A\x) - .5*(1,0)$) -- ($(A\x) - 1*(1,0)$);
\draw (A\x) -- ($(A\x) - 1/3*(0,1)$) (B\x) -- ($(A\x) - 2/3*(0,1)$);
\draw[dotted] ($(A\x) - 1/3*(0, 1)$) -- ($(A\x) - 2/3*(0,1)$);
}
\draw[->] (A4) -- (A5);
\foreach \x in {1, 2}
{
\node[left] at (B\x) {$\bar{\beta}_\x$} circle [radius=.1];
\node[above] at (A\x) {$n_\x\bar{\beta}_\x$};
}
\node[above] at (A0) {$\bar{\beta}_0$};
\node[left] at (B4) {$\bar{\beta}_g$} circle [radius=.1];
\node[above] at (A4) {$n_g\bar{\beta}_g$};

\end{tikzpicture}
    \caption{Dual resolution  graph of $C=C_1$ with the multiplicities of the relevant exceptional components.}
    \label{fig:blow-up1}
\end{figure}
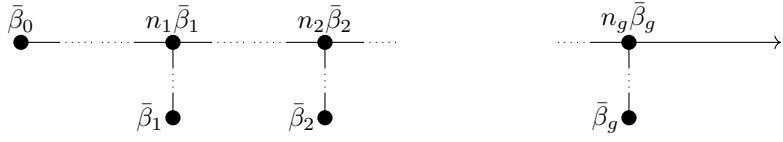

Let us consider the case $r_1=1$, i.e., $C=C_1$.
We want to show that $(C,0)$ is of type $\mathbb{A}_{2n}$ for some $n\geq 1$.   
Let us recover the notation in \cite[Theorem 10.3.10]{Wall-04}.
Assume that $C$ is singular and has $g\geq 1$ Puiseux pairs. The semigroup
of the curve is generated by $g+1$ numbers $1<\bar{\beta}_0<\bar{\beta}_1<\dots <\bar{\beta}_g$,
where $\bar{\beta}_0$ is the multiplicity of the singularity.
There are integer  sequences $e_0,e_1\dots,e_g$ and $n_1,\dots,n_g$ such that
\[
e_j:=\gcd(\bar{\beta}_0,\bar{\beta}_1,\dots\bar{\beta}_j),\quad \text{ and } \quad 
n_j:=\frac{e_{j-1}}{e_j}.
\]
With these notations, the characteristic polynomial of the monodromy is given by 
\[
\Delta(\tau)=
\frac{(\tau-1)(\tau^{n_1\bar{\beta}_1}-1)}{(\tau^{\bar{\beta}_0}-1)(\tau^{\bar{\beta}_1}-1)}
\prod_{j=2}^g
\frac{\tau^{n_j\bar{\beta}_j}-1}{\tau^{\bar{\beta}_j}-1}.
\]
Moreover, it is known  that $n_j\bar{\beta}_j\in\mathcal{E}^{\ord}_C$. Hence, by hypothesis 
$n_j\bar{\beta}_j\equiv 2\bmod{4}$ for $j>1$. If $n_j$ is odd or $>2$, then
$\frac{n_j\bar{\beta}_j}{2}\in\mathcal{E}^{\ord}_C$ and we get a 
contradiction. Then, $n_j=2$ for all $j>1$, and $\bar{\beta}_j$ is odd.
For $j=1$, a similar argument holds since 
$\bar{\beta}_0<\bar{\beta}_1$. 
If $g>1$, recall that $2=n_2 \mid  e_1 \mid  \bar{\beta}_1$. Then, we have 
$n_1\bar{\beta}_1\equiv 0\bmod{4}$, which is a contradiction.
Finally, we conclude $g=1$, $\bar{\beta}_0=2$, and $\bar{\beta}_1=2n+1$, i.e.,
the singular point is of type $\mathbb{A}_{2n}$.
\end{proof}

Let us focus on the case
$\ell>2$ and $f$-bad. The following remark will be useful in several parts of this paper.

\begin{remark}[{\cite[Lemma 2.3]{veys:99}}]\label{rem:divisibility_N}
The following property of the multiplicities of the exceptional divisor of 
an embedded resolution $\pi$ of a singularity $(C,p)$ holds. Let $D$ be an exceptional component with 
self-intersection $-e$, and let $D_1,\dots,D_m$ be the irreducible components of $\pi^{-1}(C)$ intersecting $D$.
The intersection number of the total transform of the curve $C$ with the exceptional component $D$ is zero by the projection formula. This implies that 
\begin{equation}\label{eq:vecinos}
        \sum_{j=1}^m N_{D_i}(C)= e N_D(C).
    \end{equation}
\end{remark}

\begin{lemma}\label{lema:cc}
    Let $f\in \CC\{x,y\}$. Let $\pi:X\to\CC^2$ be
    the minimal resolution of $f$, for which we consider its dual resolution graph. 
    For $n\in\NN$, let 
    $E^{(n)}$ be the union of the exceptional strata $E^\circ_J$ 
    such that the multiplicities of all the divisors through $E_J^\circ$ are multiples of~$n$, i.e., $\pi(E_j)=\zero$ for all $j \in J$ and $n$ divides  $\gcd (N_j \mid j \in J)$.

    Let $\ell>2$ be $f$-bad. Then, the connected components of 
    $E^{\left(\frac{\ell}{2}\right)}$ can be only of two types:
    \begin{enumerate}[label=\rm(\arabic{enumi})]
        \item\label{lema:cc2} ${E}^\circ_0$, where $E_0$ is a valence~$2$ component.
        \item\label{lema:cc1-3} $E^\circ_0\cup E^\circ_1 \cup E_{\{0,1\}}$, where $E_0$ is a valence~$3$ component, $E_1$ is a valence~$1$ component, and 
        $E_{\{0,1\}}= E_0\cap E_1\neq\emptyset$ is the intersection point of $E_0$ and $E_1$.
    \end{enumerate}
    For $E^{\left(\ell\right)}$ the components can be also of two types. We can have connected components of type{\rm~\ref{lema:cc2}}, and moreover 
    for each connected component $E^{(\frac{\ell}{2})}$ of type{\rm~\ref{lema:cc1-3}}, $E^\circ_0$
    is a connected component of $E^{\left(\ell\right)}$.
\end{lemma}

\begin{proof}
Let $\mathcal{C}$ be a connected component of $E^{\left(\frac{\ell}{2}\right)}$. 
    This connected component must contain the open stratum of an irreducible component
    of $\pi^{-1}(0)$.
    
    Assume that $\mathcal{C}$ contains a stratum $E_0^\circ$ where
    $E_0$ is an exceptional component of valence~$v(E_0)>2$. Let~$d_0=N_{E_0}$
    be the  multiplicity of $E_0$, in particular, $d_0\in\mathcal{E}_f^{\ord}$ and $\frac{\ell}{2}\mid d_0$. Since $\ell$ is $f$-bad,~\ref{b1},~\ref{b3} and~\ref{b4} imply that $\ell, d_0\equiv 2\bmod{4}$.
    Then, actually $\frac{\ell}{2}\mid\frac{d_0}{2}$, i.e., $\ell\mid d_0$. Let us show that $v(E_0)=3$ using properties of the dual resolution graph and the eigenvalues of the monodromy. 
   
   Consider the characteristic polynomial $\Delta(\tau) \in \mathbb{Q}[\tau]$  of the monodromy of $f$  and A'Campo's formula for  $\Delta(\tau)$ in terms of an embedded resolution, namely the minimal resolution $\pi$, see Remark~\ref{rem:acampo}, or Remark~\ref{rmk_cyclotomic}\ref{rmk_cyclotomic2} for a version in terms of the splice diagram. 
The divisibility property of Remark~\ref{rem:divisibility_N} is relevant in the following paragraphs.
   
   The  exceptional component $E_0$ contributes a factor $(\tau^{d_0} - 1)^{v(E_0) - 2}$ 
    to the numerator of A'Campo's formula for $\Delta(\tau)$. Moreover, since $v(E_0)>2$ we know that 
    $d_0\in \mathcal{E}_f^{\ord}$. However, as mentioned above, property~\ref{b4} implies that $\frac{d_0}{2}\notin\mathcal{E}_f^{\ord}$. 
    Then,
    there must be $r\geq v(E_0) - 2$ bamboos (not ending in arrowheads, i.e., components of the strict transform of $f$) attached to $E_0$, such that the multiplicity of their leaves $E_i$ for $i=1,\dots,r$ satisfy $N_{E_i}=\frac{d_0}{2}$,
    see the left-hand side of Figure~\ref{fig:bamboos}. Remark~\ref{rmk_cyclotomic}\ref{rmk_cyclotomic2} implies that the determinants $c_i$ of these bamboos must be equal to~$2$, and this is only possible if the bamboo has only one divisor of
    self-intersection~$-2$.

    Let us now use two general facts about dual graphs of resolutions of plane curve singularities, see~\cite[III\S8.4, Prop. 16]{BK86} or~\cite[\S5.4, Theorem 5.4.5]{dJP00}. Firstly, a vertex of the dual resolution graph of a plane curve singularity has at most two bamboos attached to it,
    and, secondly, if $r=2$, the determinants $c_1$ and $c_2$ must be coprime, see~\cite[\S9.8 and \S5.4]{Wall-04}. 
Since each $c_i=2$, if $r=2$ we obtain a contradiction and then $r=1$. Taking into account that $E_0$ is a branching component, we have $1\leq v(E_0) - 2 \leq r=1$.
We conclude that $v(E_0)=3$,  and
    we are in the situation of the right-hand side of Figure~\ref{fig:bamboos}.
    Hence, $\mathcal{C}$ has at least three strata, namely $E_0^\circ, E_1^\circ$, and $E_0\cap E_1$.

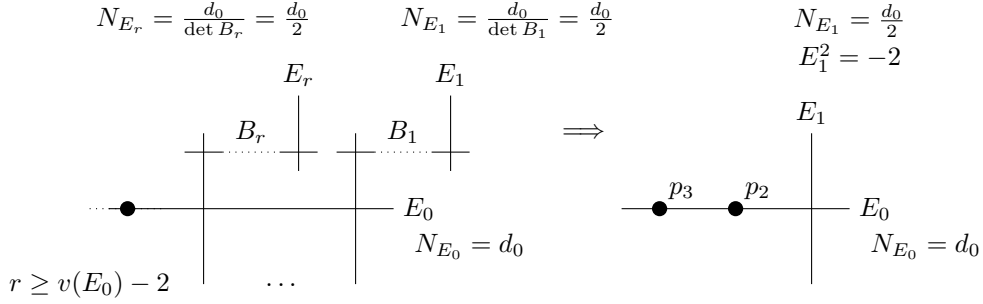
\begin{figure}[ht]
    \centering
    
    \begin{tikzpicture}
\draw (-2.25, 0) -- (1.5, 0) node[right] {$E_0$};
\draw[dotted] (-2.5, 0) -- (-1.5,0);
\draw (1, -1) -- (1, 1);
\draw (.75, .75) -- (1.25, .75);
\draw[dotted] (1.25, .75) --  node[above] {$B_1$}  (2,.75);
\draw (2, .75) -- (2.5, .75);
\draw (2.25, .5) -- (2.25, 1.5) node[above] {$E_1$};
\node at (.05, -1) {$\dots$};
\begin{scope}[xshift=-2cm]
\draw (1, -1) -- (1, 1);
\draw (.75, .75) -- (1.25, .75);
\draw[dotted] (1.25, .75) --  node[above] {$B_{r}$}  (2,.75);
\draw (2, .75) -- (2.5, .75);
\draw (2.25, .5) -- (2.25, 1.5) node[above] {$E_r$};
\end{scope}
\fill (-2,0) circle[radius=.1cm];
\node at (-2.5, -1) {$r\geq v(E_0) - 2$};
\node at (2.5, -.5) {$N_{E_0}=d_0$};
\node at (3, 2.5) {$N_{E_1}=\frac{d_0}{\det B_1}=\frac{d_0}{2}$};
\node at (-1, 2.5) {$N_{E_r}=\frac{d_0}{\det B_r}=\frac{d_0}{2}$};

\node at (4, 1) {$\implies$};

\begin{scope}[xshift=6cm]
\draw (-1.5, 0) -- (1.5, 0) node[right] {$E_0$};
\draw (1, -1) -- (1, 1) node[above] {$E_1$};

\fill (0,0) node[above right] {$p_2$} circle[radius=.1cm];
\fill (-1,0) node[above right] {$p_3$} circle[radius=.1cm];

\node at (2.5, -.5) {$N_{E_0}=d_0$};
\node at (1.5, 2.5) {$N_{E_1}=\frac{d_0}{2}$};
\node at (1.5, 2) {$E_1^2=-2$};

\end{scope}
\end{tikzpicture}
    \caption{Expected and final behavior around $E_0$.}
    \label{fig:bamboos}
\end{figure}

    Now we take $p_2$ and $p_3$ as the other two double points of $E_0$. Note that the space $\pi^{-1}(0)\setminus (E_0^\circ\cup E_1)$
    is formed by two connected components $F^2, F^3$ containing $p_2$ and $p_3$ respectively. From~\eqref{eq:vecinos} we have 
    $p_2\in E^{\left(\frac{\ell}{2}\right)}$ if and only if $p_3\in E^{\left(\frac{\ell}{2}\right)}$.
    Hence if $p_2 , p_3 \notin E^{\left(\frac{\ell}{2}\right)}$, we have that
    $\mathcal{C}=E_0^\circ\cup E_1^\circ\cup (E_0\cap E_1)$, meaning that it is a component of type~\ref{lema:cc1-3}. 
    In fact this is always the case, assume by contradiction that $p_2 , p_3 \in E^{\left(\frac{\ell}{2}\right)}$ then we will see that $\pi^{-1}(0)=E^{\left(\frac{\ell}{2}\right)}=\mathcal{C}$. This fact implies that $f=f_0^{\frac{\ell}{2}}$. As a consequence, $\frac{\ell}{2}\in\mathcal{E}_f$, which contradicts that $\ell$ is $f$-bad, namely~\ref{b4}.
    
    When $p_2, p_3\in E^{\left(\frac{\ell}{2}\right)}$ there are the following three possibilities for $F^2$, and $F^3$, see Figure~\ref{fig:p_i}.
    \begin{enumerate}[label=($\mathcal{S}$\arabic{enumi})]
    \item $F^i$ intersects the strict transform and has no branching component. The component $F^i$ is a bamboo, and we denote by $q_i$ the stratum in the strict transform. By \eqref{eq:vecinos}, we have that $F^i\subset E^{\left(\frac{\ell}{2}\right)}$, and then $F^i\subset\mathcal{C}$.
        
    \item $F^i$ does not intersect the strict transform and has no branching component.
    The component $F_i$ is a bamboo, and we denote by $E_i$ its last exceptional component. As in the previous case, $F^i\subset\mathcal{C}$.
        
    \item\label{S3} There is a branching component in $F^i$. Let $E_0'$ be the branching exceptional component \emph{closest} to $E_0$
    and let $E^i$ be the subset of $F^i$ given by the stratum   ${E_0'}^\circ$ and the strata contained in the bamboo connecting $E_0'$ and $E_0$. Using again \eqref{eq:vecinos}, we obtain that 
        $E^i\subset E^{\left(\frac{\ell}{2}\right)}$, and then $E^i\subset\mathcal{C}$.
        In particular, $E_0'$ is a branching component such that  ${E_0'}^\circ\subset E^{\left(\frac{\ell}{2}\right)}$, and then, the following holds:
        \begin{itemize}
        \item $v(E_0')=3$;
        \item $E'_0$ is attached to an exceptional divisor $E'_1$, $v(E'_1)=1$, and ${E'_1}^2=-2$;
        \item $E'_0$ contains two other strata, two points denoted by $p'_2$ and $p'_3$;
        \item $p'_i=E'_0\cap E^i$.
        \end{itemize}
        As above, since $p'_i\in E^i\subset E^{\left(\frac{\ell}{2}\right)}$, we have that 
        $p'_{5-i}\in E^{\left(\frac{\ell}{2}\right)}$. Figure~\ref{fig:p_i} illustrates these data.
        We can repeat this process replacing $p_i$ by $p'_{5-i}$. After a finite number of steps,
        we conclude the whole component $F^i$ is contained in $E^{(\frac{\ell}{2})}$.
    \end{enumerate}

    Recall that $\pi^{-1}(0)$ is the union of $E_0^\circ\cup E_1$, $F^2$, and $F^3$, and the previous arguments show that the three subsets are
    contained in $E^{\left(\frac{\ell}{2}\right)}$. Hence $\pi^{-1}(0)=\mathcal{C}$ as we wanted to prove.
    Summarizing, if $\mathcal{C}$ contains a branching component 
    we are in case~\ref{lema:cc1-3}.

    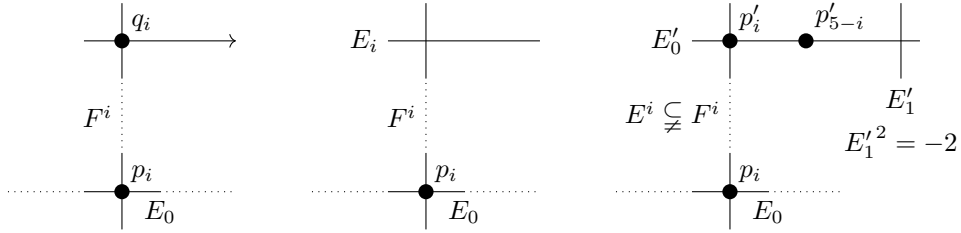
\begin{figure}
        \centering
        
    \begin{tikzpicture}
\draw (-.5, 0) -- (.5, 0) node[below] {$E_0$};
\draw[dotted] (-1.5, 0) -- (-.5,0) (.5, 0) -- (1.5, 0);
\fill (0,0) node[above right] {$p_i$} circle[radius=.1cm];
\draw (0, -.5) -- (0, .5);
\draw[dotted] (0, .5) -- (0,1.5);
\draw (0, 1.5) -- (0, 2.5);
\draw[->] (-.5, 2) -- (1.5, 2);
\fill (0,2) node[above right] {$q_i$} circle[radius=.1cm];
\node[left] at (0, 1) {$F^i$};
\begin{scope}[xshift=4cm]
\draw (-.5, 0) -- (.5, 0) node[below] {$E_0$};
\draw[dotted] (-1.5, 0) -- (-.5,0) (.5, 0) -- (1.5, 0);
\fill (0,0) node[above right] {$p_i$} circle[radius=.1cm];
\draw (0, -.5) -- (0, .5);
\draw[dotted] (0, .5) -- (0,1.5);
\draw (0, 1.5) -- (0, 2.5);
\draw (-.5, 2) node[left] {$E_i$} -- (1.5, 2);
\node[left] at (0, 1) {$F^i$};
\end{scope}

\begin{scope}[xshift=8cm]
\draw (-.5, 0) -- (.5, 0) node[below] {$E_0$};
\draw[dotted] (-1.5, 0) -- (-.5,0) (.5, 0) -- (1.5, 0);
\fill (0,0) node[above right] {$p_i$} circle[radius=.1cm];
\draw (0, -.5) -- (0, .5);
\draw[dotted] (0, .5) -- (0,1.5);
\draw (0, 1.5) -- (0, 2.5);
\draw (-.5, 2) node[left] {$E'_0$} -- (2.5, 2);
\fill (0,2) node[above right] {$p'_i$} circle[radius=.1cm];
\fill (1,2) node[above right] {$p'_{5-i}$} circle[radius=.1cm];
\draw (2.25, 1.5)  node[below] {$E'_1$} -- (2.25, 2.5);
\node[below] at (2.25, 1) {${E'_1}^2=-2$} ;
\node[left] at (0, 1) {$E^i\subsetneqq F^i$};
\end{scope}

\end{tikzpicture}

        \caption{Strata associated to $p_i$.}
        \label{fig:p_i}
    \end{figure}

    Assume now that $\mathcal{C}$ contains $E_1^\circ$ where $E_1$ is an exceptional component
    of valence~$1$, with neighbor $E_1'$. Using~\eqref{eq:vecinos}, $E_1\cap E_1'$ and ${E_1'}^\circ$ are 
    in $E^{\left(\frac{\ell}{2}\right)}$, and this property spreads along the bamboo containing $E_1$.
    Since the resolution is minimal, this bamboo connects to a branching component $E_0'$
    and ${E_0'}^\circ\subset\mathcal{C}$. Hence, we are in case~\ref{lema:cc1-3}.
    
    Assume, finally, that $\mathcal{C}$ contains $E_1^\circ$ where $E_1$ is an exceptional component
    of valence~$2$.  Let $E_1', E_1''$ be the neighbor components. Using again~\eqref{eq:vecinos}
    we get that $E_1\cap E_1'\in E^{\left(\frac{\ell}{2}\right)}$ if and only if $E_1\cap E_1''\in E^{\left(\frac{\ell}{2}\right)}$. If these points are in $E^{\left(\frac{\ell}{2}\right)}$, the propagation of this  property implies that either $\mathcal{C}$ must contain
    a component of valence $>2$ or $\mathcal{C}=E^{\left(\frac{\ell}{2}\right)}=\pi^{-1}(0)$. The latter gives a contradiction
    and for the former, we are in case~\ref{lema:cc1-3}.
    If it is not the case then 
    $\mathcal{C}=E_1^\circ$, i.e., we are in case~\ref{lema:cc2}.

    Note that $E^{(\ell)}\subset E^{\left(\frac{\ell}{2}\right)}$. If $\mathcal{C}$ is a connected component of 
    $E^{(\ell)}$ it is contained in a connected component of $E^{\left(\frac{\ell}{2}\right)}$
    and the last part of the statement follows.
    \end{proof}

\begin{example}\label{ex:f-bad-1}
    Let $f=(y^2-x^3)^3-x^6y^3$, as in Example~\ref{ex:f-bad}. The resolution of $f$ has 4 exceptional components, see Figure~\ref{fig:bad}, with the following numerical data.
\[
\begin{array}{|c|c|c|c|c|}
    \hline
    &E_1&E_2&E_3&E_4\\
    \hline
    N&6&9&18&21\\
    \hline
    \nu&2&3&5&6\\
    \hline
    E^2&-3&-2&-2&-2\\
    \hline
\end{array}
\]
    The integer 18 is $f$-bad and $E^{(9)}$ has only one component of type~\ref{lema:cc1-3}, namely, $E_2^\circ \cup E_3^\circ \cup E_{\{2,3\}}$. Note that $E_2$ has valence 1 and $E_3$ has valence 3. Then,
    \[
    \Ztop{(9)}(f, s)_\zero=\frac{-1}{5 + 18 s}+\frac{-1}{(5 + 18 s)(3 + 9s)}+\frac{1}{3+9s}=\frac{1}{5 + 18 s}.
    \]
    The only connected component of $E^{(18)}$ is $E_3^\circ$. Then,
    \[
    \Ztop{(18)}(f, s)_\zero=\frac{-1}{5 + 18 s}.
    \]
    As a  consequence of Lemma~\ref{lema:2l}\ref{lema:2l2} we have $\Ztop{(18)}(F, s)_\zero=0$.
\end{example}

\begin{lemma}\label{lem:notextbad}
Let $f\in\CC\{x, y\}$ and let $\ell>2$ be $f$-bad.
Then $\Ztop{(\ell)}(f+z^2, s)_\zero=0$.
\end{lemma}

\begin{proof}
Since we will use Lemma~\ref{lema:2l}\ref{lema:2l2}, we compute below $\Ztop{(n)}(f, s)_\zero$ for $n=\frac{\ell}{2}, \ell$. 

The connected
components of $E^{\left(\frac{\ell}{2}\right)}$ and $E^{\left(\ell\right)}$ are described in Lemma~\ref{lema:cc}.
If they are
of type~\ref{lema:cc2}, they
do not contribute. Let us fix a component $E^\circ_0\cup E^\circ_1 \cup E_{\{0,1\}}$ 
of type~\ref{lema:cc1-3} of $E^{\left(\frac{\ell}{2}\right)}$;
recall that $E_0^\circ$ is a component of $E^{\left(\ell\right)}$.
It is enough to make the computations on these divisors, see the right-hand side of Figure~\ref{fig:bamboos}.

With the notation of the proof of Lemma~\ref{lema:cc},
the multiplicities are 
$N_{E_0}=d_0$, $N_{E_1}=\frac{d_0}{2}$. In order to compute, $\nu_{E_1}, \nu_{E_0}$ 
we use the adjunction formula. Let $K_\pi$ be the relative canonical divisor.
Then $K_\pi =(\nu_{E_0} - 1) E_0 + (\nu_{E_1} - 1) E_1 + K_1$, where $K_1\cdot E_1=0$.
Then, we have
\[
-2 = E_1\cdot(E_1 + K) = E_1\cdot((\nu_{E_0} - 1) E_0 + \nu_{E_1} E_1 + K_1)=
\nu_{E_0} - 1 - 2\nu_{E_1}.
\]
Hence $\nu_{E_0}=2\nu_{E_1} - 1$. The candidate pole  of the ($\ell$-twisted) topological zeta function associated to $E_1$ is $u:=\nu_{E_1} + \frac{d_0}{2} s$ and the candidate pole  associated
to $E_0$ is $2\nu_{E_1} - 1+ d_0 s=2 u -1$. 

Let us compute the contributions of this component to the twisted zeta functions of $f$
for $\ell,\frac{\ell}{2}$:
For 
$\Ztop{(\ell)}(f, s)_\zero$, it is $\frac{-1}{2u - 1}$.
For $\Ztop{\left(\frac{\ell}{2}\right)}$ we have:
\[
\frac{-1}{2u - 1} + \frac{1}{(2u - 1) u } + \frac{1}{u}=
\frac{-u + 1 + 2 u -1}{(2u - 1) u }= \frac{1}{2u - 1}.
\]
Adding up, we obtain the result.
\end{proof}

\begin{theorem}\label{thm:holconfsuspcurve}
The holomorphy conjecture holds for $f+z^2$ if $f\in\CC\{x, y\}$.
\end{theorem}

\begin{proof}
Let $\ell\notin\overline{\mathcal{E}_{f + z^2}^{\ord}}$. If $\ell$ is not $f$-bad, then
$\Ztop{(\ell)}(f + z^2)_\zero=0$ by Lemma~\ref{lem:dvk=2}. 

We assume for now that $\ell$ is $f$-bad. If $\ell=2$ we fall in the scope of
Lemma~\ref{lem:A2}. As $2\notin\overline{\mathcal{E}_{f + z^2}^{\ord}}$ and 
$f + z^2$ is Brieskorn-Pham singularity, which satisfies
the holomorphy conjecture because of~\cite[Proposition 3.7]{dv:95}.
Then $\Ztop{(2)}(f + z^2)_\zero=0$ if $2\notin\overline{\mathcal{E}_{f + z^2}^{\ord}}$.

The case $\ell>2$ and $f$-bad follows directly from Lemma~\ref{lem:notextbad}.
\end{proof}

\begin{remark}\label{rmk:hol_surf} Let us briefly discuss the status of the holomorphy conjecture for suspensions by two points, i.e. $\ex=2$, of singular hypersurfaces of dimension greater or equal 2. 

Notice firstly that if  $f$ defines a singular surface and it is (eventually after a change of variables) a suspension, i.e., it can be written as $f(x,y,w)=f_0(x,y)+w^n$, then $f$ satisfies the holomorphy conjecture because of  Proposition~\ref{prop:dv}\ref{prop:dv1} and Theorem~\ref{thm:holconfsuspcurve}. The hypersurface $f+z^2= f_0 +  w^n + z^2$
also satisfies the holomorphy conjecture. If $n>2$, this follows from  Theorem~\ref{thm:holconfsuspcurve} applied to $g:= f_0 + z^2$ and Proposition~\ref{prop:dv}\ref{prop:dv1} applied to $g+w^n$. If $n=2$, then $f+z^2$ it is a double  suspension, and it satisfies the holomorphy conjecture because of~\cite[\S3.8(ii)]{dv:95}. 

Let us secondly focus on the case that $f$ defines a singular surface, and it is not a suspension. Moreover, we can assume $f$ satisfies the holomorphy conjecture. By Proposition~\ref{prop:dv}\ref{prop:dv2} or Lemma~\ref{lem:dvk=2} (which is a rewriting of
Denef-Veys' results), we have $\Ztop{(\ell)}(f+z^2, s)_\zero=0$ if $\ell$ is not $f$-bad. If $\ell$ is $f$-bad, Lemma~\ref{lema:2l} provides a way to compute $\Ztop{(\ell)}(f+z^2, s)_\zero$. The case $\ell=2$ is $f$-bad seems particularly relevant to the holomorphy conjecture, see the formula in Lemma~\ref{lema:2l}\ref{lema:2l2-2}.

We say that a singular hypersurface $f$ is \emph{$2$-bad} if $2$ is $f$-bad. Equivalently, $a \equiv 2\bmod{4}$ for all $a \in \mathcal{E}^{\ord}_f$
because of Remark~\ref{Rmk:f_bad}.  Lemma~\ref{lem:A2} characterizes $2$-bad curves as $\mathbb{A}_{2n}$, for $n\geq 1$,
and $x^2=0$. Moreover, let us describe a construction of $2$-bad hypersurfaces of dimension $d \geq 2$.
Let $f_0$ be a singular hypersurface such that $\mathcal{E}^{\ord}_f$ 
consists of odd integers. In such a case, we say $f_0$ is \emph{odd}.  Then $f= f_0 + w^2$, with $w$ a new variable, is $2$-bad. Examples of odd hypersurfaces can be found, for instance, in ~\cite{df:13}.

Notice that if $f_0$ is odd, then $f_0$ has no bad eigenvalue, and the  hypersurface $f=f_0 + w^2$ constructed above is $2$-bad. If additionally  $f_0$ satisfies the holomorphy conjecture, so does the suspension $f+ z^2$ because of~\cite[\S3.8(ii)]{dv:95}. 

However, it may be of interest to the status of the holomorphy conjecture  
to look for $2$-bad singularities~$f$ of surfaces (or in higher dimension) which are not suspensions
and investigate $\Ztop{(2)}(f+z^2, s)_\zero$ by means of Lemma~\ref{lema:2l}\ref{lema:2l2-2}.

\end{remark}

 \section{Topological zeta functions for Lê-Yomdin singularities}\label{sec:yomdin}

\begin{defn}\label{Def:k_LYS}
Let $(V,0)\subset(\CC^{n+1},\zero)$ be an isolated germ of singularity
defined by  $F(\mathbf{x})=0$, where $F\in \CC\{\mathbf{x}\}$ and $F=F_\exz + F_{\exz  + \ex}+ \dots$
is its homogeneous decomposition. For $r\geq m$, let $C_r:=V_\PP(F_r)$ be
the projective zero locus in $\PP^n$. We say that $V$ is a \emph{$\ex$-Lê-Yomdin
singularity}, $\ex$-LYS for short, if $\sing C_\exz\cap C_{\exz  + \ex}=\emptyset$. 
We say that $V$ is a  \emph{superisolated singularity}, SIS for short, if it is a $1$-Lê-Yomdin
singularity.
\end{defn}

Note that the projectivized tangent cone $C:=C_\exz$ of a $\ex$-LYS has only isolated singularities (for $n\geq 2$ it is reduced).
Let $\pi:\hat{\CC}^{n + 1}\to\CC^{n + 1}$ be the blow-up of $\CC^{n + 1}$ at the origin.
As a divisor $\pi^*V=\hat{V} + \exz  E$, where 
$\hat{V}:=\overline{\pi^{-1}(V\setminus\zero)}$ is the strict transform of $V$ and $E :=  \pi^{-1}(\zero)$ is the exceptional divisor. Notice that $E \cong\PP^n$ and with this isomorphism 
$\hat{V}\cap E=C$.

Let us consider the following stratification, by connected components, of the exceptional divisor 
\begin{equation}\label{partitionkLYS}
E=(E\setminus C)\sqcup(C\setminus\sing C)\sqcup\bigsqcup_{q \in\sing C}\{q\}.
\end{equation}
Using the standard coordinates of the charts of a blow-up,
the following result is straightforward.

\begin{lemma}\label{lem:localeqkLYS}
Let $F$ define a $\ex $-LYS, let $\omega$ be the standard volume form $\dif \bx$ in $\CC^{n+1}$, and let $\pi$ be the blow-up of $\CC^{n+1}$ at the origin with exceptional divisor $E$. 
For $q \in E$, there exits a local system of coordinates $(\mathbf{u}, z)=(u_1,\dots,u_n,z)$ centered at $q$ such that, up to units $\pi^*\omega=z^n\dif \mathbf{u} \wedge\dif z$, and 
\[
F\circ\pi=\begin{cases}
    z^{\exz}, & \text{ if }\ q \in E\setminus C,\\
    u_1z^{\exz}, & \text{ if }\ q \in C\setminus\sing C,\\
    z^{\exz} (z^\ex + f_q (\mathbf{u})), & \text{ if }\ q \in\sing C \text{ and }\ f_q (\mathbf{u})=0\, \text{ is a local equation of }\ (C,q).
    \end{cases}
\]
\end{lemma}
Combining
the change of variables formula of Proposition~\ref{prop:cv} and the Stratification Principle of Proposition~\ref{prop:sp}, we obtain
the following result.

\begin{prop}\label{prop:DLZFkLY}
With the previous notation we have
	\begin{align}
	\Znv{}(F,T)_\zero&=
	[\PP^n\setminus C] \cdot \Znv{}(z^{\exz}, z^{n}\dif\mathbf{u}\wedge\dif z,T)_\zero  \nonumber \\  &
    +[C\setminus\sing C] \cdot \Znv{}(u_1 z^{\exz}, z^{n}\dif\mathbf{u}\wedge\dif z,T)_\zero  \label{NaivekLYS} \\
	&+\sum_{q\in\sing C}\!\!\!\!\Znv{}(z^{\exz} (z^\ex + f_q), z^{n}\dif\mathbf{u}\wedge\dif z,T)_\zero. \nonumber
	\end{align}
Additionally, if we set $r=\frac{1 + n + (\exz  + \ex ) s}{\ex}$, we have
\begin{align}
	\Ztop{}(F, s)_\zero&=
    \frac{\chi(\PP^n\setminus C)}{\ex (r - s)} +
    \frac{\chi(C\setminus\sing C)}{\ex (r - s)(s+1)}\label{TopkLYS}
+\!\!\!\!\!\sum_{q\in\sing C}\!\!\!\Ztop{}(z^{\exz} (z^\ex + f_q), z^{n}\dif\mathbf{u}\wedge\dif z,s)_\zero . 
\end{align}
Moreover, for $\ell>1$ we have 
\begin{equation}\label{eq:tw-LYS}
\Ztop{(\ell)}(F, s)_\zero=
\delta(\ell, \exz)\frac{\chi(\PP^n\setminus C)}{\ex(r -s)} 
+ \sum_{q\in\sing C}\!\!\!\Ztop{(\ell)}(z^{\exz} (z^\ex + f_q), z^{n}\dif\mathbf{u}\wedge\dif z,s)_\zero,
\end{equation}
where $\delta(\ell, \exz)=1$ if $\ell \mid \exz$ and~$0$ otherwise. 

Finally, there is a formula for $\Zmot{}(F, \omega,T)_\zero$ similar to \eqref{NaivekLYS}.
\end{prop}

\begin{proof}
Expression \eqref{NaivekLYS} is consequence of the the Stratification Principle~\ref{prop:sp} applied to \eqref{partitionkLYS}, the isomorphism $E \cong \PP^n$ and Lemma~\ref{lem:localeqkLYS}. 
The expression \eqref{TopkLYS} is consequence of
\[
\Ztop{}(z^{\exz}, z^{n}\dif\mathbf{u}\wedge\dif z,s)_\zero = \frac{1}{\exz s + n +1} \text{ and } \Ztop{}(z^mu_1, z^{n}\dif\mathbf{u}\wedge\dif z,s)_\zero = \frac{1}{(\exz s + n + 1)(s+1)},
\]
the equality $\ex (r-s) = \exz s + n+1$, and $\chi([\PP^n])= n +1$. To obtain \eqref{eq:tw-LYS} notice that the stratum $C \setminus \sing C$ does not contribute because $\ell \nmid \gcd (m, 1)= 1$. 
\end{proof}

Combining  Theorem~\ref{thmpsuspACNLM} and Proposition~\ref{prop:DLZFkLY} we get formul{\ae} for the local (twisted) topological zeta function of $\ex$-LYS in terms of the local (twisted) topological zeta function of the local equations at singular points of the projectivized tangent cone.

\begin{theorem}\label{thm:DLZFkLY}
With the previous notations from Proposition{\rm~\ref{prop:DLZFkLY}} and Lemma{\rm~\ref{lema:arit}}, the following statements hold. 
\begin{enumerate}[label=\rm(\roman{enumi})]
\item\label{ztopklis1} For $\ell=1$, we have
\begin{equation*}
	\Ztop{}(F, s)_\zero=
        \frac{\chi(\PP^n\setminus C)}{\ex (r - s)} +
    \frac{\chi(C\setminus\sing C)}{\ex (r - s)(s+1)}
+\sum_{q\in\sing C}\frac{A_q}{\ex}
\end{equation*}
where 
\begin{align*}
    A_q\!=\!
    \frac{1}{r}\!
+\!
\frac{s(r + 1)(s - r + 1)}{r(s + 1)(r - s)} 
\Ztop{}(f_q, r)_\zero\!
	-\! \frac{s}{s+1} \sum_{1\neq e | \ex } J_2(e) \Ztop{(e)}(f_q, r)_\zero. 
\end{align*}
    \item\label{ztopklisdiv2} For $1\neq \ell\mid \exz  + \ex$, $\ell\mid \exz $ we have
\begin{equation*}
\Ztop{(\ell)}(F\!,\!s)_\zero \!=\! 
\frac{\chi(\PP^n\setminus C)}{\ex (r - s)}
+\!\!\!\!\!\sum_{q\in\sing C}\!\!
\left(\frac{1 \!-\! (r+1)\Ztop{}(f_q,r)_\zero}{\ex r}\!+ \!
\frac{\Ztop{(\ell)}(f_q,r)_\zero}{\ex (r - s)}\!-\!\!\!
\sum_{1\neq e \mid \ex }\!\! \frac{J_2(e)}{\ex}
\Ztop{(e)}(f_q,r)_\zero\!
\right)\!.
\end{equation*}

\item\label{ztopklisdivkm} For $\ell\mid \exz  + \ex$, $\ell\nmid \exz $ we have
\begin{equation*}
\Ztop{(\ell)}(F,s)_\zero= 
\sum_{q\in\sing C}
\left(
\frac{1 - (r + 1)\Ztop{}(f_q,r)_\zero}{\ex r}
-\sum_{1\neq e \mid \ex } \frac{J_2(e)}{\ex}
		\Ztop{(e)}(f_q, r)_\zero
\right).
\end{equation*}

\item\label{ztopklisdivk} For $\ell\nmid \exz  + \ex$ and $\ell\mid \exz $  we have
\begin{equation*}
\Ztop{(\ell)}(F,s)_\zero = 
\frac{\chi(\PP^n\setminus C)}{\ex (r - s)} +
\sum_{q\in\sing C}
\left(
\frac{\Ztop{(\ell)}(f_q, r)_\zero}{\ex (r - s)}-  \sum_{e | \ex } \frac{J_2(e)}{\ex } 
	\Ztop{(\lcm(e,\mkl(\ex,\ell, \exz  + \ex)))}(f_q, r)_\zero
    \right).
\end{equation*}

\item\label{ztopklisnodiv} For $\ell\nmid \exz  + \ex$ and $\ell\nmid \exz $  we have
\begin{equation*}
	\Ztop{(\ell)}(F,s)_\zero= 
    -\sum_{q\in\sing C}
	\sum_{e | \ex } \frac{J_2(e)}{\ex } 
	\Ztop{(\lcm(e, \mkl(\ex,\ell, \exz  + \ex)))}(f_q, r)_\zero.
\end{equation*}
\end{enumerate}
\end{theorem}

\begin{proof} For~\ref{ztopklis1} it is enough to apply \eqref{fsuspACNLMth-p} from Theorem~\ref{thmpsuspACNLM} with $\bnu^0 = \one$, $\nu_z=n+1$ and $\ex r= (\exz  + \ex)s+n+1$, to the term $\Ztop{}(F, s)_\zero$ in \eqref{TopkLYS}.  

For the rest of the statements we use again Theorem~\ref{thmpsuspACNLM}
and the motivic version of the formula~\eqref{NaivekLYS} in Proposition~\ref{prop:DLZFkLY}.
Note that the term corresponding to $C\setminus\sing C$
disappears and the term corresponding to $\PP^n\setminus C$  one only appears when $\ell$
divides~$m$.

For the last term in \eqref{eq:tw-LYS}, namely $\Ztop{(\ell)}(z^\exz (z^\ex + f_q), z^{n}\dif\mathbf{u}\wedge\dif z,s)_\zero$,
we apply \eqref{twisted1fsuspACNLMth-p-p+k}, 
\eqref{twisted1fsuspACNLMth-nop-p+k}, 
\eqref{twisted2fsuspACNLMth-p}, and
\eqref{twisted2fsuspACNLMth-p1} with $\bnu^0 =\one$ and $\nu_z=n+1$. 
\end{proof}

For concrete examples of Theorem~\ref{thm:DLZFkLY}\ref{ztopklis1}, see Examples~\ref{ex:m<4}.

\begin{cor}\label{cor:canpolekLYS} Let $F$ define a $\ex $-LYS with projectivized tangent cone $C=C_\exz$.
Using Notation{\rm~\ref{ntc:polos}}, the set $\pol(F)$ is contained in 
\[
\left\{
1, \frac{n+1}{\exz } 
\right\}
\cup
\left\{
\cpolo(\rho_0,n + 1,\exz,\ex)\,
\middle| \,
\rho_0\in\pol(f_q) \text{ and } q \in \sing C
\right\}.
\]

\end{cor}

\begin{proof}
Examination of the formula for $\Ztop{}(F, s)_\zero$ in Theorem~\ref{thm:DLZFkLY} reduces the set of negatives of candidate poles of $\Ztop{}(F, s)_\zero$ to the set    
\[
\left\{
1, \frac{n+1}{\exz }, \frac{n+1}{\exz + \ex } 
\right\}
\cup
\left\{
\cpolo(\rho_0,n + 1,\exz,\ex)\,
\middle| \,
\rho_0 \in \pol(f_q) \text{ and } q \in \sing C
\right\},
\]
and Corollary~\ref{cor:poles_G} discards the candidate pole $-\frac{n+1}{\exz + \ex }$ if it is not $-1$.\end{proof}

For the sake of completeness we also present the corresponding formul{\ae} for topological zeta functions of SIS.

\begin{cor}\label{cor:zetaSIS}
Let $F$ define a SIS, let $\omega$ be the standard volume form $\dif \bx$ in $\CC^{n+1}$, and let $\pi$ be the blow-up of $\CC^{n+1}$ at the origin with exceptional divisor $E$. Denote by $C$ the projectivized  tangent cone of the SIS.
 For $q \in\sing C$, denote by $f_q=0$ a local equation of $(C, q)$. Set $t=(1+ \exz)s + n+1$.  Then:
\begin{enumerate}[label=\rm(\arabic{enumi})]
    \item\label{sis-1} If $\ell=1$,
    \begin{equation*}
        \Ztop{}(F, s)_\zero=\frac{\chi(\PP^n\setminus C)}{t - s}-
    \frac{\chi(C\setminus\sing C)}{(s + 1)(t - s)}
	+\sum_{q \in\sing C}
\left(\frac{1}{t}+\frac{s(t + 1)(s - t + 1)}{t (s + 1) (t - s)} 
\Ztop{}(f_q, t)_\zero\right). 
    \end{equation*}

\item\label{sis-2} If $1\neq\ell\mid \exz + 1$ (hence $\ell\nmid \exz$):
\begin{equation*}
    \Ztop{(\ell)}(F, s)_\zero=
\frac{1}{t}\sum_{q \in\sing C}
\left(
1 -
(t + 1)\Ztop{}(f_q,t)_\zero\right).
\end{equation*}

\item\label{sis-3} If $1\neq\ell\mid \exz$ (hence $\ell\nmid \exz + 1$):
\begin{equation*}
    \Ztop{(\ell)}(F, s)_\zero=
    \frac{1}{t - s }
    \left(
    \chi(\PP^{n}\setminus C) +
(s - t  + 1)\sum_{q \in\sing C}\Ztop{(\ell)}(f_q, t)
    \right).
\end{equation*}

\item\label{sis-4} If $1\neq\ell\nmid \exz + 1$ and $\ell\nmid \exz$:
\begin{equation*}
    \Ztop{(\ell)}(F, s)_\zero=
\sum_{q\in\sing C}
	\Ztop{\left(\frac{\ell}{\gcd(\ell, \exz + 1)}\right)}(f_q,t)_\zero.
\end{equation*}
\end{enumerate}
    
\end{cor}

\begin{proof}
    Statement~\ref{sis-1} is the first part of Theorem~\ref{thm:DLZFkLY} with $\ex=1$. The case 
    $\ell$ divides both $\exz$, and $\exz +1$ is empty. We follow the same strategy
    as in Theorem~\ref{thm:DLZFkLY}.

    Statements ~\ref{sis-2} 
    and~\ref{sis-3}
    follow from \eqref{twisted1fsuspACNLMth-nop-p+k} and 
    \eqref{twisted2fsuspACNLMth-p}, respectively,
with $\ex=1$, $\bnu^0 =\one$ and $\nu_z=n+1$.  
Notice that there is only one value $e$ such that $e \mid 1$, namely, $e=1$. Notice also that, using the notation of Lemma~\ref{lema:arit}, 
  $\mkl(1, \ell, 1+ \exz )= \ell_1 = \frac{\ell}{\gcd(\ell, \exz + 1)}$.

  Statement~\ref{sis-4} follows from \eqref{twisted2fsuspACNLMth-p1} with $\ex =1$, $\bnu^0 =\one$, and $\nu_z=n+1$. As in the previous case, notice that there is only one value $e=1$, and  $\mkl(1, \ell, 1+ \exz ) = \frac{\ell}{\gcd(\ell, \exz + 1)}$.
\end{proof}
 
\section{Monodromy conjecture for \texorpdfstring{$\ex$}{\ex}-LYS}\label{sec:MCk-LYS}

Let $F:(\CC^{n+1},\zero)\to (\CC,\zero)$ be the defining equation of a germ of isolated singularity of a hypersurface. 
The monodromy of the Milnor fibration of $F$ was discussed in \S\ref{sec:hc_kLYS}.

\begin{mon_conj}[\cite{DL-JAMS}]
\label{monconj}
If $s_0$ is a pole of $\Ztop{}(F,s)_\zero$, then $\exp(-2 \pi s_0 i)$ is an eigenvalue of the monodromy of $F$ on $H_\ast (\mathcal{F}_{F,x},\CC)$ at some point $x\in F^{-1}(0)$.
\end{mon_conj}

The conjecture has been proved in~\cite{ACNLM-ASENS} for the case of SIS surfaces, i.e., $n=2$ and $\ex=1$ in Definition~\ref{Def:k_LYS}. In this section we prove the conjecture for $\ex$-LYS surfaces, i.e., $n=2$ and $k \geq 1$. For an recent account of the status of the conjecture see~\cite{Veys_Intro}. 
\begin{theorem}\label{thm:CMk_LYS} Let $F$ define a surface $\ex$-LYS with projectivized tangent cone $C=C_\exz$.
If $s_0$ is a pole of $\Ztop{}(F, s)_\zero$, then $\exp(-2 \pi s_0 i)$ is an eigenvalue of the monodromy of $F$ at some point of $F^{-1}(0)$. 
\end{theorem}

The proof of  Theorem~\ref{thm:CMk_LYS} is given in \S\ref{Sec:Proof_Mon_Conj}. It follows a similar strategy than the proof of the monodromy conjecture for SIS~\cite[Theorem 3.1]{ACNLM-ASENS}, and  relies on three key ingredients. First, the computation of the characteristic polynomial of the monodromy of a $\ex$-LYS, which is recalled in Section~\ref{Sec:Characteristic_Pol}. Second, the study of some rational pencils of plane curves that are described in Section~\ref{K-K-pencils}. Third, the notion of bad divisors explained in Section~\ref{Sec:Bad_div}. 

\subsection{The characteristic polynomial of a $\ex$-LYS}\label{Sec:Characteristic_Pol}
\mbox{}

Let us recall the following expression for the characteristic polynomial $\Delta$ of the monodromy of the $\ex$-LYS  defined by $F$.

\begin{prop}  [{\cite[Section 3.2]{GLM:97} }] 
\label{prop: char-klys}The characteristic polynomial $\Delta$ of the monodromy $\Psi$ of the $\ex$-LYS  defined by $F$ and projectivized tangent cone $C_\exz$ is given by
\begin{equation}\label{Deltak-LYS}
\Delta(\tau)= \frac{(\tau^{\exz} -1)^{\chi(\PP^2 \setminus C_{\exz})}}{\tau -1} 
\prod_{q \in \sing C_\exz} \Delta^{(\ex)}_q (\tau^{\exz + \ex })
\end{equation}
where $\Delta^{(\ex)}_q$  is  the characteristic polynomial of the $\ex$-th power of the  monodromy $\Psi_q$ of $(C_\exz, q)$.
\end{prop}

Formula \eqref{Deltak-LYS} was proved first in~\cite[Section 3.2]{GLM:97}, using a formula \emph{à la} A’Campo~\cite{MR371889} for
partial resolutions, and it was reproved in~\cite[Theorem 4.3]{jmm:14}, using $\QQ$-resolutions, see also~\cite[Proposition 5.1]{bartolo2025superisolatedsingularitiesfriends}.

The following notation is motivated by the special role of the eigenvalue 1 of the monodromy. 
\begin{notation} \label{notation: tildeD}
We denote by $\tilde{\Delta}(\tau)$ the product $(\tau - 1)\Delta(\tau)$. We call the term $(\tau^{\exz} -1)^{\chi(\PP^2 \setminus C_{\exz})}$ in \eqref{Deltak-LYS}  the \emph{global factor}
of $\tilde{\Delta}(\tau)$, and we call  the  terms $\Delta^{(\ex)}_q (\tau^{\exz + \ex })$ with $q \in \sing C_\exz$,  which are polynomials, the \emph{local factors} of $\tilde{\Delta}(\tau)$. \end{notation}

\begin{remark}\label{rmk_cyclotomic} It will be useful to have decompositions of the polynomials 
$\Delta(\tau)$,  $\tilde{\Delta}(\tau)$,  $\Delta_q(\tau)$ and $\Delta^{(\ex)}_q(\tau)$ in terms of cyclotomic polynomials. The following expressions can be deduced from A'Campo's formula.

\begin{enumerate}[label=\rm(\arabic{enumi})]

\item\label{rmk_cyclotomic0} 
If $\chi(\PP^2 \setminus C_{\exz})>0$,  the global factor is a polynomial. If $\chi(\PP^2 \setminus C_{\exz})=0$
the global factor is $1$. 

\item\label{rmk_cyclotomic1}  Because of the Monodromy Theorem, the roots of   the polynomials $\Delta(\tau)$, $\Delta_q(\tau)$ and $\Delta^{(\ex)}_q(\tau)$  are roots of unity. Therefore, all such polynomials are products of cyclotomic polynomials of type
\[
\Phi_u (\tau) = \prod_{\substack{ 1 \leq \ell < u \\ \gcd(\ell, u)=1}} 
\left(\tau - \exp\frac{2 \pi \ell i}{u}\right).
\]

\item\label{rmk_cyclotomic2} For $q \in \sing C_\exz$, the polynomial $\Delta_q(\tau)$ can be computed from the splice diagram or the dual resolution graph of $(C_\exz, q)$, see~\cite{MR371889} and~\cite[p. 96]{zbMATH04022161}; namely,
\[
\Delta_q(\tau)=\prod (\tau^{n_v} - 1)^{\delta_v - 2},
\]
where $v$ varies in the set of vertices, $n_v$ is the multiplicity and $\delta_v$ is the valence of the vertex $v$. Moreover, if  $v$ and $w$ are vertices of the diagram with 
valencies~$3$ and~$1$ and they are connected by an edge or by a chain of vertices of valence~$2$, then 
$n_w \mid n_v$ and 
\[
\frac{\tau^{n_v} - 1}{\tau^{n_w} - 1} =  \prod_{\substack{d \mid n_v\\d \nmid n_w }} \Phi_d(\tau)
\]
is a polynomial. Indeed, $n_v = c n_w$, with $c$ the decoration of the edge joining $v$ and $w$. Notice that, in terms
of the dual resolution graph, $c$ is the determinant of the bamboo associated to the edge, see~\cite[Chapter~V]{zbMATH04022161}.
\end{enumerate}
\end{remark}

Let $h(\tau)$ be a product of cyclotomic polynomials, which can
be seen as the characteristic polynomial of a finite-order
linear automorphism~$\Psi$. We denote $h^{(\ex)}(\tau)$
the characteristic polynomial of $\Psi^\ex$. 
Notice that if 
$h(\tau) = \prod_{j=1}^n (\tau - \zeta_j)$ then 
$h^{(\ex)}(\tau) = \prod_{j=1}^n (\tau - \zeta_j^{\ex})$. We have also the following description of the polynomial $h^{(\ex)}(\tau)$.

\begin{lemma}[{\cite[\S3.2]{GLM:97}}]
\label{lema:glm}
    Let $h(\tau)$ be a product of cyclotomic polynomials. Then $h(\tau)$ can be written as $\prod_{j=1}^r \left(\tau^{m_j} - 1\right)^{n_j}$, $n_j\in\mathbb{Z}$,
    and
    \[
    h^{(\ex)}(\tau)=\prod_{j=1}^r \left(\tau^{\frac{m_j}{\gcd(m_j,\ex)}} - 1\right)^{n_j\gcd(m_j,\ex)}.
    \]
\end{lemma}

The following lemma will be used to study the local factors
of $\tilde{\Delta}(\tau)$ coming from $\sing C$
in~\eqref{Deltak-LYS}. Recall Notation~\ref{ntc:polos}.

\begin{lemma}\label{lem:prop_cyclo} 
Let $h(\tau)$ be a product of cyclotomic polynomials.
\begin{enumerate}[label=\rm(\arabic{enumi})]

\item\label{rmk_cyclotomic3} If $h(\tau)=\Phi_n(\tau)$, then $h^{(\ex)}(\tau)$ is a power
of $\Phi_{\frac{n}{\gcd(n,\ex)}}(\tau)$.

\item\label{rmk_cyclotomic3e} If $a,b,c\in\NN$, $\gcd(a,c)=1$, then
$\Phi_{ab}(\tau)\mid\Phi_a(\tau^{bc})$.

\item\label{rmk_cyclotomic3c}
If $\Phi_n(\tau)$ divides $h(\tau)$
then 
$\Phi_{\frac{(\exz+\ex) n}{\gcd(n,\ex)}}(\tau)$ divides 
$h^{(\ex)}(\tau^{\exz + \ex})$. 

\item\label{rmk_cyclotomic3d}
If $n \mid m$ and if $\Phi_n(\tau)$ divides $h(\tau)$, then $\Phi_{n}(\tau)$ divides 
 $h^{(\ex)}(\tau^{\exz + \ex})$.

\item\label{rmk_cyclotomic3a} If $n$ is the order of a root of $h(\tau)$, then 
    $\frac{n}{\gcd(n,\ex)}$ is the order of a root of 
    $h^{(\ex)}(\tau)$.

    \item\label{rmk_cyclotomic3b} Fix $\rho_0 \in \QQ$ and $\ell \in \ZZ$.  If 
$\exp(-2\pi  \rho_0 i)$ is a root of $h(\tau)$, then $\exp(-2\pi \cpolo(\rho_0, \ell, \exz, \ex) i)$ is a root of $h^{(\ex)}(\tau^{\exz + \ex })$, 
and all the roots of $h^{(\ex)}(\tau^{\exz + \ex })$ are of this form.

\item\label{rmk_cyclotomic4} If $n\mid\exz$ and $\tau^n - 1$ divides $h (\tau)$, then 
$\left(\tau^{\frac{n}{\gcd(n, \ex)}} - 1\right)^{\gcd(n, \ex)}$ divides $h^{(\ex)} (\tau)$. This implies that 
$\tau^n - 1$ 
divides $h^{(\ex)} (\tau^{\exz + \ex})$.
    \end{enumerate}
\end{lemma}

\begin{proof}
The claim~\ref{rmk_cyclotomic3} is immediate from the definition
of $h^{(\ex)}(\tau)$. 

For~\ref{rmk_cyclotomic3e}, notice that $\exp(-2\pi \frac{c}{a} i)$
is a root of $\Phi_a(\tau)$ because $\gcd(a,c)=1$. Then, 
$\exp(-2\pi \frac{1}{ab} i)=\exp(-2\pi \frac{c}{abc} i)$ is a root of both, $\Phi_{ab}(\tau)$  and $\Phi_a(\tau^{bc})$. 

The statement~\ref{rmk_cyclotomic3c} 
follows from~\ref{rmk_cyclotomic3e} with $a=\frac{n}{\gcd(n,\ex)}$, $b=\exz +\ex$, and $c=1$.

For~\ref{rmk_cyclotomic3d}, take
$a = \frac{n}{\gcd(n,\ex)}$, $b = \gcd(n,\ex)$ and $c= \frac{m+k}{\gcd(n,\ex)}$, with $a$ and $c$ coprime.  By \ref{rmk_cyclotomic3e} we obtain 
that $\Phi_n(\tau)$ divides $\Phi_{a} (\tau^{bc}) = \Phi_{\frac{n}{\gcd(n,\ex)}}(\tau^{m+\ex})$, and   $\Phi_{\frac{n}{\gcd(n,\ex)}}(\tau)$ divides 
$h^{(k)} (\tau)$ by ~\ref{rmk_cyclotomic3}.

The claim~\ref{rmk_cyclotomic3a} is a direct consequence of~\ref{rmk_cyclotomic3}.

For~\ref{rmk_cyclotomic3b}, let $\zeta = \exp(-2\pi  \rho_0 i)$. Then $\zeta^\ex = \exp(-2\pi (\ex \rho_0 + \ell) i)$ is a root of  $h^{(\ex)} (\tau)$ for $\ell \in \ZZ$. Also note that, $\sqrt[\exz  + \ex ]{\zeta^\ex} = \exp(-2\pi (\frac{\ex \rho_0 + \ell}{\exz  + \ex}) i)= \exp(-2\pi \cpolo(\rho_0, \ell, \exz, \ex) i)$ is a root of $h^{(\ex)}(\tau^{\exz  + \ex })$.

Finally, suppose that $n\mid\exz$ and $g(\tau)=\tau^n - 1$. From Lemma~\ref{lema:glm},
we have that $g^{(\ex)}(\tau)=\left(\tau^{\frac{n}{\gcd(n, \ex)}} - 1\right)^{\gcd(n, \ex)}$ and the first statement
of~\ref{rmk_cyclotomic4}
follows. Since $n$ divides $\frac{n(\exz + \ex)}{\gcd(n, \ex)}$,
the second statement of~\ref{rmk_cyclotomic4} holds.
\end{proof}

\subsection{Kizuka and Kashiwara's rational pencils and their splice diagrams}\label{K-K-pencils}
\mbox{}

Let us review some facts about rational pencils that were used to prove the monodromy conjecture for SIS and will be used also in our proof for $\ex$-LYS when $\chi(\PP^2 \setminus C_{\exz})\leq 0$. 
Besides~\cite{ACNLM-ASENS}, the main references for this subsection are
{\cite[Theorem 6.1]{kashiwara}} and~{\cite[Proposition 3.2]{veys:str}}.

Assume throughout this section that $m>3$.
By a result conjectured by Veys~\cite{Veys93} and proved by de Jong and Steenbrink~\cite{MR1324440}, if 
$C$ is a curve on $\PP^2$ such that 
$\chi(\PP^2 \setminus C) \leq 0$, then all irreducible components of $C$ are rational. Moreover,  
such irreducible components lie on a pencil $\gamma$ on $\PP^2$ defined by a non-constant rational function $\gamma : \PP^2  
\dashrightarrow \PP^1$. Such a pencil may belong to one of two classes: type $(0,1)$, which were studied by Kashiwara~\cite{kashiwara}, and type $(0,2)$, studied by  Kizuka~\cite{kizuka}. Both pencils have at most two singular points and we are mainly interested in the former.

If the curve $C$ belongs to a pencil of Kizuka, its generic fibers are isomorphic to $\CC^*$ outside the base points, and have two special fibers, one irreducible with multiplicity greater or equal to one and one reduced with two irreducible components; moreover,  $\chi(\PP^2\setminus C)=0$. The properties of those pencils, relevant for our proof of Theorem~\ref{thm:CMk_LYS}, are stated in Proposition~\ref{rem:SIS_typo}. 

The rest of this section is devoted to Kashiwara's pencils. These pencils have one base point~$q$, which after a suitable change of coordinates will be $[1:0:0]$. If  $F$ is a fiber of the pencil, then $F^{\text{red}}\setminus\{q\}$
 is isomorphic to $\CC$. 
There are at most two special fibers in the pencil. In order to describe the different cases of Kashiwara's pencil, denote by $\pi_\gamma : X \rightarrow \PP^2$  the minimal resolution of the pencil $\gamma$ and denote by $\Gamma: X \rightarrow \PP^1$ the composition $\gamma\circ\pi_\gamma$, 
see \eqref{fig:diagfibration}. As a consequence of Kashiwara's classification, there is exactly one exceptional component~$E$ of $\pi_\gamma$ which
    becomes a dicritical divisor of $\Gamma$, i.e., such that 
    $\Gamma_{|E}:E\to\PP^1$ is surjective; actually, $\deg\Gamma_{|E}=1$. 
    
Figure~\ref{fig:kashiwara2}  shows a simplified version of the splice
diagrams of the resolution of the pencil $\gamma$ together with $n$ generic fibers $G_i$ and the special fibers, see~\cite{zbMATH04022161} or~\cite[Chapter 9]{Wall-04} for the basics about splice diagrams. Let us recall that the dual resolution graph
is obtained by subdividing some edges of the splice diagram.
 Note that $E$ is displayed in red in both figures.

\begin{figure}
    \centering
    \begin{tikzpicture}
\begin{scope}
\coordinate (tr) at (.25,0);
\foreach \x in {0, 1, 3, 4}
{
\coordinate (A\x) at (1.75*\x, 0);
\fill (A\x) circle [radius=.1cm];
}
\foreach \x in {0, ..., 4}
{
\coordinate (B\x) at (1.75*\x, -1);
}
\foreach \x in {1, 3}
{
\fill (B\x) circle [radius=.1cm];
\draw (A\x) -- (B\x);
}
\coordinate (G1) at ($(A0) + (-1, .5)$);
\coordinate (G2) at ($(A0) + (-1, -.5)$);
\coordinate (L) at ($(B3) + (2, 0)$);
\draw (A0) -- ($(A1) + (tr)$) ($(A3) - (tr)$) -- (A4);
\draw[dotted] ($(A1) + (tr)$) -- ($(A3) - (tr)$) (A4) -- (B4) node[right] {$L$};
\foreach \x in {1, 2}
{
\draw (A0) -- (G\x);
\filldraw[fill=white] (G\x) circle [radius=.1cm];
}
\node at ($.6*(G1) + .4*(G2)$) {$\vdots$};

\fill[red] (A0) circle [radius=.1cm];
\filldraw[fill=white] (B4) circle [radius=.1cm];

\node[left] at (G1) {$G_{1}$}; 
\node[left] at (G2) {$G_{n}$}; 
\node[above right] at (A3) {$e$};
\node[above left] at (A3) {$E_1$};
\node[below right] at (A3) {$e + 1$};
\node[above=2pt] at (A4) {$D_{1}$}; 
\node[below=3pt] at (A0) {$E$};
\node[right=1cm] at (A4) {$I_a$, $e\geq 2$};
\end{scope}
\begin{scope}[yshift=-2cm]
\coordinate (tr) at (.25,0);
\foreach \x in {0, 1, 2, 3, 4}
{
\coordinate (A\x) at (1.75*\x, 0);
\fill (A\x) circle [radius=.1cm];
}
\foreach \x in {0, ..., 4}
{
\coordinate (B\x) at (1.75*\x, -1);
}
\fill (B1) circle [radius=.1cm];
\fill (B2) circle [radius=.1cm];
\draw (A1) -- (B1);
\draw (A2) -- (B2);

\coordinate (G1) at ($(A0) + (-1, .5)$);
\coordinate (G2) at ($(A0) + (-1, -.5)$);
\coordinate (L) at ($(B3)$);
\draw (A0) -- ($(A1) + (tr)$) ($(A2) - (tr)$) -- (A4);

\draw[dotted] ($(A1) + (tr)$) -- ($(A3) - (tr)$) (A3) -- (L) node[right] {$L$};
\foreach \x in {1, 2}
{
\draw (A0) -- (G\x);
\filldraw[fill=white] (G\x) circle [radius=.1cm];
}
\node at ($.6*(G1) + .4*(G2)$) {$\vdots$};

\fill[red] (A0) circle [radius=.1cm];
\filldraw[fill=white] (L) circle [radius=.1cm];

\node[left] at (G1) {$G_{1}$}; 
\node[left] at (G2) {$G_{n}$}; 
\node[above right] at (A3) {$2$};
\node[above left] at (A3) {$E_1$};
\node[right=1cm] at (A4) {$I_b$};

\node[below=2pt] at (A4) {$D_{1}$}; 
\node[below=3pt] at (A0) {$E$};

\end{scope}

\begin{scope}[yshift=-4.5cm]
\coordinate (tr) at (.25,0);
\foreach \x in {-2, ..., 4}
{
\coordinate (A\x) at (1.75*\x, 0);
\fill (A\x) circle [radius=.1cm];
}
\foreach \x in {-1, ..., 4}
{
\coordinate (B\x) at (1.75*\x, -1);
}
\foreach \x in {1, 2}
{
\fill (B\x) circle [radius=.1cm];
\draw (A\x) -- (B\x);
}
\foreach \x in {-1,3}
{
\coordinate (B\x) at (1.75*\x, -1);
\draw[dotted] (A\x) -- (B\x);
\filldraw[fill=white] (B\x) circle [radius=.1cm];
}
\node[left] at (B-1) {$C_{p_1}$};
\node[right] at (B3) {$C_{p_2}$};
\coordinate (G1) at ($(A0) + (-.5, 1)$);
\coordinate (G2) at ($(A0) + (.5, 1)$);
\draw (A-2) -- ($(A1) + (tr)$) ($(A2) - (tr)$) -- (A4);
\draw[dotted] ($(A1) + (tr)$) -- ($(A3) - (tr)$);
\foreach \x in {1, 2}
{
\draw (A0) -- (G\x);
\filldraw[fill=white] (G\x) circle [radius=.1cm];
}
\node at ($.5*(G1) + .5*(G2)$) {$\dots$};

\fill[red] (A0) circle [radius=.1cm];

\node[above] at (G1) {$G_{1}$}; 
\node[above] at (G2) {$G_{n}$}; 

\node[above=2pt] at (A4) {$D_{1}$}; 
\node[above left] at (A3) {$E_{1}$}; 
\node[above right] at (A3) {$e_1$}; 
\node[below=3pt] at (A0) {$E$};

\node[above=2pt] at (A-2) {$D_{2}$}; 
\node[above right] at (A-1) {$E_{2}$}; 
\node[above left] at (A-1) {$e_2$}; 
\node[right=1cm] at (A4) {$II$};
\end{scope}
\end{tikzpicture}
    \caption{Kashiwara's pencils of types $I_a$, $I_b$, and $II$.}
    \label{fig:kashiwara2}
\end{figure}
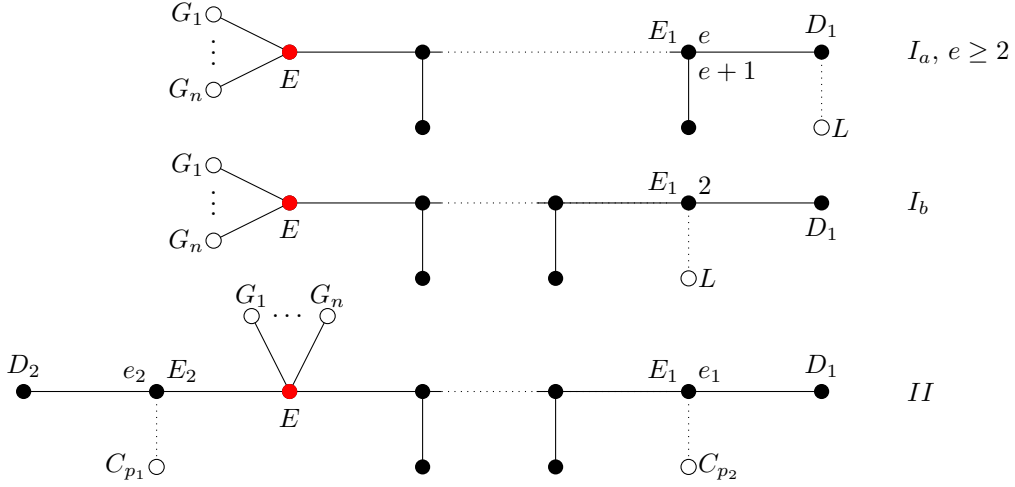

Kashiwara's pencils of type~$I$ are pencils that have at most one special fiber which is a multiple line~$L$, see Figure~\ref{fig:kashiwara2}. If there is no special fiber, it is just a pencil of lines. The irreducible components of the 
special fiber $F_L:=\Gamma^*(\gamma(L))$ correspond to the divisors to the right of~$E$. The difference between the subtypes $I_a$ and $I_b$ is due to the relative position of the special fiber $L$ and the generic fiber $G_i$. In particular, the decorations at the node $E_1$  imply that $D_1$ is the second blow-up in the case $I_a$ and the first blow-up in the case $I_b$.

Kashiwara's pencils of type~$II$ have two special fibers, see Figure~\ref{fig:kashiwara2}. 
There are two unicuspidal rational curves $C_{p_1}$ and $C_{p_2}$ of degrees~$p_1, p_2$, with $\gcd(p_1, p_2)=1$,
and $C_{p_1}\cap C_{p_2}=\{q\}$. The special fibers of $\gamma$ are $p_1C_{p_2}$ and  $p_2 C_{p_1}$. 
The irreducible components of the 
special fiber $F_1:=\Gamma^*(\gamma(C_{p_1}))$ (resp. $F_2:=\Gamma^*(\gamma(C_{p_2}))$)
correspond to the divisors on the left (resp. right) hand side of~$E$. 
If $G$ denotes a generic fiber of the pencil,
we have the relations
\begin{equation}\label{eq:relFG}
\pi_\gamma^*(p_2 C_{p_1})=F_1 + \pi_\gamma^*(G) - G
\quad \text{ and } \quad
\pi_\gamma^*(p_1 C_{p_2})=F_2 + \pi_\gamma^*(G) - G
\end{equation}
Applying  the valuation $N_{E_s}$, for an exceptional divisor $E_s$ of $\pi_\gamma$ corresponding to the left of the dicritical divisor $E$, to the divisors of  \eqref{eq:relFG}, we have that 
\begin{equation}\label{eq:multF}
    N_{E_s} (F_1) = p_2 N_{E_s} (C_{p_1}) - p_1 N_{E_s}(C_{p_2}) > 0 \quad \text{ and } \quad N_{E_s} (F_2) =0. 
\end{equation}
And  similar formulas hold for  an exceptional divisor $E_s$ of $\pi_\gamma$ corresponding to the right of the dicritical divisor $E$. Notice also that
\begin{equation}\label{eq:multG}
    N_{E_s} (G) = \min \{ p_1 N_{E_s} (C_{p_2}), p_2 N_{E_s} (C_{p_1}) \}.
\end{equation} 
For examples of Kashiwara's pencils of types $I_a$, $II_b$ and $II$, see \S\ref{sec:exam_kas_curves}.

Let us compute some numerical data associated to the pencil $\gamma$.

\begin{prop}\label{prop:kashiwara_nu} 
Let $D \subset \PP^2(\CC)$ be either a generic curve or a special curve of a Kashiwara's pencil $\gamma$ of degree $d$.  Then, for the unique dicritical divisor $E$ we have
\begin{equation}\label{eq:NEnuE}
N_E(D) = d \deg D, 
\qquad {\rm and } \qquad                   
\nu_E = 3d - 1.
\end{equation}
Let $G$ be a generic curve of the pencil $\gamma$. 

\begin{enumerate}[label=\rm(\arabic{enumi})]
    \item If $\gamma$ is of type $I_a$, we have
\begin{equation}\label{eq:ND1Ia}
N_{D_1}(G)= d 
\quad \text{ and } \quad 
N_{D_1}(L)= 2.
\end{equation}

\item If $\gamma$ is of type $I_b$, we have
\begin{equation}\label{eq:ND1Ib}
N_{D_1}(G)= \frac{d}{2}
\quad \text{ and } \quad 
N_{D_1}(L)= 1.
\end{equation}

\item If $\gamma$ is  of type $II$ and $i,j \in \{1,2\}$ with $i \ne j$, we have
\begin{equation}\label{eq:ND12II}
N_{D_1}(G)=\frac{p_1p_2^2}{e_1}, 
\quad 
N_{D_2}(G)=\frac{p_1^2p_2}{e_2},
 \quad 
 N_{D_i}(C_{p_i})= \frac{d}{e_i}
\quad \text{ and } \quad
 N_{D_i}(C_{p_j})= \frac{p_j^2+1}{e_i}.
\end{equation}

\end{enumerate}
\end{prop}
\begin{proof}We start with the proof of \eqref{eq:NEnuE}.
Let us denote by $G$ a generic curve of the pencil, different than $D$.   
Recall  $G$ and $D$ 
are unicuspidal rational plane curves with $q$ as unique point of  intersection. Hence, 
 we have 
 \[
(G \cdot D)_q = (G \cdot D)_{\PP^2} = \deg G \cdot \deg D= d\cdot\deg D.
\]
Let us recall that the minimal resolution $\pi_\gamma :X\to\PP^2$ of the pencil $\gamma$ is also a resolution of~$G + D$.  
In order to do not overload the notation we do not change the name for the
strict transforms and we will point out the ambient surface where intersection
numbers are computed. 
We are going to use extensively the following fact. 
If $A$ is an irreducible curve such that
its strict transform by $\pi_\gamma$ is a curvette intersecting a divisor $E_A$,
and $B$ is a curve such that strict transforms of $A,B$ in $X$ are disjoint,
then,
\begin{equation}\label{eq:inter_pi}
    (A\cdot B)_{\PP^2} = (A\cdot B)_{q}= (\pi_\gamma^*(A)\cdot \pi_\gamma^*(B))_X=
    (A \cdot \pi_\gamma^*(B))_X = (A\cdot N_{E_A}(B) E_A)_X= N_{E_A}(B).
\end{equation}
If we consider $A=G$, we have $E_A=E$, see Figure~\ref{fig:kashiwara2}. Then, we apply \eqref{eq:inter_pi} with $B=D$ to get
  \begin{equation}\label{eq:NED}
d \cdot\deg D =(G\cdot D)_{\PP^2} =N_E(D),
\end{equation}
and the first part of~\eqref{eq:NEnuE} is proved.
For the second part, we consider the $\QQ$-canonical divisor $K_{\PP^2}=-\frac{3}{d} G$
of $\PP^2$. 
We recall that the multiplicity of the relative canonical divisor
$K_{X/\PP^2}$
at $E$ is $\nu_E-1$. The canonical divisor $K_{\PP^2}$ induces a $\QQ$-canonical divisor $K_X$ of $X$, where 
\[
K_X=-\frac{3}{d}\pi_\gamma^*(G) + K_{X/\PP^2}=-\frac{3}{d} G+ \left(-\frac{3}{d} N_E(G) +\nu_E  - 1\right) E + H,
\]
and $H$ is a divisor in~$X$
not containing neither $G$ nor $E$ in its support. 

Using the adjunction formula, and taking into account that the self-intersection of $G$ is $0$, since any two fibers of $\gamma$ are linearly equivalent,  we have
\[
-2 = (G \cdot (G + K_X))_X = -\frac{3}{d} N_E(G) +\nu_E  - 1=
-3d+\nu_E-1,
\]
using $N_E(G)=d^2$ because of \eqref{eq:NED}. 
Hence, we conclude $\nu_E= 3d - 1,$
and~\eqref{eq:NEnuE} is done.

For the case~$I_a$, we apply \eqref{eq:inter_pi} to $A=L$ ($E_A=D_1$) and $B=G$:
\[
N_{D_1}(
G) = (L\cdot G)_{\PP^2} = d.
\]
Moreover, $N_{D_1}(L)=2$, because $D_1$ corresponds to the second blow-up.

For the case $I_b$, we apply \eqref{eq:inter_pi} to $A=L$ ($E_A=E_1$) and $B=G$, giving
\[
N_{E_1}(
G
)=(
G
\cdot L)_q = (
G
\cdot L)_{\PP^2} = d. 
\]
By properties of splice diagrams, we have $N_{E_i}(D)= e_i N_{D_1}(D)$ for $D$ any generic or special curve of a Kashiwara's pencil. 
In this case, $e_i = 2$ hence $N_{D_1}(G)= \frac{d}{2}$ and conclude  $2\mid d$.

In this case $D_1$ corresponds to the first blow-up, and $N_{D_1}(L)=1$. 

For the case $II$, note that $d=\deg G=p_1\cdot p_2$. We have several cases. For $A=C_{p_2}$, $B=C_{p_1}$, and $E_A=E_1$, we apply \eqref{eq:inter_pi}:
\begin{equation*}
    N_{E_1}(C_{p_1}) = (C_{p_2}\cdot C_{p_1})_{\PP^2} = p_1\cdot p_2.
\end{equation*}
Interchanging the roles of $C_{p_1}$ and $C_{p_2}$ we obtain also
$N_{E_2}(C_{p_2}) = p_1\cdot p_2$.

For $A=C_{p_2}$, $B=G$, we apply \eqref{eq:inter_pi}:
\begin{equation}\label{eq:NE1G}
    N_{E_1}(G) = (C_{p_2}\cdot G)_{\PP^2} = p_1\cdot p_2^2,
\end{equation}
and similarly $N_{E_2}(G) = p_1^2\cdot p_2$.

In order to compute $N_{E_2}(C_{p_1})$ and $N_{E_1}(C_{p_2})$ we consider
the special fibers $F_1,F_2$ of $\Gamma$ defined by $C_{p_1},C_{p_2}$, i.e.,  $F_i:=\Gamma^*(\gamma(C_{p_i}))$. The multiplicity of $C_{p_2}$ in $F_2$ is $p_1$, because $p_1C_{p_2}$ is a curve in $\gamma$; with a similar argument
the multiplicity of $C_{p_1}$ in $F_1$ is $p_2$.

We are going to use several classic facts. First, any two fibers of $\Gamma$
are linearly equivalent. Let $B$ be an irreducible component of $F_1$. Note
that $B\cap F_2=\emptyset$. Then,
\begin{equation}\label{eq:BFnull}
0 = (B\cdot F_2)_X = (B\cdot F_1)_X.
\end{equation}
A similar argument is true for the irreducible components of $F_2$.

In the proof of Noether's Lemma in~\cite[pp. 513--514]{gh:78}, it is proved
that in a reducible fiber of a morphism with generic fiber~$\PP^1$, there
is always a smooth rational $(-1)$-curve. This applies to $F_1, F_2$.
The irreducible components of $F_i$ which are exceptional components of $\pi_\gamma$
have self-intersection number $<-1$, since the only component with self-intersection
number $-1$ is the dicritical component~$E$. Then $(C_{p_i})^2_X=-1$. 
Combining these facts, let $a_{12}$ be the multiplicity of $E_1$ in $F_2$:  
\[
0= (C_{p_2}\cdot F_2)_X=(C_{p_2} \cdot ( a_{12} E_1 + p_1  C_{p_2}))_X = a_{12} - p_1,
\]
where the first equality is due to \eqref{eq:BFnull} and the third one is because of $(C_{p_i})^2_X=-1$.
Hence,   $a_{12}= p_1$. 

We finish with the next following fact 
which is a consequence of~\eqref{eq:multF} and~\eqref{eq:multG}:
\[
N_{E_1}(p_1C_{p_2}) = N_{E_1}(G) + p_1 = p_1 (p_2^2 + 1),
\]
where the last equality follows from \eqref{eq:NE1G}.
Then $N_{E_1}(C_{p_2}) = p_2^2 + 1$ and similarly 
$N_{E_2}(C_{p_1}) = p_1^2 + 1$.
Finally, we use again that $N_{E_i}(D)= e_i N_{D_1}(D)$ for $D$ any generic or special curve of Kashiwara's pencil of type $II$. 
\end{proof}

\begin{cor}\label{cor:einot3}
If $\gamma$ is a pencil of type $II$, we have 
$e_1 > 1$, and $e_1$ divides  $p_1, p_2^2 +1$. 
Moreover $e_1$ does not divide~$3$. And the same relations hold interchanging the subindices $1$ and $2$.  
\end{cor}
\begin{proof}
Notice that $e_1$ divides both $p_1$ and 
$p_2^2 + 1$ due to the formulas for $N_{E_1}(G) = p_1p_2^2= e_1N_{D_1}(G)$, and $N_{E_1} (C_{p_2}) = p^2_2+1=e_1N_{D_1}(C_{p_2})$ in the proof of
Proposition~\ref{prop:kashiwara_nu}. 

Moreover, assume now that $e_1$ divides~$3$. Since $3$ is prime and $e_1>1$, we have   $e_1=3$.  Since $e_1$ divides $p_1$ which is coprime with $p_2$, we have 
$p_2\not\equiv 0\bmod{3}$. Hence, 
$p_2^2+1\equiv -1\bmod{3}$, which contradicts that $e_1$ divides $p_2^2+1$. Hence, $e_1$ does not divide~$3$. 
\end{proof}

\subsection{Curves $C_m$ corresponding to a Kaswhiwara's pencil}\label{Sec:Cm}
\mbox{}

We assume in this section 
that the projectivized tangent cone $C_m \subset \PP^2$ of the $\ex$-LYS  defined by $F$ is a curve of degree 
$\deg C_\exz=\exz$
with $r \geq 2$ irreducible components of the curve all of them reduced fibers of a Kashiwara's pencil of degree $d$. We denote by $n$ the number of components of $C_m$ which are generic fibers of the pencil. 
Recall that the components of $C_m$ are rational  curves which pass through the unique singular point $q$ of $C_m$. Therefore
$\chi(\PP^2 \setminus C_{\exz}) = 2- r  \leq 0$. 
We are mainly interested in the case $\chi(\PP^2 \setminus C_{\exz}) <0$.

\begin{notation}\label{not:listcm}
There are 8 possibilities for $C_\exz $ that we will denoted as:
\begin{itemize}
    \item[${[I_a]}$] $C_\exz$ consists of just $n=r$ generic fibers $G_i$  of a pencil of type $I_a$.\item[${[I_a^L]}$] $C_\exz$ consists of $n=r-1$ generic fibers $G_i$ and the special fiber $L$ of a pencil of type $I_a$.\item[${[I_b]}$]  $C_\exz$ consists of just $n=r$ generic fibers $G_i$  of a pencil of type $I_b$.
    \item[${[I_b^L]}$] $C_\exz$ consists of $n=r-1$ generic fibers $G_i$ and the special fiber $L$ of a pencil of type $I_b$.
    \item[${[II]}$] $C_\exz$ consists of just $n=r$ generic fibers $G_i$ of a pencil of type $II$.

    \item[${[II^{1}]}$] $C_\exz$ consists of  $n=r-1$ generic fibers $G_i$ and the  special fiber $C_{p_1}$ of a pencil of 
    type~$II$.
    \item[${[II^{2}]}$] $C_\exz$ consists of  $n=r-1$ generic fibers $G_i$ and the  special fiber $C_{p_2}$ of a 
    pencil of type~$II$.
     \item[${[II^{1,2}]}$] $C_\exz$ consists of  $n=r-2$ generic fibers $G_i$ and the two special fibers $C_{p_1}$ and 
     $C_{p_2}$ of a pencil of type $II$.
\end{itemize}
\end{notation}
For examples of curves $C_m$ corresponding to Kashiwara's pencils of different types, see \S\ref{sec:exam_kas_curves}.

\begin{defn}\label{def:residual}
    We say that the curve $C_\exz$ obtained as union of fibers of a Kashiwara's pencil
    of type either $I_b$ or $II$ is \emph{residual} if not all the
    special fibers of the pencil are components of $C_\exz$.
   \end{defn}

\begin{remark}
    The curve $C_\exz $ is residual exactly in the cases  ${[I_b]}$, ${[II]}$, ${[II^{1}]}$,  and ${[II^{2}]}$.
\end{remark}

By Proposition \ref{prop: char-klys} the characteristic polynomial $\Delta$ of the monodromy $\Psi$ of the $\ex$-LYS  defined by $F$ is equal to 
\begin{equation}\label{Deltak-LYS-K}
\Delta(\tau)= \frac{1}{(\tau^{\exz} -1)^{r-2} (\tau -1)} 
 \Delta^{(\ex)}_q (\tau^{\exz + \ex }), 
\end{equation}
where $\Delta_q(\tau)$ denotes the 
characteristic polynomial of  $(C_\exz, q)$.
Recall that  $\Delta_q(\tau)$ can be computed from the splice diagram of the singularity $(C_m,q)$ using A'Campo's formula.

We introduce some notations to give the  
\emph{geometric} factorization  \eqref{eq:deltaEv} 
of $\Delta_q(\tau)$ as a product of polynomials in $\QQ[\tau]$.

\begin{notation}\label{defn:factorsdeltaq} 
We set
\begin{equation}\label{eq:factor_E}
\Delta_E(\tau):=
\begin{cases}
\dfrac{(\tau-1)(\tau^{N_E}-1)^{n}}{\Delta_{D_1}(\tau)\Delta_{D_2}(\tau)}&\text{if }C_\exz\text{ is of type }II,\\   
&\\
\dfrac{(\tau-1)(\tau^{N_E}-1)^{n - 1}}{\Delta_{D_1}(\tau)}&\text{if }C_\exz\text{ is of type }I,
\end{cases}
\end{equation}
where
\[
\Delta_{D_1}(\tau)=
\begin{cases}
\tau^{N_{D_1}}-1&\text{ if }C_{p_2} \not\subset C_\exz \text{ or }L\not\subset C_\exz,\\
1&\text{otherwise},
\end{cases}
\qquad
\Delta_{D_2}(\tau)=
\begin{cases}
\tau^{N_{D_2}}-1&\text{ if }C_{p_1}\not\subset C_\exz,\\
1&\text{otherwise,}
\end{cases}
\]
and the numbers $N_E,N_{D_1}, N_{D_2}$ are the multiplicities of $C_\exz$ at $E,D_1$ and $D_2$.

Denote by $B$ the set of branching components of the splice diagram of the curve $C_\exz $, and by $B'$ the complementary of $\{E\}$ in $B$, i.e., $B = B' \sqcup \{E\}$. 
For $v \in B'$, the \emph{$v$-polynomial} $\Delta_v(\tau)$ is determined by the equalities
\begin{equation}\label{eq:deltaEv}
\Delta_v(\tau)=\frac{\tau^{n_v}-1}{\tau^{n_w}-1}\quad \text{ and }\quad
\Delta_q(\tau)= \Delta_E(\tau) \prod_{v \in B'} \Delta_v(\tau),
\end{equation}
where  $w$ is a valence one vertex of the splice diagram attached to $v$ 
and $n_v$, $n_w$ are the corresponding multiplicities.
\end{notation}

For most indices $v\in B'$, the first equality of \eqref{eq:deltaEv} determine~$w$ since only
one vertex of valence~$1$ is attached to $v$. The second
equality is enough to make the right choice when there are two vertices
of valence~$1$ attached to~$v$.

\begin{remark} \label{rem-polynomials}
The term $\Delta_{E} (\tau)$ of the  factorization \eqref{eq:deltaEv}, which was introduced  in Definition \ref{defn:factorsdeltaq} as a rational function, is actually a polynomial
by Lemma~\ref{lema:deltaEphi}\ref{lema:deltaEphi1}.
The factors of the form  $\Delta_v(\tau)$ which appear in \eqref{eq:deltaEv} are polynomials thanks to Remark~\ref{rmk_cyclotomic}\ref{rmk_cyclotomic2}.
\end{remark}

By Definition~\ref{defn:factorsdeltaq} 
and \eqref{Deltak-LYS-K} we have 
the following formula for the polynomial 
$\tilde{\Delta} (\tau) = (\tau -1) \Delta (\tau) $ (see Notation \ref{notation: tildeD}):
\begin{equation}
\label{eq:global_local}
\tilde{\Delta}(\tau)=
\frac{\Delta_E^{(\ex)}(\tau^{\exz+\ex})}{(\tau^\exz - 1)^{r - 2}}
\prod_{v\in B'}\Delta_v^{(\ex)}(\tau^{\exz+\ex}).
\end{equation}

\begin{notation} 
The  terms $\Delta_v^{(\ex)}(\tau^{\exz+\ex})$ are called the \emph{$v$-local factors}, for $v\in B'$.
The term 
\[
\frac{\Delta_E^{(\ex)}(\tau^{\exz+\ex})}{(\tau^\exz - 1)^{r - 2}}
\]
is called the \emph{$E$-global factor} of $\Delta(\tau)$.
\end{notation}

The following lemma has three goals. The first one
is to prove that both $\Delta_E$ and the $E$-global factor are polynomials.
The second one is to show that a big enough power of the cyclotomic polynomial $\Phi_{\exz}$ divides $\Delta_E(\tau)$. The third one is to prove that for
the so-called residual curves, $\exp(-2\pi\frac{3}{m} i)$ is a root
of the $E$-global factor.
This lemma will be essential in the proof of 
Proposition~\ref{rem:SIS_typo} below. 
The ideas behind these facts came from~\cite[\S3.1 and 3.2]{ACNLM-ASENS} but we include them here for completeness.

\begin{lemma}\label{lema:deltaEphi} With the notation discussed above,
\begin{enumerate}[label=\rm(\alph{enumi})]
    \item\label{lema:deltaEphi1}  $\Delta_E(\tau)$  is a polynomial. 
    In addition, 
    $(\tau^\exz - 1)^{r - 2}$ divides $\Delta_E(\tau)$
    and $\Delta_E^{(\ex)}(\tau^{\exz + \ex})$;
    \item\label{lema:deltaEphi2} if $r > 2$, there  exists a proper positive multiple
$\tilde{\exz}$ of $m$
such that 
    $\Phi_{\tilde{\exz}} (\tau)$ divides the polynomial $\Delta_E^{(\ex)}(\tau^{\exz + \ex}) $;
    \item\label{lema:deltaEphi3} if $C_\exz$ is residual and $\ell:=\frac{\exz}{\gcd(\exz, 3)}$,
    then $\Phi_{\ell}^{r-1} (\tau)$ divides $\Delta_E(\tau)$ and 
$\Delta_E^{(\ex)}(\tau^{\exz + \ex})$.
\end{enumerate}
\end{lemma}

\begin{proof}
We compute first $\Delta_E(\tau)$, see \eqref{defn:factorsdeltaq} for all the cases introduced in Notation~\ref{not:listcm}. We use the Eisenbud-Neumann diagrams of the corresponding pencils, see Figure~\ref{fig:kashiwara2}, and the numerical dat{\ae} computed in Proposition~\ref{prop:kashiwara_nu}. In particular, 
we have that 
\begin{equation}\label{eq:NECm}
N_E = N_E(C_m) = d m,
\end{equation} since $C_m$ has $r$
components of degree $d$.

\begin{description}[wide, leftmargin=0pt]
    \item[Case ${[I_a]}$] We have $n=r$ and  $\exz = n d$, since $C_\exz$ comes from a pencil of type $I_a$ and $L\not\subset C_\exz$. 
    Hence, from Proposition~\ref{prop:kashiwara_nu}, we have 
    \[
    N_{D_1}(C_\exz)=n N_{D_1}(G) = n d = \exz.
    \]
In this case, using \eqref{eq:factor_E}, we have
\[
\Delta_E(\tau)=
\frac{(\tau-1)(\tau^{\exz d}-1)^{r - 1}}{\tau^{\exz}-1}=
(\tau-1)(\tau^{\exz d}-1)^{r - 2}
\frac{\tau^{\exz d}-1}{\tau^{\exz}-1}.
\]

\item[Case ${[I_a^L]}$] We have 
$n=r - 1$ and $\exz = n d + 1$, since $C_\exz $ comes from a pencil of type $I_a$ and $L\subset C_\exz$.   Hence, from Proposition~\ref{prop:kashiwara_nu}, we have
\[
\Delta_E(\tau)=
(\tau-1)(\tau^{\exz d}-1)^{r - 2}.
\]

\item[Case ${[I_b]}$] We have $n=r$ and $\exz = n d$, since $C_\exz $ comes from a pencil of type $I_b$ and $L\not\subset C_\exz$.
Hence, from Proposition~\ref{prop:kashiwara_nu}, we have 
$N_{D_1}=n N_{D_1}(G) = n\frac{d}{2}=\frac{\exz}{2}$, which implies
\[
\Delta_E(\tau)=
\frac{(\tau-1)(\tau^{\exz d}-1)^{r - 1}}{\tau^{\frac{\exz}{2}}-1}=
(\tau-1)(\tau^{\exz d}-1)^{r - 2}\frac{\tau^{\exz d}-1}{\tau^{\frac{\exz}{2}}-1},
\]
It is clear
that $\Phi_{\ell}^{r-1}(\tau)\mid\Delta_E(\tau)$.

\item[Case~${[I_b^L]}$] 
This case works like case ${[I_a^L]}$.

\item[Case ${[II]}$] 
We have $n=r$,   $\exz = n p_1 p_2$, and $d = p_1 p_2$.  Hence, from Proposition~\ref{prop:kashiwara_nu}, we have 
$N_{D_1}=\frac{n p_1 p_2^2}{e_1}$ and $N_{D_2}=\frac{n p_1^2 p_2}{e_2}$, and
\[
\Delta_E(\tau)=
\frac{(\tau-1)(\tau^{\exz d}-1)^{r}}{(\tau^{N_{D_1}}-1)(\tau^{N_{D_2}}-1)}=
(\tau-1)(\tau^{\exz d}-1)^{r-2}\frac{\tau^{\exz d}-1}{\tau^{N_{D_1}}-1}
\frac{\tau^{\exz d}-1}{\tau^{N_{D_2}}-1}.
\]

Since $e_1,e_2>1$, we have that 
$N_{D_1}$ and $N_{D_2}$
divide $md$, therefore
all the factors above are polynomials.

Since $p_1$ and $p_2$ are coprime and $e_1 | p_1$, $e_2| p_2$ by Corollary \ref{cor:einot3}, it follows that 
$\exz$ does not divide neither 
$N_{D_1}$ nor $N_{D_2}$.
This implies that
$\Phi_{\exz}(\tau)^{r-1}\mid\Delta_E(\tau)$. 

\begin{itemize}
\item  If $\gcd(\exz, 3)=1$,  i.e., $\ell = m$,  we have that 
$\Phi_{\ell}(\tau)^{r-1}\mid\Delta_E(\tau)$.

\item If $3\mid\exz$, we need to check that 
$\frac{\exz}{3}$ is not simultaneously a divisor of
$N_{D_1}$ and $N_{D_2}$. The statement $\frac{\exz}{3}\mid N_{D_1}$ is equivalent
to $e_1n p_1 p_2 \mid 3n p_1 p_2^2$. 
By Corollary \ref{cor:einot3}, we have 
that $e_1 | p_2^2 +1$. We obtain $e_1\mid 3$,
and this case has been excluded by Corollary~\ref{cor:einot3}. 
Similarly, the case $\frac{\exz}{3}\mid N_{D_2}$ cannot happen.
Then, $\Phi_{\ell}(\tau)^{r-1}\mid\Delta_E(\tau)$.
\end{itemize}
\item[Case ${[II^1]}$] 
We have $n=r - 1$ and  $\exz = n d + p_1 = p_1(n p_2 + 1)$. We also have
\[
N_{D_1}=n N_{D_1}(
G ) 
+ N_{D_1}(C_{p_1}) =
\frac{p_1}{e_1} p_2 (n p_2 + 1).
\]
Then
\[
\Delta_E(\tau)=
\frac{(\tau-1)(\tau^{\exz d}-1)^{r - 1}}{\tau^{N_{D_1}}-1}=
(\tau-1)(\tau^{\exz d}-1)^{r - 2}\frac{\tau^{\exz d}-1}{\tau^{N_{D_1}}-1}.
\]
It follows that $N_{D_1}| md$ and $m \nmid N_{D_1}$, since $\gcd(p_1, p_2)=1$, $e_1 >1$ and $e_1| p_1$ (see Corollary Corollary~\ref{cor:einot3}).
We have again 
that
$\Phi_{\exz}(\tau)^{r-1}\mid\Delta_E(\tau)$, i.e., 
$\Phi_{\ell}(\tau)^{r-1}\mid\Delta_E(\tau)$
if $\gcd(\exz, 3)=1$.

For $3\mid\exz$, we need to check
$\frac{\exz}{3}$ is not a divisor of
$N_{D_1}$.
As in the previous case, the condition
$\frac{\exz}{3} \mid N_{D_1}$
implies that $e_1\mid 3$, which contradicts Corollary~\ref{cor:einot3}. 
Hence, $\Phi_{\ell}(\tau)^{r-1}\mid\Delta_E(\tau)$.

\item[Case ${[II^2]}$] 
This case works like the previous one.

\item[Case ${[II^{1,2}]}$] We have $n=r - 2$, $\exz = n d + p_1 + p_2$, and
\[
\Delta_E(\tau)=
(\tau-1)(\tau^{\exz d}-1)^{r - 2}.
\]
\end{description}
We obtain the following common properties. 
First, $\Delta_E(\tau)$ is a polynomial and $(\tau^{md} -1)^{r-2}$ divides it.
By Lemma~\ref{lem:prop_cyclo}\ref{rmk_cyclotomic4}, the polynomial 
$(\tau^\exz - 1)^{r-2}$ divides $\Delta_E^{(\ex)}(\tau^{\exz + \ex})$ and~\ref{lema:deltaEphi1} is proved.

In addition, if $C_m$ is residual and $\ell=\frac{\exz}{\gcd(\exz,3)}$, we have shown that $\Phi_{\ell}(\tau)^{r-1}\mid\Delta_E(\tau)$.
For~\ref{lema:deltaEphi3}, notice that
$\Phi_{\ell}^{r-1}(\tau)$,  divides 
$\Delta_E^{(\ex)}(\tau^{\exz + \ex})$, by Lemma~\ref{lem:prop_cyclo}\ref{rmk_cyclotomic3d}.

If $r > 2$, we have proved that $\Phi_{d\exz}(\tau)$ divides $\Delta_E(\tau)$. Using Lemma~\ref{lem:prop_cyclo}\ref{rmk_cyclotomic3}, 
we have that $\Phi_{\frac{d\exz}{\gcd(d\exz, \ex)}}(\tau)$ divides $\Delta_E^{(\ex)}(\tau)$.
Using Lemma~\ref{lem:prop_cyclo}\ref{rmk_cyclotomic3c} we have that 
$\Phi_{\frac{d\exz(\exz + \ex)}{\gcd(d\exz, \ex)}}(\tau)$ divides $\Delta_E^{(\ex)}(\tau^{\exz + \ex})$.
Note that
\[
\tilde{m}:=\frac{d\exz(\exz + \ex)}{\gcd(d\exz, \ex)}=
\exz\frac{d\exz}{\gcd(d\exz, \ex)}+\exz\frac{d\ex}{\gcd(d\exz, \ex)}
\]
is a proper positive multiple of~$\exz$ and~\ref{lema:deltaEphi2} holds.
\end{proof}

The following Lemma is a direct consequence of Proposition~\ref{prop:kashiwara_nu}  about the candidate pole $\cpolo(\rho_E, 3, \exz, \ex)$ of $\Ztop{}(F,s)_\zero$ for a  $\ex$-LYS of surface
defined by $F=0$ coming from the dicritical component $E$ associated with a singular point $q \in \sing C_\exz $, see Lemma~\ref{cor:canpolekLYS}.

\begin{lemma}\label{lema:cyclo2}
Let $-\rho_E$ be the pole of $\Ztop{}(f_q,s)_\zero$
coming from~$E$. Then, $\exz \frac{\ex\rho_E + 3}{\exz + \ex}\notin\ZZ$, or equivalently 
$m\cpolo(\rho_E, 3, \exz, \ex))\notin\ZZ$,
 using Notation{\rm~\ref{ntc:polos}}.
\end{lemma}

\begin{proof}
By \eqref{eq:NECm} we have $N_E= d\exz$.
Recall that 
$\rho_E= \frac{\nu_E}{N_E}=\frac{3d-1}{\exz d}$ and $- \rho_E$  is a pole of $\Ztop{}(f_q,s)_\zero$, see~\cite{veys:mon1}. 
Hence,  
\[
\exz\frac{\ex \rho_E + 3}{ \exz +\ex}= \exz\frac{\ex\frac{3 d - 1}{\exz  d} + 3}{ \exz +\ex}=
\frac{\ex (3 d - 1) + 3\exz d}{d (\exz +\ex)}=3 - \frac{\ex}{d(\exz +\ex)}\notin \ZZ.
\qedhere
\]
\end{proof}

\subsection{Bad divisors}\label{Sec:Bad_div}
\mbox{} 

Let us recall the notion of bad divisor introduced in~\cite[Definition 2.3]{ACNLM-ASENS}.

\begin{defn}
    A curve~$C\subset\PP^2$ is a \emph{bad divisor} if $\exz:=\deg C>3$, 
    $\chi(\PP^2\setminus C)\leq 0$, and $-\frac{3}{\exz}$ is not 
    a pole of $\Ztop{}(f_q,s)_q$ for any $q\in C$.
\end{defn}

The motivation for the notion of bad divisor is the following. If the projectivized tangent cone of a $\ex$-LYS of surface is a bad divisor, it is particularly difficult to check that $\exp(-2 \pi \frac{3}{\exz} i)$ is an eigenvalue of $F$ at $0$. 
As a consequence of Theorem~\ref{thm:DLZFkLY}\ref{ztopklis1} and 
Corollary \ref{cor:canpolekLYS}
the following statement holds: The candidate pole  $s_0= -\frac{3}{\exz}$ of $\Ztop{}(F, s)_\zero$ has multiplicity at most one if and only if it is not a pole of $\Ztop{}(f_q, s)_\zero$ for $q\in\sing C$. Hence, if  the projectivized tangent cone $C_\exz$ of a $\ex$-LYS of surface
defined by $F=0$ is a bad divisor, 
then the multiplicity of $s_0= -\frac{3}{\exz}$
as pole of $\Ztop{}(F, s)_\zero$ is at most~$1$.

\begin{lemma}\label{lemma:rhoC}
If the projectivized tangent cone $C_\exz$ of $F$ is a bad divisor, then
the residue of $\Ztop{}(F, s)_\zero$ at $-\frac{3}{\exz}$ equals
    $
    \frac{1} {\exz}  \res(C_\exz)
    $
    where
    \[
    \res(C_\exz)=\chi(\PP^2\setminus C_\exz) + \frac{\exz }{\exz  - 3}\chi(C_\exz \setminus\sing C_\exz)+
    \sum_{q\in\sing C_\exz}\Ztop{}\left(f_q,-\frac{3}{\exz }\right)_\zero.
    \]
\end{lemma}

\begin{proof}  To compute the residue
it is enough to multiply the formula in Theorem~
\ref{thm:DLZFkLY}\ref{ztopklis1}
by $ s+\frac{3}{\exz }=
\frac{\ex}{\exz}(r - s)$, and evaluate it  at $s= -\frac{3}{\exz }$.
\end{proof}

\begin{remark}
     Compare the above computation of $\res(C)$  with the computation of $\rho(C)$ in~\cite[Proposition 2.1(v)]{ACNLM-ASENS}. Notice that
    the condition $-\frac{3}{\exz }$ is a pole
    of $\Ztop{}(F, s)_\zero$ is independent of the value of $\ex \geq 1$.
\end{remark}

The case where $\res(C)=0$  is not relevant
for the monodromy conjecture, because  then $-\frac{3}{\exz}$
is not a pole of $\Ztop{}(F,s)_\zero$. 
The following result, based in~\cite[Theorem 2.15]{ACNLM-ASENS}, helps to disregard Kizuka's pencils and
reduce the scope in the analysis of Kashiwara's pencils, 
recall Definition{\rm~\ref{def:residual}}.

\begin{prop}\label{rem:SIS_typo}
Let $C_\exz$ be a bad divisor such that $\res(C_\exz)\neq 0$.
Then, 
 $C_\exz$ is a residual curve.
\end{prop}

The proof of this proposition is given at the end of the next subsection (page~\pageref{subsec:tp8}), see also Remark~\ref{rmk_veys_construction}.

\subsection{Proofs of Theorem{\rm~\ref{thm:CMk_LYS}}
and Proposition{\rm~\ref{rem:SIS_typo}}}\label{Sec:Proof_Mon_Conj}
\mbox{}

Let us discuss now the proof of Theorem~\ref{thm:CMk_LYS}. It follows a similar strategy than the proof of the Monodromy Conjecture for SIS of~\cite[Theorem 3.1]{ACNLM-ASENS}.

\begin{proof}[Proof of Theorem{\rm~\ref{thm:CMk_LYS}}] If the projectivized tangent cone $C_m \subset \PP^2$ of the germ $F$ is smooth, there are only two candidate poles, $-1$ and 
$- \frac{3}{\exz}$,
 by Corollary~\ref{cor:canpolekLYS},
and $\chi (\PP^2 \setminus C_m) = 1 + (m-1)(m-2) > 0$. Hence $(\tau - 1) \Delta(\tau)$ reduces to its global factor $(\tau^m -1)^{\chi (\PP^2 \setminus C_m)}$ which is a polynomial. It follows that the candidate 
$- \frac{3}{\exz}$
induces an eigenvalue of the monodromy at $\zero \in F^{-1}(0)$. The candidate $-1$ always induces the eigenvalue 1 at a smooth point of $F^{-1}(0)$.

Now we consider the case where the projectivized tangent cone $C_m$ is singular. The case where the multiplicity of $F$ is $\exz \leq 3$ and $\chi(\PP^2\setminus C_\exz) \leq 0$ is settled in Example~\ref{ex:m<4}. We assume for now $\exz>3$ or $\chi(\PP^2\setminus C_\exz)>0$.

One needs to check that $\exp(-2\pi a i)$ is an eigenvalue of the monodromy of $F$ at some point  of $F^{-1}(0)$ for $a$ in the set
\[
\left\{
1, \frac{3}{\exz } 
\right\}
\cup
\left\{
\cpolo(\rho_0)\,
\middle| \,
\rho_0\in\pol(f_q,\omega_d) \text{ and } q \in \sing C
\right\},
\]
where $\cpolo(\rho_0)=\cpolo(\rho_0, 3,\exz,\ex) 
=\frac{k \rho_0 +3 }{\exz+k}$,
see Notation~\ref{ntc:polos}.

We distinguish several cases.
\begin{enumerate}[label=\rm(P\arabic{enumi})]
\item\label{pole1} $a=1$.
\item\label{pole2} $a=\cpolo(\rho_0)$ and $\exz a \not \in \ZZ$.
\item\label{pole3} $a=\cpolo(\rho_0)$, $\exz a  \in \ZZ$, and $\chi(\PP^2 \setminus C_\exz) \geq 0$.
\item\label{pole4} $a=\cpolo(\rho_0)$, $\exz a  \in \ZZ$, and $\chi(\PP^2 \setminus C_\exz) < 0$.
\item\label{pole5} $a = \frac{3}{\exz}$ and $\chi(\PP^2 \setminus C_m) > 0$.
\item\label{pole6} $a = \frac{3}{\exz}$ is a multiple pole and $\chi(\PP^2 \setminus C_\exz) \leq  0$.
\item\label{pole7} $a = \frac{3}{\exz}$ is a single pole   and $\chi(\PP^2 \setminus C_\exz) \leq  0$.
\end{enumerate}

Before discussing each case, let us note that case~\ref{pole4} does not appear in the proof of the Monodromy Conjecture for SIS ($\ex =1$). In fact, for  $\ex =1$, and using numerical properties of the poles of plane curves,  it is easy to prove that if $\exz \cpolo(\rho_0) \in \ZZ$ for  a candidate pole $\rho_0$ of $\Ztop{}(f_q,s)_\zero$ for some $q \in \sing C_\exz$, then $\cpolo(\rho_0) = \frac{3}{\exz}$, see~\cite[Proposition 2.1]{ACNLM-ASENS}. Note also that the the  cases~\ref{pole3} and~\ref{pole5} cover the germs $F$ with multiplicity  $\leq 3$ and $\chi(\PP^2\setminus C_\exz) > 0$.

The case~\ref{pole1}  is clear since $1$ is always a root of the monodromy at a smooth point 
of~$F^{-1}(0)$. 

The case~\ref{pole5} is also easy. In fact, $\exp(-2\pi a i)$ is a root of the global factor of $\Delta(\tau)$, which is a polynomial due to $\chi(\PP^2 \setminus C_m) > 0$, 
see Remark~\ref{rmk_cyclotomic}\ref{rmk_cyclotomic0}.

For the cases~\ref{pole2} and~\ref{pole3}, notice that $\rho_0\in\pol(f_q,\omega_d)$. Since the Monodromy
Conjecture holds for curves,  $\exp(-2\pi  \rho_0 i)$ is a root of $\Delta_q(\tau)$. Lemma~\ref{lem:prop_cyclo}\ref{rmk_cyclotomic3b} applied to  
$\Delta_q(\tau)$ implies that 
$\exp(-2\pi \cpolo(\rho_0) i) $ is a root of $\Delta_q^{(\ex)}(\tau^{\exz  + \ex })$. 
This is not enough to check that $\exp(-2\pi a i)$ is 
a root of $\Delta(\tau)$ since there might be cancellations
due to the global factor, which is $(\tau^{\exz} -1)^{\chi(\PP^2 \setminus C_{\exz})}$.

In the case~\ref{pole2},
$\exp(-2\pi a i)$ is not a root of $\tau^{\exz} -1$ 
and no cancellation can arise.

In the case~\ref{pole3}, the global factor is a polynomial, and again there
is no cancellation.

For the remaining cases, we use some facts about rational pencils, see~\S\ref{K-K-pencils}, \S\ref{Sec:Cm} 
and~\S\ref{Sec:Bad_div}.

In the case~\ref{pole4}, the curve $C_\exz$ corresponds to a Kashiwara's pencil because $\chi(\PP^2\setminus C_\exz) <  0$ (see \S\ref{Sec:Cm}).   
Since $\exz a \in \ZZ$, it follows from Lemma~\ref{lema:cyclo2}  that $\rho_0$ is not associated to the branching component $E$ of the splice diagram of $(C_\exz, q)$. 
Therefore, $\rho_0$ is associated to a branching component $v \in B'$. Lemma~\ref{lem:prop_cyclo}\ref{rmk_cyclotomic3b} applied to  
$\Delta_v(\tau)$ implies that 
$\exp(-2\pi \cpolo(\rho_0) i) $ is a root of $\Delta_v^{(\ex)}(\tau^{\exz  + \ex })$. 
Because of Veys' proof of the Monodromy Conjecture for curves~\cite{veys:mon1, veys:mon2}, we have  
$\exp (-2 \pi \rho_0 i)$ is a root of $\Delta_v(\tau)$.
Then,  we apply the factorization \eqref{eq:global_local};
notice that there is no cancellation since 
 the $E$-global factor is a
 polynomial, see Lemma~\ref{lema:deltaEphi}\ref{lema:deltaEphi1}. Hence
 we have finished the case~\ref{pole4}.

The case~\ref{pole6} and~\ref{pole7} deal with $\frac{3}{\exz}$. Recall that  $\cpolo\left(\frac{3}{\exz}\right)=\frac{3}{\exz}$ because of Remark~\ref{rem:n+1m}.

The case~\ref{pole6} assumes $-\frac{3}{\exz}$ is a multiple pole and $\chi(\PP^2\setminus C_\exz)\leq 0$. Theorem~\ref{thm:DLZFkLY} implies that $-\frac{3}{\exz }$ is 
a pole of $\Ztop{}(f_q, s)$ for some $q\in\sing C_\exz$, and then, by the Monodromy Conjecture for curves,  
$\exp\left(-2 \pi \frac{3}{\exz} i\right)$
is a root of the $\Delta_q(\tau)$.  Applying Lemma~\ref{lem:prop_cyclo}\ref{rmk_cyclotomic3b}
we conclude that $\exp\left(-2 \pi \frac{3}{\exz} i\right)$
is a root of $\Delta_q^{(\ex)}(\tau^{\exz + \ex})$. We distinguish two cases depending on the value of $\chi(\PP^2\setminus C_\exz)$.

If $\chi(\PP^2\setminus C_\exz)=0$ 
then $r=2$, and by ~\eqref{eq:global_local} we get that
$\exp\left(-2 \pi \frac{3}{\exz} i\right)$
is a root of $\tilde{\Delta}(\tau)$.

If $\chi(\PP^2\setminus C_\exz )<0$, we are in a Kashiwara's pencil
and by Lemma~\ref{lema:cyclo2}, $\exp\left(-2 \pi \frac{3}{\exz} i\right)$ is a root of a $v$-local factor  
$\Delta_v(\tau)$ and then, by Lemma~\ref{rmk_cyclotomic}\ref{rmk_cyclotomic3b}, it is also a root 
of $\Delta_v^{(\ex)}(\tau^{\exz + \ex})$.
Moreover, since $\chi(\PP^2\setminus C_\exz )=2-r$, the curve $C_{\exz}$ has $r \geq 3$ irreducible components. By Lemma~\ref{lema:deltaEphi}\ref{lema:deltaEphi1}, we have that $\exp\left(- 2 \pi \frac{3}{\exz} i\right)$ is a root of $\Delta_E(\tau)$ with multiplicity greater or equal than~$r-2$. By ~\eqref{eq:global_local} we obtain
$\exp\left(- 2 \pi \frac{3}{\exz} i\right)$ is a root of $\tilde{\Delta}(\tau)$. This finishes the case~\ref{pole6}.

The case~\ref{pole7}
assumes  $C_{\exz}$ is bad divisor with $\res(C_{\exz}) \ne 0$, see~\S\ref{Sec:Bad_div}. Proposition~\ref{rem:SIS_typo} implies that $C_{\exz}$ corresponds to a Kashiwara's pencil and $C_{\exz}$ is residual. Then, 
if $\ell=\frac{\exz}{\gcd(\exz, 3)}$,
Lemma~\ref{lema:deltaEphi}\ref{lema:deltaEphi3} implies that $\Phi_{\ell}^{r-1}$ divides $\Delta_E(\tau)$, for $\ell=\frac{\exz}{\gcd(\exz, 3)}$.
Using Lemma~\ref{lem:prop_cyclo}\ref{rmk_cyclotomic4},
we have that $\Phi_{\ell}^{r-1}$ divides $\Delta_E^{(\ex)}(\tau^{\exz + \ex})$, and then 
$\exp\left(- 2 \pi \frac{3}{\exz} i\right)$ is a root
of the $E$-global factor, and also of $\tilde{\Delta}(\tau)$  by ~\eqref{eq:global_local}.
\end{proof}

The following example deals with a few cases where the degree of the projectivized tangent cone $C_{\exz}$ is $\exz = 2$ or $3$ and $\chi(\PP^2 \setminus C_{\exz})\leq 0$. 

\begin{example}\label{ex:m<4}
Assume  $\chi(\PP^2 \setminus C_{\exz})\leq 0$ and $\deg C = m\leq 3$. As explained in~\cite[Example 2.2]{ACNLM-ASENS}, there are four possible configurations for $C_{\exz}$. It  is composed by either two lines, three lines through a point,
three generic lines, or an irreducible conic with a tangent line. For $\ex$-LYS
the two first cases are suspensions $x^\exz + y^\exz + z^{\exz + \ex}$, $\exz=2, 3$,
and hence the monodromy conjecture holds in these cases. 

If $f_3=x y z$, then $\chi(\PP^2\setminus C_3)=0$, and for each singular point~$q$
(there are three nodes), we have $\Delta_q^{(\ex)}(\tau)=\tau - 1$ (for all~$\ex$)
and 
$\Ztop{}(f_q, s)=\frac{2}{(1 + s)^2}$.
Then,
\[
\Delta(\tau)=
\frac{(\tau^{3 + \ex}-1)^3}{\tau - 1}.
\]
We use Theorem~\ref{thm:DLZFkLY}\ref{ztopklis1} for $F = f_3 + 
f_{3 + \ex}$ with e.g. $f_{3+\ex}=x^{3 + \ex} + y^{3 + \ex}+z^{3 + \ex}$ (the projective zero locus of $f_{3+\ex}$ does not
pass through $\sing C_3$).
Note that $\chi(C_3 \setminus\sing C_3)=0$:
\begin{equation*}
	\Ztop{}(F,s)_\zero=\frac{3 s^{2} + 6 s + \ex + 3}{(s + 1)^3 (\ex + 3)},
\end{equation*}
so the monodromy conjecture holds for $F$.

Consider now $f_3=x (x z + y^2)$. We have
\[
\Delta(\tau)=
\begin{cases}
\frac{(\tau^{3 + \ex}-1)(\tau^{2(3 + \ex)}+1)}{\tau -1} & \text{if } \ex\equiv1\bmod{2},\\
\frac{(\tau^{3 + \ex}-1)(\tau^{3 + \ex}+1)^2}{\tau -1} & \text{if } \ex\equiv2\bmod{4},\\
\frac{(\tau^{3 + \ex}-1)^3}{\tau -1} & \text{if } \ex\equiv0\bmod{4},
\end{cases}
\]
and for the topological zeta function
\[
\Ztop{}(F, s)_\zero=
\begin{cases}
\frac{3(\ex + 4 + s)}{(s + 1)(3(\ex + 4) + 4 (\ex + 3) s)} & \text{if } \ex\equiv1\bmod{2},\\
\frac{3(\ex + 4)}{(s + 1)(3(\ex + 4) + 4 (\ex + 3) s)} & \text{if } \ex\equiv2\bmod{4},\\
\frac{3(\ex + 4 + 4 s)}{(s + 1)(3(\ex + 4) + 4 (\ex + 3) s)} & \text{if } \ex\equiv0\bmod{4},
\end{cases}
\]
so the monodromy conjecture holds for $F$.
\end{example}

\begin{proof}[Proof of Proposition{\rm~\ref{rem:SIS_typo}}]
\phantomsection\label{subsec:tp8}
The fact that we can exclude Kizuka's pencil is the main
contribution of~\cite[Theorem 2.15]{ACNLM-ASENS}. One of
the main tools of the proof of that result is the characterization of the  
structure of connected curves $C_\exz\subset\PP^2$ with $\chi(\PP^2\setminus C_m)\leq 0$ achieved by Veys~\cite[Theorem 3.5]{veys:str}. This theorem claims the existence of a connected extension $D \supset C_\exz$, with $\chi(\PP^2\setminus D)\leq 0$ (hence, all
the irreducible components of $D$ are rational), and  such  that all the irreducible components of $D$ are
 fibers of a rational pencil $\gamma:\PP^2\dashrightarrow\PP^1$. Moreover, such a pencil~$\gamma$ fits 
in the commutative diagram \eqref{fig:diagfibration}.
\begin{equation}\label{fig:diagfibration}
    \begin{tikzcd}
    X\ar[d,"\pi_2" left]\ar[rr,"\pi_{\gamma_2}"]\ar[dr, "\Gamma"]&
    &X_1\dar["\pi_{\gamma_1}"]\ar[dl, dashrightarrow, "\gamma_1" above left]\\
    \Sigma\rar["p_\Sigma" below]&\PP^1&
    \lar[dashrightarrow,"\gamma"]\PP^2
    \end{tikzcd}
\end{equation}
Let us describe the elements of \eqref{fig:diagfibration}:
\begin{itemize}
    \item $\pi_{\gamma_1}$ is the minimal resolution of $\sing D$;
    \item  $\pi_{\gamma_2}$ is a sequence of blowing-ups such that $\pi_\gamma:=\pi_{\gamma_1}\circ\pi_{\gamma_2}$ is a resolution of the pencil $\gamma$, i.e., the induced map~$\Gamma$ is a morphism whose generic fibers are connected smooth rational curves;
    \item $\pi_2:X\to\Sigma$ is  the contraction of the reducible fibers of $\Gamma$ 
producing the Hirzebruch surface~$\Sigma$,  and $p_\Sigma$ is the corresponding  projection.
\end{itemize}
The number of fibers in $ \pi_2 (\pi_{\gamma}^{-1}(D)) \subset \Sigma$ is at least~$3$ and it includes
the fibers contained in~$D$ and the special fibers;
as we are in Kashiwara's case, there are at most two special fibers.
Actually, the extra curves of $D$ are the special fibers of $\gamma$
which are not in~$C_\exz$.

Definition 2.7 in~\cite{ACNLM-ASENS}
introduces a numerical invariant, namely  $\zeta_K$, associated to
some canonical $\QQ$-divisors on a projective surface $X$.
If $X$ is a ruled surface and the support of the canonical
divisor is a section and some fibers, this invariant vanishes.
If $D$ is a bad divisor in $\PP^2$ and $K_1$ is the canonical
divisor in $X_1$ obtained as $-\frac{3}{\exz}\pi^*(C_\exz) + K_{\pi_1}$, then $\zeta_{K_1}=\res(C_\exz)$.

Proposition 2.10 in~\cite{ACNLM-ASENS}
allows to control the behavior of the invariant $\zeta_{K}$ under blowing-ups and/or blowing-downs.
Assume that $K$ is a 
$\QQ$-divisor in $X$, whose support $\bigcup_{i\in I} E_i$ is a simple normal crossing divisor, with invariant $\zeta_K$. Let $K=\sum_{i\in I}(\tilde{\nu}_i - 1) E_i$.
Let $P\in X$ and let $\sigma:\hat{X}\to X$ be the blow-up of $X$ at~$P$, and $E_P$ the corresponding  exceptional
divisor. The divisor
$\hat{K}=\sigma^*(K) + E_P$ is a canonical $\QQ$-divisor of $\hat{X}$. Then,
$\zeta_{\hat{K}}\neq\zeta_K$ if and only if $P$ belongs to only  one irreducible component
$E_i$, with $\tilde{\nu}_i=0$, and $E_i$ intersects exactly two other irreducible components, namely 
$E_j$ and $E_k$, for which $\tilde{\nu}_j+\tilde{\nu}_k=0$ and $\abs{\tilde{\nu}_j}\neq 1$. Note that this cannot happen if the pencil is of type~$I_a$.

If the curve $C_\exz$ contains all the special fibers,
then $\pi_2$ is the identity, and we fall in the scope 
of~\cite[Theorem 2.15, Case 1(i)]{ACNLM-ASENS}. More precisely, the strict transform of $E_i$ in $X_1=X$ (the minimal resolution of $C_\exz$)
is a branching component. Since $\tilde{\nu}_i=\nu_i-\frac{3}{\exz} N_i$,
then $-\frac{3}{\exz}$ is a pole of $\Ztop{}(f_q,s)_\zero$, and 
$C_\exz$ is not a bad divisor.
\end{proof}

\begin{remark}\label{rmk_veys_construction}
    What happens if $C_\exz$ does not contain all the special fibers? In this case,
    this special fiber is the image by $\pi_1$ of the strict transform of $E_P$,
    which has $\tilde{\nu}_P=0$, i.e., it is not in the support of the canonical divisor.
\end{remark}

\subsection{Examples of Kashiwara's pencils and curves \texorpdfstring{$C_\exz$}{Cm}}\label{sec:exam_kas_curves}
\mbox{}

We present several examples of Kashiwara's pencils to illustrate
how the invariants behave for the curves $C_\exz $ arising from these pencils.
Let us consider a Kashiwara's pencil~$\gamma$ with base point $q=[1:0:0]\in\PP^2$
and the diagram~\eqref{fig:diagfibration}.
It is known that if a curve $C_\exz$ satisfies that $\chi(\PP^2\setminus C_\exz)\leq 0$ and $\mathcal{R}(C_\exz)\neq 0$, 
then $\pi_2(\pi_\gamma^{-1}(C_\exz))$ contains the $(-1)$-section of ${p}_\Sigma$
and at least three fibers including the images of the exceptional fibers
(whether they are or not included in $C_\exz$). Let us now study some 
examples of curves arising from some Kashiwara's pencils.

\begin{example}\label{ex:quartic}
We consider a pencil $\gamma$ of curves of degree~$4$ generated by $G_0:(z x - y^2)^2 + y z^3=0$
and $4L:z^4=0$. The generic curves of $\gamma$, all of them but $4L$, have at $q$ a singular point of type~$\mathbb{A}_6$. This is a pencil of type $I_b$. We see the minimal
resolution~$\pi_\gamma$ of the pencil~$\gamma$ in Figure~\ref{fig:kaswhiwara-1-27} and the numerical invariants
in Table~\ref{tab:kaswhiwara-1-27}.

\begin{table}[ht]
    \centering
    \begin{tabular}{|c|c|c|c|c|c|c|c|c|c|}
        \hline
         & $E_1$ & $E_2$ & $E_3$ & $E_4$& $E_5$ & $E_6$
         & $E_7$ & $L$& $G_0$\\
         \hline
         $\nu$ & $2$ & $3$ & $4$ & $5$& $9$ & $10$
         & $11$ & $1$& $1$ \\
         \hline
         $N_L$ & $1$ & $2$ & $2$ & $2$& $4$ & $4$
         & $4$ & $1$& $0$ \\
         \hline
         $N_G$ & $2$ & $4$ & $6$ & $7$& $14$ & $15$
         & $16$ & $0$& $1$ \\
         \hline
         $E^2$ & $-2$ & $-2$ & $-3$ & $-2$& $-2$ & $-2$
         & $-1$ & $-1$& $0$ \\
         \hline
    \end{tabular}
    \caption{Numerical values for the pencil of quartic curves}
    \label{tab:kaswhiwara-1-27}
\end{table}

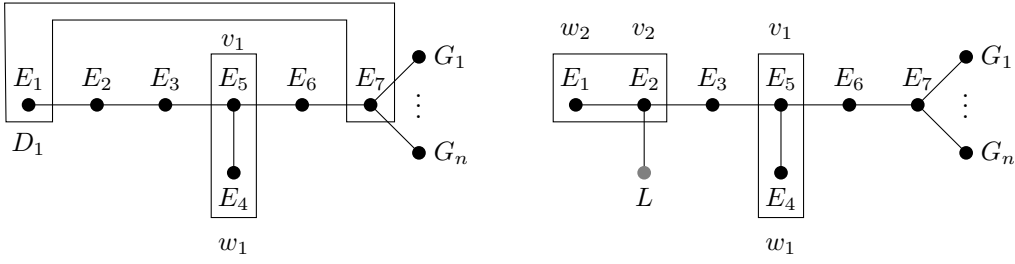
\begin{figure}
    \centering
    \begin{tikzpicture}[scale=.9]
\begin{scope}
\foreach \x in {-2,...,3}
{
\coordinate (A\x) at (\x,0);
\fill (A\x) circle [radius=.1];
}
\foreach \x in {-1, 1}
{
\coordinate (B\x) at (\x,-1);
}
\draw (A-2) -- (A3) (A1) -- (B1); \draw ($(A3)$) -- ($(A3) + (45:1)$);
\node at ($(A3) + (.7,0.1)$) {$\vdots$};

\draw ($(A3)$) -- ($(A3) + (-45:1)$);
\fill (B1) circle [radius=.1];
\fill ($(A3) + (45:1)$) circle [radius=.1];
\fill ($(A3) + (-45:1)$) circle [radius=.1];
\node[above=2pt] at (A-2) {$E_1$};
\node[above=2pt] at (A-1) {$E_2$};
\node[above=2pt] at (A0) {$E_3$};
\node[below=2pt] at (B1) {$E_4$};
\node[above=2pt] at (A1) {$E_5$};
\node[above=2pt] at (A2) {$E_6$};
\node[above=2pt] at (A3) {$E_7$};
\node[right=2pt] at ($(A3) + (45:1)$) {$G_1$};
\node[right=2pt] at ($(A3) + (-45:1)$) {$G_n$};

\draw ($(B1) - (.33, .66)$) rectangle ($(A1) + (.33, .75)$);
\node[below=.75cm] at (B1) {$w_1$};
\node[above=.6cm] at (A1) {$v_1$};
\node[below=.25cm] at (A-2) {$D_1$};

\draw ($(A-2) - (.33, .25)$) --
($(A-2) + (.35, -.25)$) --
($(A-2) + (.35, 1.25)$) --
($(A3) + (-.35, 1.25)$) -- 
($(A3) + (-.35, -.25)$) -- 
($(A3) + (.35, -.25)$) -- 
($(A3) + (.35, 1.5)$) -- 
($(A-2) + (-.35, 1.5)$) -- cycle;

\end{scope}

\begin{scope}[xshift=8cm]
    \foreach \x in {-2,...,3}
{
\coordinate (A\x) at (\x,0);
\fill (A\x) circle [radius=.1];
}
\foreach \x in {-1, 1}
{
\coordinate (B\x) at (\x,-1);
}
\draw (A-2) -- (A3) (A1) -- (B1) (B-1) -- (A-1);
\draw ($(A3)$) -- ($(A3) + (45:1)$);
\node at ($(A3) + (.7,0.1)$) {$\vdots$};

\draw ($(A3)$) -- ($(A3) + (-45:1)$);
\fill[gray] (B-1) circle [radius=.1];
\fill (B1) circle [radius=.1];
\fill ($(A3) + (45:1)$) circle [radius=.1];
\fill ($(A3) + (-45:1)$) circle [radius=.1];
\node[above=2pt] at (A-2) {$E_1$};
\node[above=2pt] at (A-1) {$E_2$};
\node[above=2pt] at (A0) {$E_3$};
\node[below=2pt] at (B1) {$E_4$};
\node[above=2pt] at (A1) {$E_5$};
\node[above=2pt] at (A2) {$E_6$};
\node[above=2pt] at (A3) {$E_7$};
\node[below=2pt] at (B-1) {$L$};
\node[right=2pt] at ($(A3) + (45:1)$) {$G_1$};
\node[right=2pt] at ($(A3) + (-45:1)$) {$G_n$};

\draw ($(B1) - (.33, .66)$) rectangle ($(A1) + (.33, .75)$);
\node[below=.75cm] at (B1) {$w_1$};
\node[above=.75cm] at (A1) {$v_1$};
\draw ($(A-2) - (.33, .25)$) rectangle ($(A-1) + (.33, .75)$);
\node[above=.75cm] at (A-1) {$v_2$};
\node[above=.75cm] at (A-2) {$w_2$};

\end{scope}
\end{tikzpicture}
    \caption{Pencils of quartics of type $I_b$.}
    \label{fig:kaswhiwara-1-27}
\end{figure}

The divisor $F_L:=E_4 + E_6 + 2 (E_1 + E_3 + E_5) + 4 (E_2 + L)$
is the fiber $\Gamma^*(\gamma(L))$ of $\Gamma$. We can blow-down all the irreducible components of $F_L$, except~$E_6$, to  obtain~$\Sigma$. For a union  $C_\exz$ of fibers containing $G_1,\dots,G_n$, we have 
$\pi_2(\pi_{\gamma}^{-1}(C_\exz))=E_6+E_7+G_1+\dots+G_n$, i.e., the ($-1)$-section $E_7$
of $\Sigma$ and $n+1$ fibers.
There are two choices for the curve $C_\exz $.
\begin{description}[wide, leftmargin=0pt]
\item[${[I_b]}$] For $C_\exz=G_1+\dots+G_n$, i.e., $\exz=4 n$ and $n\geq 3$,
the canonical divisor obtained from $-\frac{3}{\exz}C_\exz$ is 
\begin{align*}
K_X=&-\frac{3}{4n}\sum_{j=1}^n G_j - \frac{1}{2} E_1 - E_2 - \frac{3}{2} E_3 -\frac{5}{4} E_4 
- \frac{5}{2} E_5 -\frac{9}{4} E_6 - 2 E_7.    
\end{align*}
In this case the coefficient $-1$ implies $\res(C_\exz)\neq 0$
and $-\frac{3}{\exz}$ is a pole of the topological zeta
function of a $\ex$-LYS with projectivized tangent cone~$C_\exz$.
Note that \eqref{eq:deltaEv} gives
\[
\Delta_q(\tau)= \Delta_E(\tau )\Delta_{v_1}(\tau)=
\underbrace{(\tau^{7n} + 1)}_{\Delta_{v_1}(\tau)}\frac{\tau^{16 n}-1}{\tau^{2n} - 1}(\tau^{16 n}-1)^{n - 2}(\tau - 1);
\]
as a consequence, $\Delta_E(\tau)$ has roots
of order~$\exz=4n$ with multiplicity~$n-1=r-1$,
and has roots of order~$4\exz=16n$. This is a residual curve.

\item[${[I_b^L]}$]  For $C_\exz=L+G_1+\dots+G_n$, i.e., $\exz=4 n + 1$ and $n\geq 2$,
the canonical divisor obtained from $-\frac{3}{\exz}C_\exz$ is 
\begin{align*}
K_X=& -\frac{3}{4 n + 1}\sum_{j=1}^n G_j  - \frac{3}{4 n+1} L - 2\frac{n + 1}{4 n + 1} E_1 
-4\frac{n + 1}{4 n + 1} E_2\\
& - 3\frac{2 n + 1}{4 n + 1} E_3 -\frac{5 n + 2}{4 n+1} E_4 -2\frac{5 n+2}{4 n+1}E_5
 -3\frac{3 n+1}{4 n + 1}E_6 - 2 E_7.    
\end{align*}
No coefficient is $-1$ so $\res(C_\exz)=0$, despite $C_\exz$ is a bad divisor, as stated in Proposition~\ref{rem:SIS_typo}
since the only exceptional fiber is in the curve.
Hence $-\frac{3}{\exz}$ is not a pole  for any $\ex$-LYS
with projectivized tangent cone~$C_\exz$. Note that \eqref{eq:deltaEv} gives
\[ 
\Delta_q(\tau)= \Delta_E(\tau) \Delta_{v_1}(\tau)\Delta_{v_2} (\tau)=
\underbrace{(\tau^{1 + 2n} + 1)}_{\Delta_{v_2}(\tau)}
\underbrace{(\tau^{2 + 7n} + 1)}_{\Delta_{v_1}(\tau)}(\tau^{16 n + 4}-1)^{n - 1}(\tau - 1).
\]
Hence, the multiplicity of a root of order~$\exz=4 n + 1$
in $\Delta_q^{(\ex)}(\tau^{\exz + \ex})$ is $n - 1 = r - 2 = -\chi(\PP^2\setminus C_\exz)$. Notice that $\Delta(\tau)$
might not have a root of order $\exz$ but this is not a problem because $C_\exz $ is not residual and, therefore,  $-\frac{3}{\exz}$ is not a pole of  the $\ex$-LYS defined by $F$.

\end{description}
\end{example}

\FloatBarrier
\begin{example}\label{ex:sextic}
We consider a pencil $\gamma$ of curves of degree~$6$ generated by $G_0:(z^2 x - y^3)^2 + y z^5=0$
and $6L:z^6=0$. The generic curves  of $\gamma$ have two Puiseux pairs at $q$. 
This is a pencil of type~$I_a$. We see the minimal
resolution $\pi_{\gamma}$ of the pencil in Figure~\ref{fig:kaswhiwara-1-23-23} and the numerical invariants
in Table~\ref{tab:kaswhiwara-1-23-23}. 

\begin{table}[ht]
    \centering
    \begin{tabular}{|c|c|c|c|c|c|c|c|c|c|c|c|}
        \hline
         & $E_1$ & $E_2$ & $E_3$ & $E_4$& $E_5$ & $E_6$
         & $E_7$ & $E_8$ & $E_9$ & $L$& $G_0$\\
         \hline
         $\nu$ & $2$ & $3$ & $5$ & $6$& $7$ & $8$
         & $15$ & $16$ & $17$ & $1$& $1$ \\
         \hline
         $N_L$ & $1$ & $2$ & $3$ & $3$& $3$ & $3$
         & $6$ & $6$ & $6$ & $1$& $0$ \\
         \hline
         $N_G$ & $4$ & $6$ & $12$ & $14$& $16$ & $17$
         & $34$ & $35$ & $36$ & $0$& $1$ \\
         \hline
         $E^2$ & $-3$ & $-2$ & $-2$ & $-2$& $-3$ & $-2$
         & $-2$ & $-2$ & $-1$ & $-1$& $0$ \\
         \hline
    \end{tabular}
    \caption{Numerical values for the pencil of sextic curves with two Puiseux pairs}
    \label{tab:kaswhiwara-1-23-23}
\end{table}

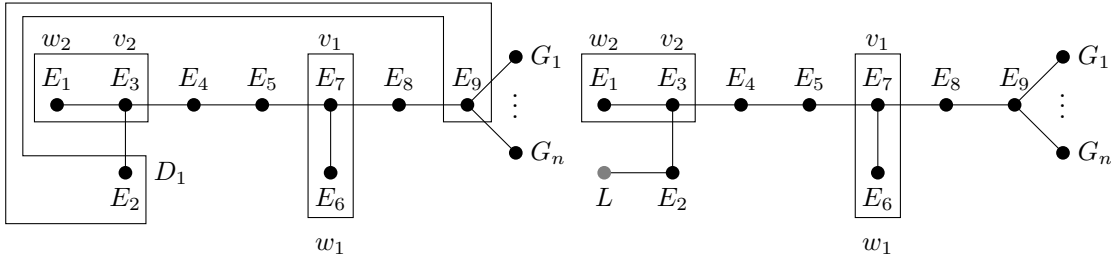
\begin{figure}
    \centering
    \begin{tikzpicture}[scale=.9]
	\begin{scope}
		\foreach \x in {-3,...,3}
		{
			\coordinate (A\x) at (\x,0);
			\fill (A\x) circle [radius=.1];
		}
		\foreach \x in {-3,-2,-1,1}
		{
			\coordinate (B\x) at (\x,-1);
		}
		\draw (A-3) -- (A3) (A1) -- (B1) (B-2) -- (A-2);
		\draw ($(A3)$) -- ($(A3) + (45:1)$);
		\node at ($(A3) + (.7,0.1)$) {$\vdots$};
		
		\draw ($(A3)$) -- ($(A3) + (-45:1)$);
\fill (B-2) circle [radius=.1];
		\fill (B1) circle [radius=.1];
		\fill ($(A3) + (45:1)$) circle [radius=.1];
		\fill ($(A3) + (-45:1)$) circle [radius=.1];
		\node[above=2pt] at (A-3) {$E_1$};
		\node[below=2pt] at (B-2) {$E_2$};
		\node[above=2pt] at (A-2) {$E_3$};
		\node[above=2pt] at (A-1) {$E_4$};
		\node[above=2pt] at (A0) {$E_5$};
		\node[below=2pt] at (B1) {$E_6$};
		\node[above=2pt] at (A1) {$E_7$};
		\node[above=2pt] at (A2) {$E_8$};
		\node[above=2pt] at (A3) {$E_9$};
\node[right=2pt] at ($(A3) + (45:1)$) {$G_1$};
		\node[right=2pt] at ($(A3) + (-45:1)$) {$G_n$};
		
		\draw ($(B1) - (.33, .66)$) rectangle ($(A1) + (.33, .75)$);
\draw ($(A-3) - (.33, .25)$) rectangle ($(A-2) + (.33, .75)$);
		
		\node[below=.75cm] at (B1) {$w_1$};
		\node[above=.6cm] at (A1) {$v_1$};
		\node[above=.6cm] at (A-3) {$w_2$};
		\node[above=.6cm] at (A-2) {$v_2$};
		\node[right=.25cm] at (B-2) {$D_1$};
		
		\draw ($(B-1) - (.7, -.25)$) --
		($(B-1) - (.7, .75)$) --
		($(A-3) - (.75, 1.75)$) --
		($(A-3) - (.75, -1.5)$) --
		($(A3) + (.35, 1.5)$) -- 
		($(A3) + (.35, -.25)$) -- 
		($(A3) + (-.35, -.25)$) -- 
		($(A3) + (-.35, 1.3)$) -- 
		($(A-3) + (-.5, 1.3)$) --
		($(A-3) + (-.5, -.75)$)-- cycle;
	\end{scope}
	
	\begin{scope}[xshift=8cm]
		\foreach \x in {-3,...,3}
		{
			\coordinate (A\x) at (\x,0);
			\fill (A\x) circle [radius=.1];
		}
		\foreach \x in {-3,-2,1}
		{
			\coordinate (B\x) at (\x,-1);
		}
		\draw (A-3) -- (A3) (A1) -- (B1) (B-3) -- (B-2) -- (A-2);
		\draw ($(A3)$) -- ($(A3) + (45:1)$);
		\node at ($(A3) + (.7,0.1)$) {$\vdots$};
		
		\draw ($(A3)$) -- ($(A3) + (-45:1)$);
		\fill[gray] (B-3) circle [radius=.1];
		\fill (B-2) circle [radius=.1];
		\fill (B1) circle [radius=.1];
		\fill ($(A3) + (45:1)$) circle [radius=.1];
		\fill ($(A3) + (-45:1)$) circle [radius=.1];
		\node[above=2pt] at (A-3) {$E_1$};
		\node[below=2pt] at (B-2) {$E_2$};
		\node[above=2pt] at (A-2) {$E_3$};
		\node[above=2pt] at (A-1) {$E_4$};
		\node[above=2pt] at (A0) {$E_5$};
		\node[below=2pt] at (B1) {$E_6$};
		\node[above=2pt] at (A1) {$E_7$};
		\node[above=2pt] at (A2) {$E_8$};
		\node[above=2pt] at (A3) {$E_9$};
		\node[below=2pt] at (B-3) {$L$};
		\node[right=2pt] at ($(A3) + (45:1)$) {$G_1$};
		\node[right=2pt] at ($(A3) + (-45:1)$) {$G_n$};
		
		\draw ($(B1) - (.33, .66)$) rectangle ($(A1) + (.33, .75)$);
		\draw ($(A-3) - (.33, .25)$) rectangle ($(A-2) + (.33, .75)$);
		\node[below=.75cm] at (B1) {$w_1$};
		\node[above=.6cm] at (A1) {$v_1$};
		\node[above=.6cm] at (A-3) {$w_2$};
		\node[above=.6cm] at (A-2) {$v_2$};
	\end{scope}
\end{tikzpicture}
    \caption{Pencils of sextics with two Puiseux pairs of type $I_a$.}
    \label{fig:kaswhiwara-1-23-23}
\end{figure}

The divisor $F_L:= E_6 + E_8 + 2 (E_1 + E_5 + E_7) + 4 E_4 + 6 (E_2 + E_3 + L)$
is the fiber $\Gamma^*(\gamma(L))$. We can blow-down all the irreducible components of $F_L$, except~$E_8$, to  obtain~$\Sigma$. For a union  $C_\exz$ of fibers containing $G_1,\dots,G_n$, we have 
$\pi_2(\pi_{\gamma}^{-1}(C_\exz))=E_8+E_9+G_1+\dots+G_n$, i.e., the $(-1)$-section $E_9$
of $\Sigma$ and $n+1$ fibers.
There are two choices for the curve $C_\exz $.
\begin{description}[wide, leftmargin=0pt]
\item[${[I_a]}$] For $C_\exz=G_1+\dots+G_n$, i.e., $\exz=6 n$,
the canonical divisor obtained from $-\frac{3}{\exz}C_\exz$ is 
\begin{align*}
	K_X=&-\frac{1}{2n}\sum_{j=1}^n G_j - (E_1 + E_2) - 2 (E_3 + E_4 + E_5 + E_9)
	-\frac{3}{2} E_6 - 3 E_7 - \frac{5}{2} E_8.    
\end{align*}
Even though $-1$ appears as a coefficient, $\res(C_\exz)=0$, since the blow-ups from the ruled surface do not change the residue.
Note that in this case $\Delta(\tau)$ has roots of order~$\exz$
even if it is not needed for the monodromy conjecture.
This is a non-residual curve and in this case the eigenvalue
associated to $-\frac{e}{\exz}$ is not of order $\exz$
but of order~$\frac{\exz}{3}=2n$.

Note that \eqref{eq:deltaEv} gives
\[ 
\Delta_q(\tau)= \Delta_E(\tau) \Delta_{v_1}(\tau)\Delta_{v_2} (\tau)=
\underbrace{\dfrac{{\tau^{12n} + 1}}{\tau^{4n + 1}}}_{\Delta_{v_2}(\tau)}
\underbrace{(\tau^{17n} + 1)}_{\Delta_{v_1}(\tau)}\tau^{36 n}-1)^{r - 2}(\tau - 1)\frac{\tau^{36 n}-1}{\tau^{6n} - 1}.
\]
In this case $\Phi_{6n}(\tau)$ is not a divisor of $\Delta_E(\tau)$. Actually, it is a divisor of $\Delta_{v_2}(\tau)$, but 
we do not need to check the divisibility since the curve $C_\exz$ is not residual.

\item[${[I_a^L]}$] For $C_\exz=L+G_1+\dots+G_n$, i.e., $\exz=6 n + 1$,
the canonical divisor obtained from $-\frac{3}{\exz}C_\exz$ is 
\begin{align*}
K_X=& -\frac{3}{6 n + 1}\sum_{j=1}^n G_j  - \frac{3}{6n+1} L - 2\frac{3 n + 1}{6 n + 1} E_1 
- 2\frac{3 n + 2}{6 n + 1} E_2 - \frac{12 n + 5}{6 n + 1} E_3\\
& - 4\frac{3n + 1}{6n+1} E_4-3\frac{4n+1}{6n+1}E_5
 -\frac{9n+2}{6n + 1}E_6  -2 \frac{9n+2}{6n+1} E_7- 3 \frac{5n+1}{6n+1} E_8 - 2 E_9.    
\end{align*}
No coefficient is $-1$ so $\res(C_\exz)=0$, despite $C_\exz$ is a bad divisor, as stated in Proposition~\ref{rem:SIS_typo}
since the only exceptional fiber is in the curve, as in Example~\ref{ex:quartic}.
This is a non-residual curve.

\end{description}
\end{example}

\FloatBarrier

\begin{example}\label{ex:10-curve}
We consider a pencil of curves of degree~$10$ generated by $2C_5$ and $5C_2$
where $C_5:x (x z - y^2)^2 
-2 (x z - y^2) y z^2+ z^5=0$
and $C_2:x z - y^2=0$. The generic curves have a singular point $q=[1:0:0]$
of type~$(4, 25)$. 
This is a pencil of type $II$. We see the minimal
resolution of the pencil in Figure~\ref{fig:kaswhiwara-2-4-25} and the numerical invariants
in Table~\ref{tab:kaswhiwara-2-4-25}.

\begin{table}[ht]
    \centering
    \begin{tabular}{|c|c|c|c|c|c|c|c|c|c|c|c|c|c|}
        \hline
         & $E_1$ & $E_2$ & $E_3$ & $E_4$& $E_5$ & $E_6$
         & $E_7$ & $E_8$ & $E_9$ & $E_{10}$ & $C_2$ & $C_5$ & $R_0$\\
         \hline
         $\nu$ & $2$ & $3$ & $4$ & $5$& $6$ & $7$
         & $8$ & $15$ & $22$ & $29$ & $1$ & $1$ & $1$ \\
         \hline
         $N_{C_2}$ & $1$ & $2$ & $3$ & $4$& $5$ & $5$
         & $5$ & $10$ & $15$ & $20$ & $1$& $0$ & $0$ \\
         \hline
         $N_{C_5}$ & $2$ & $4$ & $6$ & $8$& $10$ & $12$
         & $13$ & $26$ & $38$ & $50$ & $0$ & $1$& $0$ \\
         \hline
         $N_{G}$ & $4$ & $8$ & $12$ & $16$& $20$ & $24$
         & $25$ & $50$ & $75$ & $100$ & $0$ & $0$ & $1$ \\
         \hline
         $E^2$ & $-2$ & $-2$ & $-2$ & $-2$& $-2$ & $-5$
         & $-2$ & $-2$ & $-2$ & $-1$ & $-1$& $-1$ &$0$ \\
         \hline
    \end{tabular}
    \caption{Numerical values for the pencil of curves of degree~$10$}
    \label{tab:kaswhiwara-2-4-25}
\end{table}

\begin{figure}
    \centering
    \begin{tikzpicture}
\foreach \x in {-4,...,5}
{
\coordinate (A\x) at (\x,0);
\fill (A\x) circle [radius=.1];
}
\foreach \x in {0, 4}
{
\coordinate (B\x) at (\x,-1);
}
\draw (A-4) -- (A5) (A0) -- (B0) (B4) -- (A4);
\draw ($(A2)$) -- ($(A2) + (120:1)$);
\node at ($(A2) + (0, .7)$) {$\dots$};
\draw ($(A2)$) -- ($(A2) + (60:1)$);
\fill[gray] (B0) circle [radius=.1];
\fill[gray] (B4) circle [radius=.1];
\fill ($(A2) + (120:1)$) circle [radius=.1];
\fill ($(A2) + (60:1)$) circle [radius=.1];
\node[above=2pt] at (A-4) {$E_1$};
\node[above=2pt] at (A-3) {$E_2$};
\node[above=2pt] at (A-2) {$E_3$};
\node[above=2pt] at (A-1) {$E_4$};
\node[above=2pt] at (A0) {$E_5$};
\node[above=2pt] at (A1) {$E_6$};
\node[above=2pt] at (A5) {$E_7$};
\node[above=2pt] at (A4) {$E_8$};
\node[above=2pt] at (A3) {$E_9$};
\node[below=2pt] at (A2) {$E_{10}$};
\node[below=2pt] at (B0) {$C_2$};
\node[below=2pt] at (B4) {$C_5$};
\node[left=2pt] at ($(A2) + (120:1)$) {$R_1$};
\node[right=2pt] at ($(A2) + (60:1)$)  {$R_n$};
\end{tikzpicture}
    \caption{Pencil of curves of degree~$10$}
    \label{fig:kaswhiwara-2-4-25}
\end{figure}
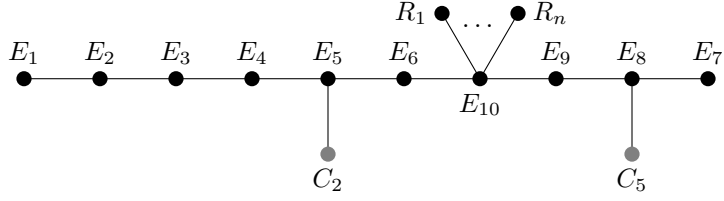

The divisors $\Gamma^*( \gamma(  C_2)) = E_1 + E_6 + 2 E_2 + 3 E_3 + 4 E_4 + 5 (E_5 + C_2)$
and $\Gamma^*( \gamma(  C_5)) = E_9 + E_7 + 2 (E_8 + C_5)$
are fibers of the ruling defined by the pencil. We can blow-down these fibers such
that the image of $E_6, E_9$ become fibers. In the ruling, we have as divisor $E_6+E_9+E_{10}+G_1+\dots+G_n$. There are four choices for the curve $C_\exz $.
\begin{description}[wide, leftmargin=0pt]   
\item[${[II]}$]
For $C_\exz=G_1+\dots+G_n$, i.e., $\exz=10 n$ and $r=n$,
the canonical divisor obtained from $-\frac{3}{\exz}C_\exz$ is 
\begin{align*}
	K_X=&-\frac{1}{2n}\sum_{j=1}^n G_j - \frac{1}{5} (E_1 + 2 E_2 + 3 E_3 + 4 E_4 + 5 E_5 + 6 E_6)\\
	&- \frac{1}{2} (E_7 + 2 E_8 + 3 E_9 + 4 E_{10}).    
\end{align*}
In these $3$ cases we have $\res(C_\exz)\neq 0$ and $\Delta(\tau)$ has roots of order~$\exz$.

\item[${[II^1]}$]
For $C_\exz=C_5+G_1+\dots+G_n$, i.e., $\exz=5 (2 n + 1)$ and $r=n+1$,
the canonical divisor obtained from $-\frac{3}{\exz}C_\exz$ is 
\begin{align*}
K_X=& -\frac{3}{5 (2 n + 1)}\sum_{j=1}^n G_j  - \frac{3}{5 (2 n + 1)} C_5 
- \frac{1}{5} (E_1 + 2 E_2 + 3 E_3 + 4 E_4 + 5 E_5 + 6 E_6) \\
& - \frac{5 n + 4}{5(2 n + 1)} (E_7 + 2 E_8) - \frac{3 (5 n + 3)}{5(2 n + 1)} E_9
- 5 \frac{n + 1}{10 n + 7} (E_7 + 2 E_8) - 3\frac{5 n + 4}{10 n + 7} E_9
 - 2 E_{10}.    
\end{align*}

\item[${[II^2]}$]
For $C_\exz=C_2+G_1+\dots+G_n$, i.e., $\exz=2 (5 n + 1)$,
the canonical divisor obtained from $-\frac{3}{\exz}C_\exz$ is 
\begin{align*}
K_X=& -\frac{3}{2 (5 n + 1)}\sum_{j=1}^n G_j  - \frac{3}{2 (5 n + 1)} C_5 
- \frac{2 n + 1}{2 (5 n + 1)} (E_1 + 2 E_2 + 3 E_3 + 4 E_4 + 5 E_5 + 6 E_6) \\
& - \frac{1}{2} (E_7 + 2 E_8 + 3 E_9 + 4 E_{10}).    
\end{align*}

\item[${[II^{1,2}]}$]
For $C_\exz=C_2+C_5+G_1+\dots+G_n$, i.e., $\exz=10 n + 7$,
the canonical divisor obtained from $-\frac{3}{\exz}C_\exz$ is 
\begin{align*}
	K_X=& -\frac{3}{10 n + 17}\sum_{j=1}^n G_j  - \frac{3}{10 n + 17} (C_2 + C_5) 
	- 2 \frac{n + 1}{10 n + 7} (E_1 + 2 E_2 + 3 E_3 + 4 E_4 + 5 E_5) \\
	& - 3\frac{4 n + 3}{10 n + 7} E_6 
	- 5 \frac{n + 1}{10 n + 7} (E_7 + 2 E_8) - 3\frac{5 n + 4}{10 n + 7} E_9
	- 2 E_{10}.    
\end{align*}
No coefficient is $-1$ so $\res(C_\exz)=0$, despite $C_\exz$ is a bad divisor, as stated in Proposition~\ref{rem:SIS_typo}
since the special fiber is in the curve, see
Example~\ref{ex:quartic}.

\end{description}
\end{example}

 \section{Holomorphy conjecture for \texorpdfstring{$\ex$}{\ex}-LYS of surfaces}
\label{subsec:hol_lys}
\mbox{}

Let us use a  modification of the Strategy~\ref{strategy}, using the formulas in Theorem~\ref{thm:DLZFkLY},
to prove the Holomorphy Conjecture~\ref{hol_conj}   
for $\ex$-LYS surfaces, see Theorem~\ref{thm:hc_n=2}.

Let us fix a germ of $\ex$-LYS singularity of surface defined by  $F:(\CC^{3},\zero)\to (\CC,\zero)$ with projectivized tangent cone $C_\exz$.
We also need the following notations.

\begin{notation}\label{not:poloq} 
For 
any prime number $p$ let us denote the corresponding $p$-valuation by~$\nu_p$. We use the notations
\[
\ell_1 := \frac{\ell}{\gcd(\ell, m +\ex)}, \qquad  {\rm and } \qquad  \mkl(\ex,\ell,\exz + \ex ):= \ell_1\gcd(\ex,\ell_1^u), \quad 
{\rm for } \quad  u\gg 1,
\]
from Lemma~\ref{lema:arit}. 
Moreover, if $n \in \mathbb{N}$ we denote  
\[
\poloq(n, \exz, \ex):=(m + \ex)\frac{n}{\gcd(n, \ex)}.
\]
Note that if $n_1$ divides $n_2$, say $n_2=n_0n_1$, then $\poloq(n_1, \exz, \ex)$ divides
$\poloq(n_2, \exz, \ex)$:
\begin{align*}
    \frac{\poloq(n_2, \exz, \ex)}{(m + \ex)}&=
    \frac{n_0 n_1}{\gcd(n_0 n_1, \ex)}=
    \frac{n_0}{\gcd\left(n_0\frac{n_1}{\gcd(n_1,\ex)}, \frac{\ex}{\gcd(n_1, \ex)}\right)}
    \frac{n_1}{\gcd(n_1, \ex)}\\
    &=\frac{n_0}{\gcd\left(n_0, \frac{\ex}{\gcd(n_1, \ex)}\right)}
    \frac{\poloq(n_1, \exz, \ex)}{(m + \ex)}.
\end{align*}
\end{notation}

It is worth mentioning that in the superisolated case, i.e., if $\ex=1$, we have 
\[
\mkl(1,\ell,\exz + 1 )= \ell_1 = \frac{\ell}{\gcd(\ell, \exz + 1)}\quad \text{and}\quad\poloq(n, \exz, 1)=(\exz +1)n.
\] 

We need the following arithmetical lemmas.
For $\ex=1$ they are trivial and therefore the proof of the Holomorphy Conjecture in that case becomes much simpler.

\begin{lemma}\label{lem:divmlk}
Assume $\ell\in\NN$ does not divide
$\poloq(n, \exz, \ex)$. 
Then, $\lcm(e,\mkl(\ex,\ell,\exz  + \ex ))$ 
does not divide~$n$, for any divisor~$e$ of~$\ex$.
\end{lemma}
\begin{proof}
\phantomsection\label{subsec:tp9-1}
It is enough to prove the statement for $e=1$, i.e., it is enough to show that $\mkl(\ex,\ell,\exz  + \ex )$ does not divide $n$.
The hypothesis $\ell$ does not divide
$\poloq(n, \exz, \ex)$ is equivalent to $\ell_1$
does not divide $\frac{n}{\gcd(n,\ex)}$, and is equivalent to the existence of a prime~$p$ such that 
\begin{equation}\label{eqhol1}
\nu_p(\ell_1) > \nu_p(n) - \min(\nu_p(n),\nu_p(\ex)) \geq 0.
\end{equation}
Since $\nu_p(\ell_1) > 0$, for $u \gg 1$, we have that 
\begin{equation}\label{eqhol2}\min(u\nu_p(\ell_1),\nu_p(\ex))=\nu_p(\ex).
\end{equation} 
Hence, combining \eqref{eqhol1} and \eqref{eqhol2}, we have 
\begin{align*}
\nu_p(\mkl(\ex,\ell,\exz  + \ex ))=
\nu_p(\ell_1) + \nu_p(\ex) > \nu_p(n) +  \nu_p(\ex) - \min(\nu_p(n),\nu_p(\ex)).
\end{align*}
We conclude that  $ \nu_p(\mkl(\ex,\ell,\exz  + \ex )) > \nu_p(n)$ for some prime number $p$.  Therefore
$\mkl(\ex,\ell,\exz  + \ex )$ does not divide $n$. 
\end{proof}

\begin{lemma}\label{lem:divell}
Assume $\ell\in\NN$  does not divide
$\poloq(n, \exz, \ex)$, and $\ell$ divides $\exz$.
Then, 
$\ell$ does not divide~$n$.
\end{lemma}

\begin{proof}
\phantomsection\label{subsec:tp9-2}
As in the proof of Lemma~\ref{lem:divmlk} we have  that $\ell_1 = \frac{\ell}{\gcd(\ell, m +\ex)}$
does not divide $\frac{n}{\gcd(n,\ex)}$. Moreover, since $\ell\mid m$, 
we have that $\gcd(\ell, m +\ex)=\gcd(\ell, \ex)$. Therefore, there exits a prime~$p$ such that
\[
\nu_p(\ell) - \min(\nu_p(\ell), \nu_p(\ex)) >
\nu_p(n) - \min(\nu_p(n), \nu_p(\ex)),
\]
which is equivalent to
\[
\nu_p(\ell) - \nu_p(n) >
\min(\nu_p(\ell), \nu_p(\ex)) - \min(\nu_p(n), \nu_p(\ex)).
\]
In order to check the positivity of $\nu_p(\ell) - \nu_p(n)$, we notice that the last inequality  is equivalent
to the following
\begin{align}
    \label{eq:div1}
    \nu_p(\ell) - \nu_p(n) &>
\min(\nu_p(\ell), \nu_p(\ex)) - \nu_p(n)\\
\label{eq:div2}
    \nu_p(\ell) - \nu_p(n) &>
\min(\nu_p(\ell), \nu_p(\ex)) - \nu_p(\ex)
\end{align}
Inequality \eqref{eq:div1} is equivalent to 
$\nu_p(\ell) >
\min(\nu_p(\ell), \nu_p(\ex))$.  Hence, 
$\nu_p(\ell) >\nu_p(\ex)$. And, inequality
\eqref{eq:div2} becomes 
\[
\nu_p(\ell) - \nu_p(n) >
\min(\nu_p(\ell), \nu_p(\ex)) - \nu_p(\ex) =0.
\qedhere
\]
\end{proof}

\begin{prop}\label{prop:ord-Delta-F} Let  $F:(\CC^{3},\zero)\to (\CC,\zero)$ define a germ of $\ex$-LYS singularity of surface with projectivized tangent cone $C_\exz$. Then,  $$\overline{\mathcal{E}_F^{\ord}} = \begin{cases} 
\overline{ \{m\} \cup \{ \poloq(n, \exz, \ex) \, | \, n \in \mathcal{E}_{f_q}^{\ord}   \text{ and }  q \in \sing C_m \}} & \text{ if } \chi(\PP^2 \setminus C_m) \ne 0 \\
\overline{  \{ \poloq(n, \exz, \ex) \, | \, n \in \mathcal{E}_{f_q}^{\ord}   \text{ and }  q \in \sing C_m \}} & \text{ otherwise. }
\end{cases}
$$
\end{prop}

\begin{proof}
Let $n_1\in\mathcal{E}_F^{\ord}$. We are going to use \eqref{Deltak-LYS}. If
$\chi(\PP^2 \setminus C_{\exz})>0$,
the integer $n_1$ is either the order of a root of the global factor $(\tau^m-1)^{\chi(\PP^2 \setminus C_{\exz})}$
or it is the order of a root of some local factor $\Delta_q^{(\ex)}(\tau^{\exz+\ex})$ of $\tilde{\Delta}(\tau)$. 
If
$\chi(\PP^2 \setminus C_{\exz})\leq 0$,
the integer $n_1$ is the order of a root of some local factor $\Delta_q^{(\ex)}(\tau^{\exz+\ex})$ of $\tilde{\Delta}(\tau)$. 

If $n_1$ is the order of a root of the global factor (in particular, 
this implies that $\chi(\PP^2 \setminus C_{\exz})>0$), then 
$n_1$ is a divisor of $m$. 

If $n_1$ is the order of a root $\zeta_1$ of a local factor $\Delta_q^{(k)} (\tau^{\exz +k}) $, then $\zeta_1 = \exp( - 2\pi i\frac{k \rho_0 + \ell}{m+k} )$ for some integer $\ell$ and $\zeta = \exp(-2\pi i \rho_0) \in \mathcal{E}_{f_q}$ (see Lemma~\ref{lem:prop_cyclo}\ref{rmk_cyclotomic3b}). We get  $\zeta_1^{m+k} = \zeta^{k}$. If  $n$ is the order of 
$\zeta$ then we have the equality
\[
\frac{n_1}{\gcd(n_1, m+k)} = \frac{n}{\gcd(n, k)}.
\]
This implies that 
\[
\frac{n}{\gcd(n, \ex)} \mid n_1 \mid (m + \ex)\frac{n}{\gcd(n, \ex)} = \poloq(n, \exz, \ex).
\]
This shows the inclusion 
$\subset$ for all cases.

By \eqref{Deltak-LYS} the other inclusion is clear if  $\chi(\PP^2\setminus C_\exz) \geq 0$. Otherwise, the projectivized tangent cone $C_\exz$ is a reduced curve with $r >2$ components corresponding to  a Kashiwara's pencil; in particular $\Delta (\tau)$ is given by \eqref{Deltak-LYS-K},
see \S\ref{Sec:Cm}). 
Lemma~\ref{lema:deltaEphi}~\ref{lema:deltaEphi1} implies that the values $\poloq(n, \exz, \ex)$, 
for $n$ running through the orders of the roots of  the $v$-local factors of $\Delta_q$, are in  $\overline{\mathcal{E}_F^{\ord}}$. Also, Lemma~\ref{lema:deltaEphi}~\ref{lema:deltaEphi2} implies
    that $\exz$ and the values
    $\poloq(n, \exz, \ex)$, for $n$ running through the orders of the roots of
    the factor $\Delta_E(\tau)$ of $\Delta_q(\tau)$, are in $\overline{\mathcal{E}_F^{\ord}}$.
\end{proof}

\begin{theorem}\label{thm:hc_n=2}
    The holomorphy conjecture holds for $\ex$-LYS surfaces.
\end{theorem}

\begin{proof}
If the projectivized tangent cone $C_{\exz} \subset \PP^2$ is smooth, then, $\overline{\mathcal{E}_F^{\ord}} = \overline{\{m\}}$ according to Proposition~\ref{prop:ord-Delta-F}. Assuming  $\ell \nmid m$, Theorem~\ref{thm:DLZFkLY}~\ref{ztopklisdivkm} and~\ref{ztopklisnodiv} imply that  $\Ztop{(\ell)}(F,s)_\zero$  trivially vanishes, i.e., the Holomorphy Conjecture holds if $C_m$ is smooth.

Assume now that $\sing C_{\exz} \neq \emptyset$. Fix  $\ell \notin\overline{\mathcal{E}_F^{\ord}}$, and consider  the set 
\[
\mathcal{T}_{F, \ell} =
\left\{ \lambda \in \ZZ_{\geq 1}   \middle|  \begin{array}{c} 
\Ztop{(\lambda)}(f_q,t)_\zero  \text{ for some } q \in \sing C_m \text{ appears in } \\ 
\text{ the formula for } \Ztop{(\ell)}(F,s)_\zero 
\text{ in Theorem{\rm~\ref{thm:DLZFkLY}}}
\end{array}
\right\}. 
\]
Note that Proposition~\ref{prop:ord-Delta-F} implies that $\exz + \ex\in\overline{\mathcal{E}_F^{\ord}}$ and $\ell\nmid \exz+\ex$. Hence, we  only consider the formulas for $\Ztop{(\ell)}(F,s)_\zero$ in the  cases~\ref{ztopklisdivk}
and~\ref{ztopklisnodiv} of Theorem~\ref{thm:DLZFkLY}. We have 
\[
\mathcal{T}_{F, \ell} = 
\begin{cases}
    \{ \ell \} \cup \{\lcm(e,\mkl(\ex,\ell, \exz  + \ex)) \, : \,  e \mid k\} & \text{ if } \ell \nmid \exz + \ex \text{ and } \ell \mid m \, \text{ (case~\ref{ztopklisdivk})}, \\
 \{ \lcm(e,\mkl(\ex,\ell, \exz  + \ex)) \, : \,  e \mid k \} & \text{ if }   \ell \nmid \exz + \ex \text{ and } \ell \nmid m \, \text{ (case~\ref{ztopklisnodiv})}. \\
\end{cases}
\]
Let us first check that $\Ztop{(\lambda)}(f_q,t)_\zero=0$ for $\lambda \in  \{ \lcm(e,\mkl(\ex,\ell, \exz  + \ex)) \, : \,  e \mid k \}$ in case~\ref{ztopklisdivk} and in case~\ref{ztopklisnodiv}.
Let $n \in \mathcal{E}_{f_q}^{\ord}$. As explained in the proof of the Proposition~\ref{prop:ord-Delta-F}, the integer $\poloq(n, \exz, \ex) \in \overline{\mathcal{E}_F^{\ord}}$. Hence $\ell \nmid \poloq(n, \exz, \ex)$. As consequence of Lemma~\ref{lem:divmlk}, the integer $\lcm(e,\mkl(\ex,\ell,\exz  + \ex ))$ 
does not divide~$n$, for any divisor~$e$ of~$\ex$. The Holomorphy Conjecture for $f_q$ implies, as desired, that  $\Ztop{(\lambda)}(f_q,t)_\zero=0$ for $\lambda \in  \{ \lcm(e,\mkl(\ex,\ell, \exz  + \ex)) \, : \,  e \mid k \}$.

Hence,  if     $\ell \nmid \exz + \ex$  and  $\ell \nmid m $, i.e., for the  
case~\ref{ztopklisnodiv} of
Theorem~\ref{thm:DLZFkLY}, we have checked that
$\Ztop{(\lambda)}(f_q,t)_\zero=0$ for any $\lambda \in \mathcal{T}_{F, \ell}$. Therefore, we conclude that $\Ztop{(\ell)}(F,s)_\zero=0$ if $\exz\in\overline{\mathcal{E}_F^{\ord}}$and $\ell \not \in \overline{\mathcal{E}_F^{\ord}}$.

In the case~\ref{ztopklisdivk}, i.e., $\ell \nmid \exz + \ex$  and  $\ell \mid m $, Lemma~\ref{lem:divell} implies that $\ell \nmid n$. Hence, we have $\Ztop{(\ell)}(f_q,t)_\zero=0$ because of the Holomorphy Conjecture for curves. Hence, $\Ztop{(\lambda)}(f_q,t)_\zero=0$ for any $\lambda \in \mathcal{T}_{F, \ell}$, and $\Ztop{(\ell)}(F,t)_\zero= \frac{\chi(\PP^2 \setminus C_m)}{t-s}$, because of  
Theorem~\ref{thm:DLZFkLY}~\ref{ztopklisdivk}). However, $\exz \not \in  \overline{\mathcal{E}_F^{\ord}}$ implies that $\chi(\PP^2 \setminus C_m)=0$ because \eqref{Deltak-LYS}. And, hence, $\Ztop{(\ell)}(F,s)_\zero=0$. 
\end{proof}

 \section{Comparison with previous works}\label{comp}

\subsection{Previous works on suspensions}
\label{sec:preSUS}
\mbox{}

The original formul{\ae}~\eqref{eq:orgininalF2points} for the topological zeta functions of the suspension by two points  $F= z^2 + f(x_1, \ldots, x_{d})$ of a hypersurface singularity $f$
was proved in~\cite[Theorem 1.1]{ACNLM-JLMS}. The formul{\ae}~ \eqref{eq:orgininalF2points} are special cases of \eqref{fsuspACNLM} and/or 
Theorem~\ref{thmfsuspACNLM} for $\ex = 2$.
F.~Loeser suggested to the first named author, Cassou-Noguès, Luengo and Melle the following generalization for $F = z^\ex  + f(x_1, \ldots, x_{d})$, i.e. the suspension by $\ex$ points,
\begin{align}\label{fsuspACNLM-Loeser}
	\frac{\ex -1}{\ex } 
	\frac{s}{s+1} 
	\frac{t+1}{t}\cdot \Ztop{}(f,t)_0
	- \frac{s}{s+1} \sum_{1\neq e | \ex } \frac{\psi(e)}{\ex } \Ztop{(e)}(f,t)_0 + \frac{1}{\ex t},
\end{align}
taking $\psi(e)=(e + 1) J_1(e)$
where $J_1$ is the Euler totient function or the first Jordan totient function, see \S~\ref{ssec:arithmetic_functions}. Behind the  expression \eqref{fsuspACNLM-Loeser} is the motivic Thom-Sebastiani theorem~\cite{Denef-LoeserDuke},
see~\cite[\emph{Note added in proof}]{ACNLM-JLMS}. 
The main characteristic of \eqref{fsuspACNLM-Loeser} is the contributions of several twisted topological zeta functions of~$f$
combined with the values of an arithmetic function. While this idea is right,
the value of $\psi$ is not correct. The right value, as in~\eqref{fsuspACNLM},  is $\psi(e)=J_2(e)$ where $J_2$ is the second Jordan totient function, see \S~\ref{ssec:arithmetic_functions}. Note that both formul{\ae} \eqref{fsuspACNLM-Loeser} and \eqref{fsuspACNLM}  coincide when $\ex$ is prime.

Let us check that, in general,  \eqref{fsuspACNLM-Loeser} is inaccurate, with $\psi(e)=(e+1)J_1(e)$,  if $\ex$ is not prime. For instance, consider the polynomial $F(x,y,z) = x^\ex  + y^\ex  + z^\ex $ that is non degenerated with respect to its Newton polyhedra. According to~\cite[Theorem 5.3 or Example (5.4)]{DL-JAMS},  we have that
\begin{equation}\label{NonDegFormula}
\Ztop{}(F,s)_\zero= \frac{(\ex^2 -3 \ex + 3)s+3}{(s+1)(\ex s+3)}.
\end{equation}
On the other hand, $F(x,y,z)$ is the suspension by $\ex $ points of the plane curve singularity defined by $f(y,z)=y^\ex +z^\ex $. An embedded resolution of $\{f=0\}$ is obtained by a single point blow-up. One has a unique exceptional divisor $E_1$ with numerical data $(N_1, \nu_1)=(\ex ,2)$ that is intersected by the $\ex $ components of the strict transform of $\{f=0\}$. Hence, using the formulas in terms of an embedded resolution \eqref{eq:zetas_res_top} or again the formulas of the local topological zeta functions for polynomials non degenerated with respect to its Newton polyhedra, one gets that
\[
Z^{(\ell)}_{\mathrm{top}}(f,s)_\zero=
\begin{cases}
 \frac{2+(2-\ex )s}{(\ex s+2)(s+1)} & \text{ for } \ell =1,\\
\frac{2-\ex }{\ex s+2}& \text{ for } 1 \ne \ell | \ex ,\\
0 & \text{ otherwise.} 
\end{cases}
\]
Applying the above formula \eqref{fsuspACNLM-Loeser} for the suspension one gets the expression
\begin{equation}\label{CalcSusp}
\frac{3 \ex - (\ex^2 -4\ex + 2) s }{\ex(s + 1)(\ex s+3)} - \frac{s}{s+1}\cdot\frac{2-\ex }{\ex s+3} \sum_{e | \ex , e \ne 1} \frac{(e+1) J_1(e)}{\ex }.
\end{equation} 
The difference between  \eqref{NonDegFormula}  and \eqref{CalcSusp} is given by 
\begin{equation}\label{resta}
\frac{s}{s+1}\cdot\frac{\ex - 2}{\ex s+3} \left(\ex^2  -1 -  \sum_{1\neq e \mid \ex} (e+1) J_1(e) \right).
\end{equation} 
Note that if $\ex$ is not prime, since there are more
than one term in the sum, then
\[
\sum_{1\neq e \mid \ex} (e+1) J_1(e) < (\ex + 1)
\sum_{1\neq e \mid \ex} J_1(e) =
(\ex + 1) (\ex - 1) = \ex^2 - 1,
\]
using Gauss' identity \eqref{eq:Gauss}.

\subsection{Previous works on the holomorphy conjecture}
\label{sec:preholomorphy}
\mbox{}

With respect to the holomorphy conjecture
for suspensions we extensively use the work~\cite{dv:95}
of Denef and Veys. We make two  contributions here.
On one side, we make more clear the concept of
\emph{bad eigenvalue}, see Definition~\ref{def:bad_div}, adding the condition~\ref{b4}.
On the other side, we finished the proof of the holomorphy conjecture
for suspensions by $2$ points of curves, see Theorem~\ref{thm:holconfsuspcurve}. In particular, we check the conjecture in the case that the monodromy of the curve has some bad eigenvalue. In order to achieve this, 
the formul{\ae} for the (twisted) topological zeta functions
for suspensions of curves are essential, see Theorem~\ref{thmfsuspACNLM} and Lemma~\ref{lema:2l}.

We have proved also the holomorphy conjecture for $\ex$-LYS
of surfaces. The strategy of~\cite{dv:95} and the formulas
for the twisted topological zeta functions are the key points.
Note that all the proofs need to take care of some
arithmetical subtleties.

\subsection{Previous works on SIS}
\label{sec:preSIS}
\mbox{}

The formula   for the topological zeta functions of SIS, see Corollary~\ref{cor:zetaSIS}\ref{sis-1}, 
was originally obtained in~\cite[Corollary 1.12]{ACNLM-ASENS}. In this paper,
we also obtain formulas from Corollary~\ref{cor:zetaSIS}\ref{sis-2},~\ref{sis-3}, and~\ref{sis-4}  for the twisted  topological zeta functions of SIS
in terms of the combinatorial data of the projectivized tangent cone
and of the twisted topological zeta functions of its singular points.
All these formul{\ae} are extended to $\ex$-LYS in Theorem~\ref{thm:DLZFkLY}. The presence
of~$\ex$ introduces  some arithmetic complexity in the results
and in the proofs.

The monodromy conjecture was proved for SIS in~\cite{ACNLM-ASENS}. The structure of the proof for $\ex$-LYS in \S\ref{Sec:Proof_Mon_Conj} follows
the same guidelines, once the formul{\ae} for the topological 
zeta functions in Theorem~\ref{thm:DLZFkLY} are obtained. As we have already pointed out, 
there is a case that does not appear in the proof
of the monodromy conjecture for SIS and it appears for
general $\ex$-LYS, namely case~\ref{pole4} in the proof of Theorem{\rm~\ref{thm:CMk_LYS}}. This case refers to the following situation. If $-\rho_0$ is pole for a singular point
of the projectivized tangent cone~$C_\exz$ for the corresponding candidate pole 
$\cpolo(\rho_0):=\cpolo(\rho_0,3,\exz,\ex)$ 
it may happen that $\exp(-2\pi i \cpolo(\rho_0))$ is not an eigenvalue
of the monodromy of the $\ex$-LYS. This may happen if the following to conditions hold: $\exz\cpolo(\rho_0)\in\ZZ$
and $\chi(\PP^2\setminus C_\exz)<0$. In~\cite{ACNLM-ASENS}
the authors get rid of this case showing that the only option
is $\rho_0=\frac{3}{\exz}$ which has to be studied independently (as for $\ex$-LYS).
For $\ex$-LYS other situations may happen, but the combination
of Lemma~\ref{lema:cyclo2}, the structure of Kashiwara's pencils, and Veys's proof the monodromy conjecture for curves allow us to prove that $\exp(-2\pi i \cpolo(\rho_0))$ is an eigenvalue
in this case.

For the other situations we adapt the corresponding proof
for SIS but the presence of $\ex$ makes  the generalization
 far from being straightforward. To overcome the  issues coming from the arithmetic, we change slightly the 
strategy of the proof. For example,
Proposition~\ref{rem:SIS_typo} is based on~\cite[Theorem 2.15]{ACNLM-ASENS}, but the statement contains more results which
are key for the proof of the monodromy conjecture. Namely,
with this result we can rewrite~\cite[Theorem 3.2]{ACNLM-ASENS}
as follows.  As in our proof, let $r$ be the number of irreducible components of $C_\exz$. The first result in
\cite[Theorem 3.2]{ACNLM-ASENS} is:
\begin{quote}\em
If $r\geq 3$ then $\exp\left(-2\pi i \frac{3}{\exz} \right)$ is a root of the Alexander polynomial of the germ of $C_\exz$ at~$q$
with multiplicity is at least $r - 1$.
\end{quote}
This result is not true because of 
Lemma~\ref{lema:deltaEphi}\ref{lema:deltaEphi1}, see the case $I_b^L$ of Example~\ref{ex:quartic}
for an explicit counterexample.
The second result of~\cite[Theorem 3.2]{ACNLM-ASENS} is:
\begin{quote}\em
	Otherwise, if $\res(C_\exz)$ = 0, then $\exp\left(-2\pi i \frac{3}{\exz} \right)$ is a root of the Alexander polynomial of
	the germ of $C_\exz$ at its singular point.
\end{quote}
The result actually proved in~\cite[Theorem 3.2]{ACNLM-ASENS} is:
\begin{quote}\em
	If $\res(C_\exz)$ = 0, then $exp\left(-2\pi i \frac{3}{\exz} \right)$ is a root of the Alexander polynomial of
	the germ of $C_\exz$ at its singular point of multiplicity at least $r-2$.
\end{quote}
This last statement is the one that is needed in~\cite{ACNLM-ASENS} to prove the monodromy conjecture for SIS.

\section*{Acknowledgments}

The authors would like to thank Pierrette Cassou-Nogu\`es, Ignacio Luengo Velasco and Alejandro Melle Hernández for many discussions about superisolated and Lê-Yomdin singularities, and the monodromy conjecture. They also thank Sabir Gusein-Zade for the suggestion to use the stratification principle and  binomial polynomials to compute local Denef-Loeser zeta functions of Lê-Yomdin singularities.
The authors thank Carlos de Vera and Glenier Bello for their help 
in the proof of Proposition~\ref{prop:Nbullet} and several arithmetic statements. Finally, they thank Miriam Bocardo Gaspar for discussions at an early stage of this project.
 
 \providecommand\noopsort[1]{}
\providecommand{\bysame}{\leavevmode\hbox to3em{\hrulefill}\thinspace}
\providecommand{\MR}{\relax\ifhmode\unskip\space\fi MR }
\providecommand{\MRhref}[2]{%
  \href{http://www.ams.org/mathscinet-getitem?mr=#1}{#2}
}

\end{document}